\documentclass{amsart}
\usepackage[T1]{fontenc}
\usepackage{barz}
\author{Michael Barz}
\DeclareMathOperator{\can}{can}
\DeclareMathOperator{\AS}{AS}
\DeclareMathOperator{\DCZ}{dCZ}
\DeclareMathOperator{\Conn}{Conn}
\newcommand{\sW}{{}^{s}W}

\usepackage{epigraph}

\newcommand{\SPEC}{\mathbf{Spec}} 

\DeclareMathOperator{\cone}{cone}
\DeclareMathOperator{\res}{res}
\DeclareMathOperator{\Dol}{Dol}
\renewcommand{\hat}{\widehat}
\title{Logarithmic de Rham stacks and non-abelian Hodge theory}
\begin{document}

\begin{abstract}
		In this article, we introduce the \emph{logarithmic} de Rham stack of a pair \((X, D),\) for \(X\) a smooth variety over a field \(k\) of positive characteristic \(p,\) and \(D\) a strict normal crossings divisor on \(X.\) Using this stack, we prove a new version of logarithmic Cartier descent, and a new logarithmic non-abelian Hodge theorem for curves, both stated using a certain logarithmic Frobenius twist. Our logarithmic non-abelian Hodge theorem implies an earlier logarithmic non-abelian Hodge theorem of de Cataldo-Zhang \cite{dCZ}.
\end{abstract}

\maketitle
\tableofcontents

\epigraph{`It's about \emph{a certain machine},' I said, getting right to the point.}{\emph{Pinball, 1973}, Haruki Murakami}

\section{Introduction}

Fix a prime \(p.\) For \(k\) a field of characteristic \(p,\) and \(X\) a smooth variety over \(k\) with a strict normal crossings divisor \(D,\) we will introduce a stack \(X^{\dR}(\log D),\) which we dub the \emph{logarithmic de Rham stack} of \((X, D).\) The definition is straightforward; setting \(\Theta = \A^1/\G_m,\) if \(D\) has \(n\) components, then we can encode the divisor \(D\) as a certain smooth morphism \(c_D : X \to \Theta^n\) (this will be explained in \autoref{logdrsec}).
\begin{defn}
		The \emph{logarithmic de Rham stack} \(X^{\dR}(\log D)\) of \((X, D)\) is the relative de Rham stack \((X/\Theta^n)^{\dR}\) of \(X\) over \(\Theta^n.\) The notion of a de Rham stack relative to a general base will be made precise in \autoref{derhamsection}.
\end{defn}

de Rham stacks in characteristic 0 were first defined by Simpson \cite{simpson}, as a way to `geometrize' algebraic de Rham cohomology. That is, for \(Y\) a complex algebraic variety, Simpson defined a certain stack \(f : Y^{\operatorname{Simpson-dR}} \to \Spec(\C)\) having the following two properties:
\begin{enumerate}
		\item there is an equivalence of categories \(\Qcoh(Y^{\operatorname{Simpson-dR}}) \heq \Conn(Y),\) for \(\Conn(Y)\) the category of quasicoherent sheaves on \(Y\) endowed with a flat connection, 
		\item for each \(q,\) there are isomorphisms \(R^qf_*\OO_{Y^{\operatorname{Simpson-dR}}} \heq H^q_{\dR}(Y, \C)\) between the coherent cohomology of \(Y^{\operatorname{Simpson-dR}}\) and the algebraic de Rham cohomology of \(Y.\)
\end{enumerate}
The utility of Simpson's construction is that now algebraic de Rham cohomology behaves as a coherent cohomology theory. Working over \(Y^{\operatorname{Simpson-dR}}\) also allows one to better understand phenomena in non-abelian Hodge theory.

Since then, Drinfeld \cite{prismatization} and Bhatt--Lurie \cite{bhattlurie} have made mixed and positive characteristic analogues of Simpson's de Rham stack. Recently, Bhatt--Kanaev--Mathew--Vologodsky--Zhang \cite{bkmvz} have introduced a `sheared' variant of this construction of Drinfeld and Bhatt--Lurie, which has slightly better properties. Our logarithmic de Rham stack is essentially a logarithmic variant of these constructions, which allows one to geometrize \emph{logarithmic} de Rham cohomology, in the following sense.
\begin{remark}
		Analogously to Simpson's de Rham stack, we will prove in \autoref{logdrsec} that
		\[\Qcoh(X^{\dR}(\log D)) \heq \Conn^{\psi_p\text{ nilp}}(X, D),\]
where by \(\Conn(X, D)\) we mean the category of pairs \((\mathcal{E}, \nabla : \mathcal{E} \to \mathcal{E} \otimes \Omega^1_X(\log D)),\) where \(\mathcal{E} \in \Qcoh(X)\) and \(\nabla\) is a flat connection whose \(p\)-curvature is nilpotent. It is inconvenient for some applications to need to restrict to connections with \(p\)-curvature zero; to eliminate this problem, we will also define a \emph{sheared} variant of \(X^{\dR}(\log D),\) which we define by \(X^{\hat{\dR}}(\log D).\) The theory of sheared de Rham stacks will be reviewed in \autoref{derhamsection}, but the key difference is that
		\[\Qcoh(X^{\hat{\dR}}(\log D)) \heq \Conn(X, D),\]
		where now we have the category of \emph{all} flat connections \(\nabla : \mathcal{E} \to \mathcal{E} \otimes \Omega^1_X(\log D),\) with no restriction on \(p\)-curvature.
\end{remark}
\begin{remark}
		The idea of making logarithmic constructions by working over \(\A^1/\G_m\) is due to Olsson; see Olsson's \cite{olsson3} and \cite{olsson} for some examples especially relevant to our applications.
\end{remark}

Many connections which arise in practice have logarithmic poles; thus, for any theorem about the category \(\Conn(X)\) (meaning flat connections on \(X\) with \emph{no} poles), it is often essential to have a corresponding theorem about \(\Conn(X, D).\) Our logarithmic de Rham stack will allow one to easily transmute the proof of a theorem about \(\Conn(X)\) to a proof of a theorem about \(\Conn(X, D).\) The point is that, since the logarithmic de Rham stack is defined in almost exactly the same way as the usual de Rham stack \((X/k)^{\dR},\) most arguments involving the de Rham stack will not care about if the base is \(\Spec(k)\) or if it is \((\A^1/\G_m)^n.\) 

To persuade the reader of this point, we show how two interesting theorems -- Cartier descent, and Groechenig's \cite{groechenig} non-abelian Hodge theorem for curves -- can be proven in the logarithmic setting in nearly exactly the same way as they are proven in the non-logarithmic setting. Besides just suggesting proofs, in both cases the logarithmic de Rham stack also \emph{suggests what the statement} should be. 

We now overview these two applications. 

\subsection{Logarithmic Cartier descent}

\begin{defn}
		We write \(X'\) for the Frobenius twist of \(X\) relative to \(k.\)
\end{defn}

In characteristic 0, one can reconstruct a linear ODE from its local system of solutions. Cartier found an analogue to this in positive characteristic; the starting point is that, instead of a local system, solutions to a linear ODE on in characteristic \(p\) naturally live over the Frobenius twist (this is more or less the observation that \(x^p\) has derivative 0 in characteristic \(p\)), allowing us to define a functor
\[\Sol : \Conn(X) \to \Qcoh(X'),\]
sending a vector bundle with flat connection \((\mathcal{V}, \nabla)\) on \(X\) to its quasicoherent sheaf of solutions on the Frobenius twist \(X'.\) 

\begin{remark} \label{rem75}
		Geometrically, this fact that solutions live on \(X'\) is related to the fact that \((X/k)^{\dR}\) admits a natural morphsim to \(X'.\) In \autoref{pcurvsec}, we will similarly construct a certain natural morphism
		\[X^{\dR}(\log D) \to (X/\Theta^n)^{(1)},\]
		where by \((X/\Theta^n)^{(1)}\) we mean the Frobenius twist of \(X\) relative to \(\Theta^n.\) 
\end{remark}
\begin{remark}
		This relative Frobenius twist \((X/\Theta^n)^{(1)}\) appearing in \autoref{rem75} is a certain root stack. More precisely, \((X/\Theta^n)^{(1)}\) is the \((p, p, ..., p)\)-multiroot stack of \(X'\) along each component of \(D'.\) In \autoref{logdrsec}, we will review multiroot stacks.
\end{remark}

The analogue of the Riemann-Hilbert correspondence, reconstructing an ODE from its local system of solutions, then becomes the following.
\begin{theorem}[Cartier descent, Theorem 5.1 of \cite{katz}] \label{thm71} 
		The functor
		\[\Sol : \Conn(X) \to \Qcoh(X')\]
		restricts to an equivalence of categories
		\[\Conn(X)^{\psi_p = 0} \inclusion \Conn(X) \xto{\Sol} \Qcoh(X')\]
		gives an equivalence of categories between \(\Qcoh(X')\) and the category of flat connections on \(X\) with \(p\)-curvature 0.
\end{theorem}

As suggested by \autoref{rem75}, using the logarithmic de Rham stack suggests that, in the logarithmic setting, solutions to an ODE naturally live on this stacky Frobenius twist \((X/\Theta^n)^{(1)}.\) We now give an example illustrating this in a concrete situation.
\begin{exmp} \label{exmp91} 
		Let \(X = \A^1 = \Spec k[x].\) For each \(a \in \{0, 1, ..., p-1\},\) we can define a connection 
		\[\nabla_a : \OO_X \to \Omega_X^1, \nabla_a = d - \frac{a}{x} dx.\] 

		In this case, \(X' = \Spec k[x^p],\) and the solutions to \(\nabla_a\) are the \(k[x^p]\)-module \(\Sol_a := x^a \cdot k[x^p].\) In particular, as \(k[x^p]\)-modules, the \(\Sol_a\) are all isomorphic. 

		This prevents a naive version of Cartier descent (\autoref{thm71}) from holding in the logarithmic setting, as the connections \(\nabla_a\) are all distinct -- and all have \(p\)-curvature zero -- yet have isomorphic sheaves of solutions. 

		However, if we view the solutions as sheaves on the Frobenius twist \((X/\Theta)^{(1)},\) then the situation improves. A quasicoherent sheaf on \((X/\Theta)^{(1)}\) consists of a \(k[x^p]\)-module, together with a \(p\)-step filtration obeying certain properties (see ). In this situation, the \(p\)-step filtration on \(\Sol_a\) is given by
		\[F^p = x^{a+p} \cdot k[x^p]\subseteq \cdots F^{a+1} = x^{a+p} \cdot k[x^p] \subseteq F^a = x^a \cdot k[x^p] \subseteq \cdots \subseteq F^0 = x^a \cdot k[x^p].\]

		In other words, \(F^i\) consists of all solutions to \(\nabla_a\) which are divisible by \(x^i.\) Note that these filtrations allow us to distinguish the modules \(\Sol_a\); and so, while a naive version of Cartier descent fails in the logarithmic situation, we might hope that if we replace \(X'\) by \((X/\Theta^n)^{(1)}\) in \autoref{thm71}, we can salvage Cartier descent.
\end{exmp}

The phenomenon of \autoref{exmp91} generalizes, and we will prove in \autoref{pcurvsec} that one always has the following form of logarithmic Cartier descent.
\begin{theorem}[Logarithmic Cartier descent] \label{thm105} 
		There is a morphism \(\nu : X^{\dR}(\log D) \to (X/\Theta^n)^{(1)}\) (which we will explicitly construct in \autoref{pcurvsec}) such that
		\[\nu^* : \Qcoh((X/\Theta^n)^{(1)}) \to \Qcoh(X^{\dR}(\log D))\]
		is fully faithful, and has essential image those logarithmic connections of \(p\)-curvature zero. 
\end{theorem}
\begin{remark}
		Implicit in the statement of \autoref{thm105} is that there is a notion of \(p\)-curvature of logarithmic connections; we will define a notion of \(p\)-curvature for any relative de Rham stack in \autoref{pcurvsec}.
\end{remark}
\begin{remark}
		In practice, one can formally manipulate quasicoherent sheaves on \((X/\Theta^n)^{(1)}\) without needing to think too hard about the implicit filtrations (called \emph{parabolic structures}). We hope this is apparent in our proof of logarithmic Cartier descent, which is completely formal, and never needs to explicitly track parabolic structures. Our proof of logarithmic non-abelian Hodge theory also never needs to explicitly use parabolic structures, except for when we deduce de Cataldo-Zhang's \cite{dCZ} non-abelian Hodge theorem from ours. 
\end{remark}
\begin{remark}
		The stack \((X/\Theta^n)^{(1)}\) appears also in the work of Esnault-Groechenig \cite{eg} (see their Lemma 3.16). Esnault-Groechenig's \cite{eg} was a large inspiration for this project, as the author was trying to understand the constructions of that paper in terms of de Rham stacks.
\end{remark}
\begin{remark}
		While the stack \((X/\Theta^n)^{(1)}\) does not appear explicitly in Hablicsek's \cite{azumayaparabolic}, Hablicsek does essentially work with quasicoherent sheaves on this stack (in the guise of parabolic bundles -- see Borne \cite{borne2007} for the relationship). Encountering Hablicsek's work gave the author confidence that de Rham stacks relative to \(\Theta^n\) should behave analogously to de Rham stacks over \(\Spec(k).\)
\end{remark}
\begin{remark}
		Lorenzon \cite{lorenzon} has an alternative version of logarithmic Cartier descent, phrased in the language of \emph{indexed} Azumaya algebras. At the present moment, the author is unaware of a direct relationship between Lorenzon's indexed Azumaya algebras and the parabolic structures which show up in our logarithmic Cartier descent. 
\end{remark}

\subsection{Logarithmic non-abelian Hodge theory}

In characteristic 0, Simpson and others established the \emph{non-abelian Hodge theorem}, which relates a moduli space of \emph{vector bundles with connection} on a Riemann surface to the moduli space of \emph{Higgs fields} on that same surface. Inspired by this characteristic 0 story, Groechenig \cite{groechenig} found a wonderful characteristic \(p\) variant. 

In the characteristic 0 story, the curves are defined over \(\C,\) yet the comparison of moduli spaces is only \emph{real analytic}. Similarly, Groechenig's \cite{groechenig} non-abelian Hodge theorem involves a Frobenius twist. Using the Frobenius twist \((X/\Theta^n)^{(1)}\) suggested by the logarithmic de Rham stack also allows us to prove a logarithmic non-abelian Hodge theorem for curves using almost verbatim the same argument that Groechenig \cite{groechenig} used in the non-logarithmic setting. More precisely, when \(\dim X = 1,\) we have the following.

Let \(\mathcal{M}_{\Dol} := \mathcal{M}_{\Dol}((X/\Theta^n)^{(1)}, r)\) denote the moduli stack of rank \(r\) Higgs fields on \((X/\Theta^n)^{(1)}.\) This moduli stack admits a natural Hitchin morphism
\[h : \mathcal{M}_{\Dol}((X/\Theta^n)^{(1)}, r) \to \mathcal{A}^{(1)} := \bigoplus_{i=1}^r H^0((X/\Theta^n)^{(1)}, \omega_{(X/\Theta^n)^{(1)}}^{\otimes i}).\] 

Similarly, let \(\mathcal{M}_{\dR} := \mathcal{M}_{\dR}(X, D, r)\) denote the moduli stack of pairs \((\mathcal{E}, \nabla),\) where \(\mathcal{E}\) is a rank \(r\) vector bundle on \(X\) and \(\nabla : \mathcal{E} \to \mathcal{E} \otimes \Omega_X^1(\log D)\) is a connection. In characteristic \(p,\) there is a natural \emph{linear} invariant of a connection, called the \emph{\(p\)-curvature}. The \(p\)-curvature will be a certain \emph{linear} map \(\psi_p : \mathcal{E} \to \mathcal{E} \otimes F^*_X\Omega_X^1(\log D).\) Though naturally the characteristic polynomial of such a map has coefficients in \(\mathcal{A} := \bigoplus_{i=1}^r H^0(X, F^*_X(\Omega_X^1(\log D)^{\otimes i})),\) Laszlo-Pauly \cite{laszlopauly} (though see also Zhang \cite{siqingfix}) proved the wonderful result that the \(p\)-curvature of a connection always has a charateristic polynomial whose coefficients lie in a subspace of \(\mathcal{A}\) isomorphic to \(\mathcal{A}^{(1)}.\) In this way, Laszlo-Pauly \cite{laszlopauly} allow us to define a de Rham Hitchin morphism 
\[h_{\dR} : \mathcal{M}_{\dR}(X, D, r) \to \mathcal{A}^{(1)}.\]

\begin{remark}
		We give a definition of \(p\)-curvature in terms of de Rham stacks, which allows us to give a new construction of \(h_{\dR},\) by viewing \(\psi_p\) as a morphism of quasicoherent sheaves on \(X^{\hat{\dR}}(\log D).\) See \autoref{charpolypsip} for more information.
\end{remark}

\begin{theorem}[Our \autoref{thm1447}] \label{thm1}
		There is a group scheme \(\mathcal{P}\) over \(\mathcal{A}^{(1)}\) acting naturally on \(\mathcal{M}_{\Dol},\) and a \(\mathcal{P}\)-torsor \(\mathcal{S}\) such that
		\[\mathcal{S} \times^{\mathcal{P}} \mathcal{M}_{\Dol} \heq \mathcal{M}_{\dR}\]
		as \(\mathcal{A}^{(1)}\)-stacks. 
\end{theorem}
\begin{remark}
		In particular, the fibers of \(h : \mathcal{M}_{\Dol} \to \mathcal{A}^{(1)}\) and \(h_{\dR} : \mathcal{M}_{\dR} \to \mathcal{A}^{(1)}\) over any point of \(\mathcal{A}^{(1)}\) are (non-canonically!) isomorphic. 
\end{remark}

de Cataldo-Zhang \cite{dCZ} have a different logarithmic non-abelian Hodge theorem; we are able to deduce theirs from ours. The key differences between our theorems is that
\begin{enumerate}
		\item de Cataldo-Zhang do not use this stacky Frobenius twist \((X/\Theta^n)^{(1)}\) in defining their Dolbeault side, and instead only use \(X',\)
		\item de Cataldo-Zhang have a slightly different Hitchin base, remembering not just the spectral data of \(p\)-curvature, but also the spectral data of the \emph{residue} of a connection.
\end{enumerate}

One may understand our difference in approaches as follows. Naively, one might hope to compare \(\mathcal{M}_{\dR}(X, D, r)\) and \(\mathcal{M}_{\Dol}(X', D', r).\) Unfortunately, the fiber of \(\mathcal{M}_{\dR}(X, D, r)\) over a point \(a \in \mathcal{A}^{(1)}\) is in general much larger than the fiber of \(\mathcal{M}_{\Dol}(X', D', r)\) over \(a.\) To fix this, our \autoref{thm1} makes the Dolbeault side bigger by replacing \(X'\) with the stacky Frobenius twist \((X/\Theta^n)^{(1)}.\) However, one could instead fix this by making the fibers on the de Rham side smaller; de Cataldo-Zhang \cite{dCZ} accomplish this by enlarging the Hitchin base.

To deduce de Cataldo-Zhang's non-abelian Hodge theorem from ours, we track carefully how the residue of a connection is related to parabolic structures (as a parabolic structure is the difference between a sheaf on \(X'\) and a sheaf on \((X/\Theta^n)^{(1)}\)). 

\subsubsection{Relation to work of Li--Sun} 

Li--Sun \cite{lisun} explore a version of logarithmic non-abelian Hodge theory for curves in positive characteristic using parahoric structures, more closely analogous to the characteristic 0 story. The author is presently unaware of any direct relationship between our \autoref{thm1} and their results. While quasicoherent sheaves on root stacks (such as \((X/\Theta^n)^{(1)}\)) are often called sheaves with parabolic structure, it seems the parahoric structures used are of a very different nature than the parabolic structures considered here. Indeed, their parahoric structures depend heavily on the group used, whereas our parabolic structures are always just \(p\)-step filtrations.

\subsubsection{Relation to work of Ogus--Vologodsky and Schepler}

Ogus--Vologodsky \cite{ov} prove a different sort of logarthmic non-abelian Hodge theorem; Schepler \cite{schepler} found a logarithmic variant of their work. The key differences between our work and Schepler is that 
\begin{enumerate}
		\item we give a logarithmic version of the non-abelian Hodge theorem of Groechenig \cite{groechenig}, which gives a correspondence between \emph{all} Higgs fields and \emph{all} connections, but only in dimension 1;
		\item Schepler \cite{schepler} gives a logarithmic version of Ogus--Vologodsky's \cite{ov} non-abelian Hodge theorem, which gives a correspondence in \emph{all} dimensions, but only for Higgs fields and connections obeying a certain nilpotence condition.
\end{enumerate}

Gleb Terentiuk is currently writing an article on Ogus--Vologodsky's work, but re-done in the language of de Rham stacks; it seems very likely to the author of this work that Terentiuk's results will also give a strengthening Schepler's \cite{schepler} result. 

\subsection{Weight filtrations}

The logarithmic de Rham stack \(X^{\hat{\dR}}(\log D)\) naturally lives over \((\A^1/\G_m)^n,\) since it is defined as the relative de Rham stack of a morphism over \((\A^1/\G_m)^n.\) Thus, the cohomology of any quasicoherent sheaf on \(X^{\hat{\dR}}(\log D)\) is naturally equipped with a \(\Z^n\)-indexed filtration. In a forthcoming work, the author will focus more on these aspects of the logarithmic de Rham stack (which are interesting even in characteristic 0). 

For now, we pose the following questions. Recently, Annala-Pstragowski \cite{weightfiltrations} constructed filtrations on logarithmic de Rham cohomology in positive characteristic; they then proved that these filtrations actually depend only on the open part \(U := X\setminus D,\) and not on the specific choice of compactification. 
\begin{question}
		Can one geometrize the weight filtration of Annala-Pstragowski?
\end{question}
\begin{remark}
		The logarithmic de Rham stack \(X^{\dR}(\log D)\) does \emph{not} depend on just \(X\setminus D,\) so in searching for a geometrization of this weight filtration of Annala-Pstragowski, it would be nice to find some modification of the logarithmic de Rham stack which \emph{does} only depend on \(X\setminus D.\)
\end{remark}

In this vein, it is also interesting to ask the following question.
\begin{question} \label{quest149} 
		Let \(U\) be a smooth variety over \(k.\) Is there some stack \(U^{\dR, \reg}\), depending only on \(U,\) such that \(\Qcoh(U^{\dR, \reg}) \heq \Conn(U)^{\reg},\) the category of connections on \(U\) with regular singularities at infinity?
\end{question}
\begin{remark}
		As a smooth curve has a unique compactification, logarithmic de Rham stacks provide a positive answer to \autoref{quest149} for \(\dim U = 1.\)
\end{remark}

\subsection{Organization of this paper}

In \autoref{derhamsection}, we will define \(X^{\dR}(\log D)\) and \(X^{\hat{\dR}}(\log D)\) more formally, by reviewing the theory of (sheared) de Rham stacks and then defining relative de Rham stacks. Afterwards, we will study \(p\)-curvature and Cartier descent of relative de Rham stacks more closely in \autoref{pcurvsec}. Then, we turn towards non-abelian Hodge theory. \autoref{hitchinsec} sets up the theory of Hitchin morphisms in some genearlity, to prepare for our discussion of non-abelian Hodge theory. In \autoref{logdrsec}, we explain what the results of \autoref{pcurvsec} and \autoref{derhamsection} specialize to in the case of logarithmic de Rham stacks. After this prepatory work, we finally state and prove our logarithmic non-abelian Hodge theorem in \autoref{lognonsec}. 

In the appendix, we review some work of Niels Borne on quasicoherent sheavse on root stacks, which will be used in studying \(\Qcoh((X/\Theta^n)^{(1)}).\) 
\subsection{Acknowledgements}

The author thanks Bhargav Bhatt for many very useful discussions on this work, and for suggesting the author use these techniques to re-prove de Cataldo--Zhang's \cite{dCZ} logarithmic non-abelian Hodge theorem. 

The author also thanks Siqing Zhang for giving very detailed feedback on a preliminary draft of this paper, and for catching some typos in the first draft.

In addition, the author is grateful to Toni Annala for explaining aspects of \cite{weightfiltrations}; to Niven Achenjang for explaining aspects of his work \cite{achenjang}; to Zachary Berens, Kenta Suzuki and Gleb Terentiuk for discussions about relative de Rham stacks over \(B\G_m\); to Niels Borne for communications about \cite{borne2007}; to Sanath Devalapurkar for communications about \(p\)-curvature; to Longke Tang and Qixiang Wang for simplifying the construction of \autoref{pcurvmorphism}; to Siqing Zhang for giving detailed feedback on a draft of this article; and to Mingjia Zhang for discussions on sheared de Rham stacks.

\section{The relative de Rham stack and its sheared variant} \label{derhamsection} 

In this \emph{purely expository} section, we recall a few results on de Rham stacks and sheared de Rham stacks. This is done mostly because the results we need about sheared de Rham stacks are not yet in the literature, but nothing in this section should be viewed as original; all arguments we make in this section essentially appear in Bhatt's notes \cite{bhatt} (which themselves are based on the related works \cite{drinfeld} of Drinfeld and \cite{bhattlurie} of Bhatt--Lurie). We hope this section can be useful to a reader less familiar with the theory of de Rham stacks.

\begin{warn}
		To reiterate, no results in this section are original; all come from Bhatt--Mathew--Vologodsky--Zhang's forthcoming work \cite{bmvz}, Drinfeld's \cite{shearedwitt} and \cite{drinfeld}, and Bhatt--Lurie's \cite{bhattlurie}. We recall these results in part because \cite{bmvz} has not yet appeared publicly and in part to make the article a little more self-contained.
\end{warn}	

\subsection{The ring stacks \(\G_a^{\dR}\) and \(\G_a^{\hat{\dR}}\)}

\begin{remark}[Motivational remark] \label{rem93} 
		In characteristic 0, Simpson \cite{simpson} defines the de Rham stack of a scheme \(X \in \Sch_{\C}\) by 
		\begin{equation} \label{eq95} X^{\operatorname{Simpson-dR}}(A) := X(A_{\red}).\end{equation}
		Simpson makes this definition for the following reason: \(X^{\operatorname{Simpson-dR}}\) is a stack whose coherent cohomology is the algebraic de Rham cohomology of \(X,\) and whose quasicoherent sheaves are precisely quasicoherent \(\mathcal{D}_{X/\C}\)-modules. This allows one to reason about de Rham cohomology and \(\mathcal{D}\)-modules using the machinery of algebraic geometry. 
		In characteristic \(p,\) we will use a very similar formula, but instead of replacing a ring \(A\) by its quotienet \(A_{\red} = A/\nil(A)\), we will quotient by elements admitting a system of divided powers; this is related to the difference between the infinitesimal and crystalline sites. One can think of Simpson's de Rham stack as geometrizing the infinitesimal site, the characteristic \(p\) de Rham stack as geometrizing the crystalline site, and the sheared de Rham stack as geometrizing the \emph{nilpotent} crystalline site. 

		Unfortunately, while being nilpotent is a \emph{property} of an element, a divided power structure is \emph{data}. This means that, even to handle de Rham stacks of schemes, one needs to use (1-truncated) animated rings (also called \emph{ring groupoids}).
\end{remark}

Analogously to Simpson's formula \autoref{eq95}, the de Rham and sheared de Rham stacks are defined by \emph{transmutation} of certain ring stacks. That is, for \(X\) an \(\F_p\)-stack, we will see that
\[X^{\dR}(A) = X(\G_a^{\dR}(A)), X^{\hat{\dR}}(A) = X(\G_a^{\hat{\dR}}(A)),\]
for certain ring stacks \(\G_a^{\dR}\) and \(\G_a^{\hat{\dR}}.\) The purpose of this section is to define these ring stacks. 

\subsubsection{Reminder on quasi-ideals} \label{quasiideals} Drinfeld developed a very convenient way to model 1-truncated ring stacks, called \emph{quasi-ideals}; see section 3.1 of \cite{prismatization}, or the short article \cite{drinfeld}, for the full story, but for the reader's convenience we will recall a few of Drinfeld's results.

\begin{defn}
		A \emph{quasi-ideal} in a ring \(R\) is a pair \((I, d)\) where \(I\) is an \(R\)-module and \(d : I \to R\) is an \(R\)-linear map obeying \(d(x) \cdot y = d(y) \cdot x\) for all \(x, y\in I.\) 

		A \emph{sheaf of quasi-ideals} on a sheaf of rings \(\mathcal{R}\) (for us, \(\mathcal{R}\) will be a sheaf of rings on \(\Sch_{\F_p}\) with the fpqc topology, but one can make this definition on any Grothendieck site) is a sheaf of \(\mathcal{R}\)-modules \(\II\) together with an \(\mathcal{R}\)-linear map \(d : \II \to \mathcal{R}\) such that, for every test object \(T,\) the map \(d(T) : \II(T) \to \mathcal{R}(T)\) is a quasi-ideal on the ring \(R(T).\) 
\end{defn}

For any quasi-ideal \(d : I \to R,\) Drinfeld constructs a 1-truncated animated \(R\)-algebra which he calls
\[\Cone(I \to R).\]
See section 3 of \cite{drinfeld} for the construction; one should think of it as a suitably derived version of \(R/I\). 

\begin{remark} \label{rem54}
		The underlying 1-truncated animated \(R\)-module of \(\Cone(I \to R)\) has a very succint definition: 
		\[\Cone(I \to R) := \tau_{\leq 1}\cofib(I \to R),\]
		where we take the cofiber in the category of \(R\)-modules; see \cite{drinfeld}'s remark just before section 3.4.8.

		In an \(\infty\)-categorical context, the cofiber \(\cofib(I \to R)\) is a little more natural, and the truncation can make the cone seem a little confusing to work with. We will explain that, for our examples of interest, the truncation is not truly needed.
\end{remark}

Drinfeld's construction is functorial in the quasi-ideal, and hence if \(d : \II \to \mathcal{R}\) is a sheaf of quasi-ideals on \(\Sch_{\F_p}\) (with the flat topoplogy), then one can define a functor 
\[T \mapsto \Cone(\II(T) \to \mathcal{R}(T))\] 
from \(\Sch_{\F_p}\) to the category of animated rings. The fpqc \emph{sheafification} of this functor is denoted by \(\Cone(\II \to \mathcal{R}),\) and is an animated ring stack on \(\Sch_{\F_p}.\) As \(\Cone(\II(T) \to \mathcal{R}(T))\) is an \(\mathcal{R}(T)\)-algebra, we can also treat \(\Cone(\II \to \mathcal{R})\) as a sheaf of \(\mathcal{R}\)-algebras.

\begin{remark}[Sheafified cones are still 1-truncated]
		Sheafification is a left exact functor, and so by proposition 5.5.6.16 of Lurie's \cite{HTT}, it takes 1-truncated objects to 1-truncated objects; thus \(\cone(\II \to \mathcal{R})\) will still be a sheaf of 1-truncated animated rings.
\end{remark}

One might be worried about set-theoretic subtleties implicit in fpqc sheafification; fortunately, we can give an explicit construction of the sheafification as a quotient stack. 

We will start by explaining how to define \(\Cone(\II \to \mathcal{R})\) as an fpqc stack of \(\mathcal{R}\)-modules.
\begin{lemma} \label{lem73} 
		Let \([\mathcal{R}/\II]\) denote the quotient stack; explicitly, we can define \([\mathcal{R}/\II]\) as the stack of groupoids defined by setting \([\mathcal{R}/\II](T)\) to be the groupoid of pairs \((\FF, \phi : \FF \to \mathcal{R}|_T)\), where \(\FF\) is an fpqc \(\II|_T\)-torsor on \(T\) and \(\phi\) is \(\II|_T\)-equivariant. In particular, \([\mathcal{R}/\II]\) is 1-truncated.

		Then \([\mathcal{R}/\II]\) is an fpqc sheafification of the presheaf of anima defined by
		\[T \mapsto \Cone(\II(T) \to \mathcal{R}(T)).\] 
\end{lemma}
\begin{proof}
		By \autoref{rem54}, the object \(\Cone(\II(T) \to \mathcal{R}(T)),\) considered as an anima, is nothing but the 1-truncation of \(\cofib(\II(T) \to \mathcal{R}(T)).\) The sheafification of \(T \mapsto \cofib(\II(T) \to \mathcal{R}(T))\) is by definition the sheaf \(\cofib(\II \to \mathcal{R}),\) and so the sheafification of 
		\[T \mapsto \tau_{\leq 1}(\cofib(\II(T) \to \mathcal{R}(T)))\]
		is \(\tau_{\leq 1}\cofib(\II \to \mathcal{R}) \heq [\mathcal{R}/\II],\) as desired. This last identification with the quotient stack is because the cofiber, computed in the category of \emph{presheaves}, has functor of points
		\[\cofib(\II(T) \to \mathcal{R}(T)),\]
		whose 1-truncation is just the groupoid whose objects are elements of the ring \(\mathcal{R}(T),\) and where morphisms \(r \to r'\) are just the data of some \(i \in \II(T)\) so that \(r' = r + i.\) Sheafifying this presheaf of groupoids produces the usual quotient stack.
\end{proof}

However, we defined \(\Cone(\II \to \mathcal{R})\) to be a sheaf of animated \emph{rings}, not a sheaf of anima. Fortunately, the fpqc sheafification of rings we use exists because ring stacks are equivalent to ring objects in the category of stacks, and so we can sheafify the ring structure morphisms (addition, multiplication, etc.) on the presheaf
\[T \mapsto \Cone(\II(T) \to \mathcal{R}(T))\]
to get a ring structure on \([\mathcal{R}/\II].\)

\subsubsection{The quasi-ideals \(\G_a^{\#}\) and \(\G_a^{\#, \wedge}\)}

As explained in \autoref{rem93}, to define the de Rham stack in positive characteristic, we will want to keep track of divided power structures on elements of \(\F_p\)-algebras. We now introduce a sheaf which does this. 
\begin{defn}[Systems of divided powers]
		Let \(\G_a = \Spec \F_p[t], \G_a^{\#} = \Spec \F_p[t, t^2/2!, t^3/3!, ...].\) We will subsequently write \(\F_p\inner{t}\) to denote \(\F_p[t, t^2/2!, ...].\) 

		Let \(A\) be an \(\F_p\)-algebra. A \emph{system of divided powers} for \(a\in A\) is a sequence of \(\gamma_n \in A\) so that \(\gamma_0 = 1, \gamma_1 = a,\) and 
		\[\gamma_n \cdot \gamma_m = \binom{n+m}{n} \cdot \gamma_{n+m}\]
		for all \(n, m \geq 0.\) It's easy to check that divided power structures on \(a\in A\) are in natural bijection with \(\F_p\)-algebra homomorphisms \(\F_p\inner{t} \to A\) sending \(t\) to \(a.\) Thus \(\G_a^{\#}(A)\) is the set of pairs \((a, \gamma_{\bullet}\) where \(a\in A\) and \(\gamma_{\bullet}\) is a system of divided powers for \(a.\) 
\end{defn}

For the sheared de Rham stack, we will need a variant of \(\G_a^{\#}\) keeping track of \emph{nilpotent} divided powers. One should think of the difference between the following definition and the previous one as being the difference between the crystalline site and the nilpotent crystalline site. 
\begin{defn}[Systems of nilpotent divided powers]
		Set \(\G_a^{\#, \wedge} := \Spf \F_p\inner{t}^{\wedge},\) where we let \(\F_p\inner{t}^{\wedge}\) denote the topological ring of all formal power series of the form
		\[\sum_{n=0}^{\infty} a_n \cdot \frac{t^n}{n!}\]
		with \(a_n \in \F_p,\) and the topology is given by the metric
		\[d(f(t), g(t)) = p^{-\nu(f(t) - g(t))},\]
		where \(\nu(f(t))\) is defined to be the smallest \(n\) so that the \(t^n/n!\) coefficient in \(f(t)\) is nonzero. 

		Concretely, \(\G_a^{\#,\wedge}(A)\) can be identified with the set of all pairs \((a, \gamma_{\bullet})\) of \(a \in A\) and \(\gamma_{\bullet}\) a system of divided powers for \(a\in A\) such that \(\gamma_N = 0\) for all \(N \gg 0.\) 
\end{defn}
\begin{remark}
		One can make analogous definitions in mixed characteristic, but we will not need those notions in this paper.
\end{remark}

The morphisms
\[\G_a^{\#} \to \G_a\]
and 
\[\G_a^{\#, \wedge} \to \G_a\]
(corresponding to forgetting the divided power structure) are both quasi-ideals in the sense of Drinfeld. 

\begin{defn}
		We define
		\[\G_a^{\dR} := \Cone(\G_a^{\#} \to \G_a),\]
		\[\G_a^{\hat{\dR}} := \Cone(\G_a^{\#, \wedge} \to \G_a),\]
		where we use \(\Cone\) in the sense of \autoref{quasiideals}. In particular, \(\G_a^{\dR}\) and \(\G_a^{\hat{\dR}}\) are fpqc stacks of 1-truncated animated rings on \(\Sch_{\F_p}.\)
\end{defn}

The definition of \(\Cone\) in \autoref{quasiideals} was a little subtle because it involved sheafification. We now explore how one can make these two ring stacks a bit more concrete. 

\subsubsection{Witt vectors and \(\G_a^{\#}\)}

We start by explaining some simplifications one can do to \(\G_a^{\dR}\), stemming from the relationship between \(\G_a^{\#}\) and Witt vectors. These results all come from \cite{bhattlurie} and \cite{prismatization}, and our presentation follows section 2.6 of Bhatt's \cite{bhatt}. 

\begin{remark}
		We omit most of the proofs in this section because the proofs in the sheared case are entirely analogous, and we feel it is more valuable to prove the sheared versions since these non-sheared results are already very well-documented in the literature.
\end{remark}

\begin{defn}
		We define \(W\) the ring scheme of Witt vectors, sending an \(\F_p\)-algebra \(A\) to the ring \(W(A)\) of \(p\)-typical Witt vectors.
\end{defn}

\begin{lemma} \label{lem114} 
		Let \(R : W \to \G_a\) be the 0th ghost map; in this way, we can view \(\G_a\) (and hence \(\G_a^{\#}\)) as a sheaf of \(W\)-modules. Let \(F : W \to W\) be the Frobenius on Witt vectors. There is a short exact sequence of sheaves of \(W\)-modules
		\[0 \to \G_a^{\#} \to W \xto{F} F_*W \to 0.\]
		In other words, \(F\) is faithfully flat and there is a \(W\)-linear equivalence \(\G_a^{\#} \heq W[F].\) 
\end{lemma}
\begin{proof}
		See \cite{prismatization} lemma 3.2.6, or \cite{bhatt} lemma 2.6.1. The argument is not very involved, though: \(F\) being faithfully flat is clear since \(W \heq \prod_{i\geq 0} \A^1_{\F_p}\) as an \(\F_p\)-scheme. To check that the kernel of Frobenius can be identified with \(\G_a^{\#},\) one does a simple computation to check that \((a_0, a_1, ...) \in W(A)\) lies in the kernel of Frobenius if and only if the \(a_n\) are, up to units, a system of divided powers on \(a_0\) (more precisely, \(a_n\) will be a divided power \(a_0^{p^n}/(p^n)!,\) up to units). 
\end{proof}	

\begin{cor} \label{cor122}
		For \(A\) any \(\F_p\)-algebra, the complex \(R\Gamma_{\operatorname{flat}}(\Spec(A), \G_a^{\#})\) is concentrated in degrees 0 and 1. Moreover, for \(A\) semiperfect, this complex is quasi-isomorphic to \(\G_a^{\#}(A)[0].\) 
\end{cor}
\begin{remark}
		As pointed out by Bhatt (see footnote 22 in \cite{bhatt}, at the bottom of page 28), this shows that the subtleties implicit in fpqc cohomology of \(\G_a^{\#}\) do not matter: the notion of \(H^0\) and \(H^1\) are always meaningful (global sections and torsors), and the higher cohomologies, which a priori might depend on a cutoff cardinal, always vanish. 
\end{remark}
\begin{proof}
		By \autoref{lem114}, it suffices to prove that \(R\Gamma_{\operatorname{flat}}(\Spec(A), W)\) is concentrated in degree 0 (even the semiperfectness claim follows from this, since for \(A\) semiperfect the Frobenius \(W(A) \to W(A)\) is surjective). 

		Write \(W = \prod_{i\geq 0} \G_a\) (which is true as an equality of abelian sheaves). Then
		\begin{align*}
				R\Gamma(\Spec(A), W) &\heq R\Gamma\left(\Spec(A), \inverselim_n \prod_{i=1}^n \G_a\right) \\
									 &\heq R\inverselim_n R\Gamma(\Spec(A), (\G_a)^{\oplus n}) \\
									 &\heq R\inverselim_n A^{\oplus n} \\
									 &\heq W(A),
		\end{align*}
		since the transition maps \(W_n(A) \to W_{n-1}(A)\) of truncated Witt vectors are clearly surjective, and so the Mittag-Leffler condition ensures we have no \(R^1\lim.\) We conclude. 
\end{proof}
\begin{cor} \label{cor138} 
		For any \(\F_p\)-algebra \(A,\) there is a (functorial in \(A\)) equivalence
		\[\G_a^{\dR}(A) \heq R\Gamma_{\operatorname{flat}}(\Spec(A), \G_a^{\dR}).\]
		When \(A\) is semiperfect, we even have
		\begin{equation} \label{eq172} \G_a^{\dR}(A) \heq \cofib(\G_a^{\#}(A) \to A).\end{equation} 
\end{cor}
\begin{proof}
		By \autoref{lem73}, we have \(\G_a^{\dR}(A) \heq [\G_a/\G_a^{\#}](A),\) where the right hand side is the 1-truncated quotient stack. But because of \autoref{cor122}, this 1-truncated quotient stack is actually equal to \(\cofib(\G_a^{\#} \to \G_a)\) (the cofiber taken in the category of fpqc sheaves; since it is equal to the quotient stack \([\G_a/\G_a^{\#}]\) by \autoref{cor122}, the set theoretic issues implicit in its definition are irrelevant). Thus
		\[\G_a^{\dR}(A) \heq [\G_a/\G_a^{\#}](A) \heq R\Gamma_{\operatorname{flat}}(\Spec(A), \G_a^{\dR}).\] 
		The last equality is because
		\[\cofib(\G_a^{\#} \to \G_a)(A) = R\Gamma_{\operatorname{flat}}(\Spec(A), \G_a^{\dR}),\]
		essentially by definition. 

		Moreover, because \(\G_a^{\dR} \heq \cofib(\G_a^{\#} \to \G_a),\) we have an exact triangle 
		\[\G_a^{\#} \to \G_a \to \G_a^{\dR}\]
		of sheaves of \(\G_a\)-modules. Taking flat cohomology and using \autoref{cor122}, for \(A\) semiperfect we get an exact triangle
		\[\G_a^{\#}(A) \to A \to R\Gamma_{\operatorname{flat}}(\Spec(A), \G_a^{\dR}),\]
		of \(A\)-modules, which implies \autoref{eq172}. 
\end{proof}

As the semiperfect \(\F_p\)-algebras form a basis for the flat topology (see \cite{BMS} proposition 4.31 for an even more general statement), it often suffices to check properties of \(\G_a^{\dR}\) just on semiperfect \(\F_p\)-algebras; thus the formula \autoref{eq172} allows us to prove properties of \(\G_a^{\dR}\) without needing to worry about the sheafification used in its definition.

We end with one more useful fact about \(\G_a^{\dR},\) expressing it in terms of Witt vectors. This allows us to compute its homotopy groups, which will be useful later. 
\begin{prop} \label{prop158} 
		View \(\G_a^{\dR}\) as a \(W\)-module via the composition \(W \xto{R} \G_a \to \G_a^{\dR},\) where recall \(R : W\to \G_a\) is the 0th ghost component. There is a natural equivalence of \(W\)-modules 
		\[\G_a^{\dR} \heq \cone(F_*W \xto{p} F_*W).\]
\end{prop}
\begin{proof}
		This is proven in exactly the same way as the sheared version of this statement, which is our \autoref{prop316}. Alternatively, see Corollary 2.6.8 of \cite{bhatt}.
\end{proof}
\begin{cor} \label{cor198}
		As fpqc sheaves of \(W\)-modules, we have
		\[\pi_0(\G_a^{\dR}) = F_*\G_a, \pi_1(\G_a^{\dR}) = F_*\G_a^{\#}.\] 
\end{cor}
\begin{warn}
		In contrast to the other results of this section, in this corollary it is \emph{crucial} that we are working over \(\F_p.\) This is because at some point in the proof we will use that \(FV = VF,\) which of course only holds when we take Witt vectors of rings of characteristic \(p.\) 
\end{warn}
\begin{proof}
		As with \autoref{prop158}, below we will prove the sheared version of this, which has exactly the same proof; we remark that \cite{bhatt} includes this as Corollary 2.6.11. 
\end{proof}

It is incredibly desirable to have sheared variants of \autoref{eq172} and \autoref{cor198}. To this end, we introduce the \emph{sheared} Witt vectors. 

\subsubsection{Sheared Witt vectors and \(\G_a^{\#, \wedge}\)}

Our first goal is to give a variant of \autoref{lem114} involving \(\G_a^{\#, \wedge}.\) As \(\G_a^{\#, \wedge}\) classifies divided powers which are eventually zero, it is natural to introduce the following ideal of \(W.\) 
\begin{defn}
		For any \(\F_p\)-algebra \(A,\) we define \(W'(A) \subseteq W(A)\) the ideal consisting of all sequences \((a_0, a_1, a_2, ...) \in W(A)\) such that \(a_n = 0\) for all \(n \gg 0.\)
\end{defn}

\begin{lemma}[Drinfeld] \label{lem184}
		The equivalence \(W[F] \heq \G_a^{\#}\) induces to an equivalence \(W'[F] \heq \G_a^{\#, \wedge}.\) 
\end{lemma}
\begin{proof}
		See Drinfeld \cite{shearedwitt}, lemma 3.2.5 for full details, but we give a sketch. When we constructed the equivalence \(W[F] \heq \G_a^{\#}\) in \autoref{lem114}, we stated that \((a_0, a_1, ...) \in W\) belongs to \(W[F]\) if and only if the \(a_n\) form a system of divided powers (up to minor technicalities). A system of divided powers belongs to \(\G_a^{\#, \wedge}\) if and only if the powers are eventually zero, we see that this equivalence identifies \(\G_a^{\#, \wedge}\) with \(W'[F],\) as desired.
\end{proof}

Unfortunately, it is not yet immediate to prove an analogue of \autoref{cor122} for \(\G_a^{\#, \wedge}.\) The trouble is that it's not clear that, for \(A\) semiperfect, the map \(W'(A) \to W'(A)\) is surjective. So we introduce another ideal, for which this surjectivity is apparent.
\begin{defn}
		We define \(\hat{W}(A) \subseteq W'(A)\) the ideal consisting of all sequences \((a_0, a_1, ...)\in W(A)\) such that \(a_n = 0\) for all \(n \gg 0,\) and so that each \(a_i\) is nilpotent.
\end{defn}

\begin{lemma}[Drinfeld] \label{lem185} 
		The canonical inclusion \(\hat{W}[F] \inclusion W'[F]\) is an equivalence, and so \autoref{lem184} gives \(\hat{W}[F] \heq \G_a^{\#, \wedge}.\)
\end{lemma}
\begin{proof}
		By \autoref{lem184}, the Witt components of any element of \(W'[F]\) admit divided powers; but any element of a characteristic \(p\) ring admitting divided powers is automatically nilpotent, and so the inclusion \(\hat{W}[F] \inclusion W'[F]\) is an equality, as desired. 
\end{proof}	
\begin{cor} \label{cor1222}
		For \(A\) any \emph{semiperfect} \(\F_p\)-algebra, the complex \(R\Gamma_{\operatorname{flat}}(\Spec(A), \G_a^{\#, \wedge})\) is quasi-isomorphic to \(\G_a^{\#, \wedge}(A)[0].\) 
\end{cor}
\begin{proof}
		Observe that \(\hat{W}\) is stable under \(F\) (this is easiest to see using ghost components). We claim first that if \(A\) is semiperfect, then \(F : \hat{W}(A) \to \hat{W}(A)\) is surjective; this is easy to check by hand, since given \(a = (a_0, a_1, ..., a_n, 0, 0, ...) \in \hat{W}(A),\) by semiperfectness of \(A\) we can find \(b_i \in A\) with \(b_i^p = a_i\), and then \(b = (b_0, b_1, ..., b_n, 0, 0, ...) \in \hat{W}(A)\) has \(F(b) = a.\) 

		By \autoref{lem185}, it thus suffices to prove that \(R\Gamma_{\operatorname{flat}}(\Spec(A), W')\) is concentrated in degree 0 to get our claim, since we have a short exact sequence
		\begin{equation} \label{eq210} 0 \to \G_a^{\#, \wedge} \to \hat{W} \xto{F} \hat{W} \to 0\end{equation}
		of fpqc sheaves, with \(F : \hat{W}(A) \to \hat{W}(A)\) surjective for \(A.\) 

		Thus to get our result, it suffices to prove that \(\Spec(A)\) is \(\hat{W}\)-acyclic in the fpqc topology. As an abelian sheaf, we may write \(\hat{W} = \directlim_n \prod_{i=0}^n \hat{\G_a},\) since every element of \(\hat{W}(A)\) can be represented as a finite tuple of nilpotents. 

		In remark 2.2.18 of \cite{bhatt}, Bhatt uses the short exact sequence
		\[0 \to \hat{\G_a} \to \G_a \to \G_{a,\perf} \to 0\]
		of fppf sheaves (where \(\G_{a,\perf}(R) := R_{\perf}\) is the \emph{colimit} perfection) to prove that 
		\[R\Gamma_{\operatorname{fppf}}(\Spec(A), \hat{\G_a}) = \cofib(A \to A_{\perf})[-1].\]
		For \(A\) semiperfect, the map \(A \to A_{\perf}\) is surjective, and so the above cohomology is just \(\hat{\G_a}(A).\) Note that \(\G_a\) is a quasicoherent sheaf and hence has no higher cohomology; \(\G_{a, \perf},\) being a filtered colimit of quasicoherent sheaves (namely copies of \(\G_a\)), is still quasicoherent and hence the higher fpqc cohomology of \(\G_{a,\perf}\) on \(\Spec(A)\) vanishes; thus we deduce \(R\Gamma_{\operatorname{fpqc}}(\Spec(A), \hat{\G_a}) = \cofib(A \to A_{\perf})[-1]\) as well. 

		Therefore
		\begin{align*}
				R\Gamma_{\operatorname{fpqc}}(\Spec(A), \hat{W}) &\heq R\Gamma_{\operatorname{fpqc}}\left(\Spec(A), \directlim_n \prod_{i=0}^n \hat{\G_a}\right) \\
									  &\heq \directlim_n A^{\oplus (n+1)} \\
									  &\heq \hat{W}(A).
		\end{align*}
		In particular, for a \emph{semiperfect} \(\F_p\)-algebra, we have \(R\Gamma_{\operatorname{fpqc}}(\Spec(A), \hat{W}) = \hat{W}(A).\) 

		We conclude that
		\[R\Gamma_{\operatorname{fpqc}}(\Spec(A), \G_a^{\#, \wedge}) = \G_a^{\#, \wedge}(A)[0],\]
		as desired.
\end{proof}

\autoref{cor1222} furnishes a sheared version of the incredibly powerful \autoref{eq172}. We next introduce the \emph{sheared} Witt vectors, and use them to compute the homotopy groups of \(\G_a^{\hat{\dR}}.\) The idea is to introduce an fpqc sheaf of rings \(\sW\), analogous to the Witt vectors, so that \(\G_a^{\hat{\dR}} \heq F_*\sW/p.\) 

Set \(Q_n := W/\hat{W}[F^n],\) where we take the quotient in the category of fpqc sheaves. Drinfeld then defines
\[\sW := \inverselim_n Q_n,\]
where the transition maps \(Q_{n+1} \to Q_n\) are given by Frobenius. There is a natural map \(\sW \to Q_0 = W.\) Composing with the map \(W \to \G_a\) sending a Witt vector to its 0th ghost component, we get a natural map \(R : \sW \to \G_a.\) 

This definition has a number of nice properties. Firstly, \(\sW\) naturally has a Frobenius \(F : \sW \to \sW\) and Verschiebung \(\tilde{V} : F_*(\sW) \to \sW,\) which have exactly the same properties as the usual Frobenius and Verschiebung on \(W\) do; see \cite{shearedwitt} for extensive discussion of these points, particularly Drinfeld's sections 1.4 and 3.2. In particular, Drinfeld proves 
\begin{enumerate}
		\item \(\tilde{V}(x) \cdot y = \tilde{V}(x \cdot F(y))\) (see 3.2.4 of \cite{shearedwitt}),
		\item \(F\tilde{V} = \tilde{V}F = p\) (see \cite{bmvz}), 
		\item \(F : \sW \to \sW\) is faithfully flat, and \(\sW[F] = \hat{W}[F]\) (Lemma 3.2.5 of \cite{shearedwitt})
		\item \(\tilde{V} : F_*\sW \to \sW\) is injective, and \(\im\tilde{V} = \ker(R : \sW \to \G_a)\) (Lemma 3.2.5 of \cite{shearedwitt}),
		\item moreover, \(R : \sW \to \G_a\) is surjective (Lemma 3.2.5 of \cite{shearedwitt}). 
\end{enumerate}

\begin{prop} \label{prop316}
		There is a natural quasi-isomorphism of animated \(W\)-algebras 
		\[\G_a^{\hat{\dR}} \heq \Cone(F_*\sW \xto{p} F_*\sW) =: F_*\sW/p.\] 
\end{prop}
\begin{proof}
		We will use the properties of \(\tilde{V}, F\) mentioned above. There is a commutative diagram
		\begin{center}
				\begin{tikzcd}
						\G_a^{\#, \wedge} \ar[d] & \G_a^{\#, \wedge} \oplus F_*\sW \ar[d, "{(\can, V)}"] \ar[r] \ar[l] & F_*\sW \ar[d, "p"] \\
						\G_a & \sW \ar[l, "R"] \ar[r, "F"] & F_*\sW.
				\end{tikzcd}
		\end{center}
		Here, the map \(\can : \G_a^{\#, \wedge} \to \sW\) is the one furnishing the identification
		\[\G_a^{\#, \wedge} \heq \sW[F].\] 
		
		Each column is a quasi-ideal (because \(\tilde{V}(x) \cdot y = \tilde{V}(x \cdot F(y))\)), and the two squarse are morphisms of pairs of a ring with a quasi-ideal. Thus the left square induces a map
		\begin{equation} \label{eq264} \cone(\G_a^{\#, \wedge} \oplus F_*\sW \to \sW) \to \G_a^{\hat{\dR}}.\end{equation}

		As \(R : \sW \to \G_a\) is surjective, the map \autoref{eq264} is surjective. As \(\im\tilde{V} = \ker(R),\) we deduce \autoref{eq264} is also injective (as every element of \(R^{-1}(\G_a^{\#, \wedge})\) belongs to the image of \(\G_a^{\#, \wedge} \oplus F_*\sW \to \sW\)), and hence 
		\[\cone(\G_a^{\#, \wedge} \oplus F_*\sW \to \sW) \heq \G_a^{\hat{\dR}}.\] 

		Similarly, the right square induces a map
		\[\cone(\G_a^{\#, \wedge} \oplus F_*\sW \to W) \to F_*\sW/p,\]
		which is also an equivalence: surjectivity is because \(F\) is already surjective, and injectivity is as \(\G_a^{\#, \wedge}\to \sW\) has image \(\hat{W}[F] = \sW[F].\) 

		Composing these two equivalences, we conclude. 
\end{proof}
\begin{cor} \label{cor270} 
		As sheaves of \(\sW\)-modules, we have
		\[\pi_0(\G_a^{\hat{\dR}}) = F_*\G_a, \pi_1(\G_a^{\hat{\dR}}) = F_*\G_a^{\#, \wedge}.\] 
\end{cor}
\begin{proof}
		By \autoref{prop316}, we have \(\G_a^{\hat{\dR}} \heq F_*\sW/p.\) 

		We may write \(p = \tilde{V}F.\) Thus 
		\[\pi_0(F_*\sW/p) = F_*\sW/F_*(V\sW)) \heq F_*(\sW/V\sW).\]
		But the short exact sequence (3.13) of Drinfeld's \cite{shearedwitt} identifies \(\sW/V\sW \heq \G_a,\) and so
		\[\pi_0(\G_a^{\hat{\dR}}) = F_*\G_a,\]
		as desired.

		For \(\pi_1,\) we have
		\[\pi_1(F_*\sW/p) = F_*\ker(p : \sW \to \sW).\]
		But \(p = \tilde{V}F\) with \(\tilde{V}\) injective and \(F\) surjective, so 
		\[\ker(p) = \ker(F : \sW \to \sW) = \sW[F] = \hat{W}[F] = \G_a^{\#, \wedge},\] 
		and so we conclude.
\end{proof}

We have now finally checked the main technical properties of \(\G_a^{\dR}\) and \(\G_a^{\hat{\dR}}\) which we shall use in this note.

\subsection{Defining de Rham stacks}

We can now define the de Rham stack.
\begin{defn}
		Let \(X\) be an \(\F_p\)-stack. We define \(X^{\dR}\) and \(X^{\hat{\dR}}\) by transmutation from the ring stacks \(\G_a^{\dR}\) and \(\G_a^{\hat{\dR}},\) as 
		\[X^{\dR}(A) := X(\G_a^{\dR}(A)),\]
		\[X^{\hat{\dR}}(A) := X(\G_a^{\hat{\dR}}(A)).\] 
\end{defn}

Note that formation of the (sheared) de Rham stack is functorial, and commutes with all derived limits. Morever, the natural maps \(\G_a \to \G_a^{\hat{\dR}} \to \G_a^{\dR}\) (coming from the natural maps of \(\G_a\)-quasi-ideals \(0 \to \G_a^{\#, \wedge} \to \G_a^{\#}\)) induce maps
\begin{equation} \label{eq345} X \to X^{\dR} \to X^{\hat{\dR}}.\end{equation} 

\begin{defn}
		For \(f : X \to S\) a morphism of \(\F_p\)-stacks, we define \(S\)-stacks
		\[(X/S)^{\dR} := X^{\dR} \times_{S^{\dR}} S,\]
		\[(X/S)^{\hat{\dR}} := X^{\hat{\dR}} \times_{S^{\hat{\dR}}} S.\] 
\end{defn}
\begin{remark}
		In principle, one can also define these stacks by transmutation relative to the base \(S,\) but the theory of transmutation over an arbitrary base seems quite delicate, and so we have chosen not to develop it as this more direct definition suffices for our applications. The author thanks Bhargav Bhatt for observing that this transmutation-free definition would lead to a more efficient presentation. 
\end{remark}

We can make a relative version of the morphisms \autoref{eq345}. 
\begin{construction}
		Fix \(f :X \to S\) a morphism of \(\F_p\)-stacks. We define
		\[\pi_{X/S} : X \to (X/S)^{\dR}\]
		the morphism of \(S\)-stacks induced by the commutative diagram
		\begin{center}
				\begin{tikzcd}
				X \ar[r, "\pi_X"] \ar[d, "f"] & X^{\dR} \ar[d, "f^{\dR}"] \\
				S \ar[r, "\pi_S"] & S^{\dR}.
		\end{tikzcd}
		\end{center}	
		To check commutativity, take \(x : \Spec(A) \to X\) an \(A\)-point of \(X.\) Then \(\pi_X(x) \in X^{\dR}(A) = X(\G_a^{\dR}(A))\) arises as the composition
		\[\Spec(\G_a^{\dR}(A)) \xto{\can} \Spec(A) \xto{x} X.\]
		Applying \(f^{\dR}\) gives us the point \(f^{\dR}(\pi_X(x)) \in S^{\dR}(A) = S(\G_a^{\dR}(A))\) defined by the composition
		\begin{equation} \label{eq371} \Spec(\G_a^{\dR}(A)) \xto{\can} \Spec(A) \xto{x} X \xto{f} S.\end{equation}
		
		In contrast, \(\pi_S(f(x)) \in S(\G_a^{\dR}(A))\) is formed as the composition
		\[\Spec(\G_a^{\dR}(A)) \xto{\can} \Spec(A) \xto{f(x)} S,\]
		which is evidently equivalent to \autoref{eq371}. 

		Similarly, we may construct a morphism
		\[\pi_{X/S, \wedge} : X \to (X/S)^{\hat{\dR}}.\] 
\end{construction}

We now prove a few results about \((X/S)^{\dR}\) and \((X/S)^{\hat{\dR}}\) which allow us to more easily prove statements about de Rham stacks via descent.

\subsubsection{Change of base for relative de Rham stacks}

\begin{prop} \label{prop385} 
		Let \(f : X \to S\) be a morphism of \(\F_p\)-stacks, and \(g : S'\to S\) a base change. Define \(X' := X \times_S S',\) recalling that to us all stacks are derived, and so this fiber product is really a \emph{derived} fiber product. Then the diagrams
		\begin{center}
				\begin{tikzcd}
						(X'/S')^{\dR} \ar[r] \ar[d] & (X/S)^{\dR} \ar[d] \\
						S' \ar[r, "g"] & S
				\end{tikzcd}
		\end{center}
		and
		\begin{center}
				\begin{tikzcd}
						(X'/S')^{\hat{\dR}} \ar[r] \ar[d] & (X/S)^{\hat{\dR}} \ar[d] \\
						S' \ar[r, "g"] & S
				\end{tikzcd}
		\end{center}
		are Cartesian diagrams of derived stacks.
\end{prop}
\begin{proof}
		The proof for the de Rham and sheared de Rham stacks are identical, so we only state the proof for the sheared version. As formation of sheared de Rham stacks commutes with derived limits, the diagram
		\begin{center}
				\begin{tikzcd}
						(X')^{\hat{\dR}} \ar[r] \ar[d] & X^{\hat{\dR}} \ar[d] \\
						(S')^{\hat{\dR}} \ar[r] & S^{\hat{\dR}}
				\end{tikzcd}
		\end{center}
		is derived Cartesian. Hence
		\begin{align*}
				(X'/S')^{\hat{\dR}} &\heq (X')^{\hat{\dR}} \times_{(S')^{\hat{\dR}}} S' \\
									&\heq (X^{\hat{\dR}} \times_{S^{\hat{\dR}}} (S')^{\hat{\dR}}) \times_{(S')^{\hat{\dR}}} S' \\
									&\heq X^{\hat{\dR}} \times_{S^{\hat{\dR}}} S',
		\end{align*}
		as desired.
\end{proof}

\autoref{prop385} allows us to check many properties of relative de Rham stacks after performing passing to a flat cover of the base. Sometimes it is useful to also pass to a flat cover of the \emph{source}; our next section will explain how to do this.

\subsubsection{Flatness of \(\pi_{X/S}\) and \(\pi_{X/S, \wedge}\)}

Before giving our result, we recall a basic result from derived deformation theory.
\begin{lemma} \label{lem589}
		Let \(X \to \Spec(k)\) be a \(k\)-scheme, and \(A\) an \emph{animated} \(k\)-algebra. If \(B\) is a square-zero extension of \(A\) by \(M,\) then for any point \(\eta \in X(A),\) there is a fiber sequence of anima
		\[\{\eta\} \times_{X(A)} X(B) \to \pt \to \Map_A(\eta^*\mathbb{L}_{X/k}, M[1]).\]
\end{lemma}
\begin{remark}
		The middle term is \(\pt\) because we've fixed \(k\)-algebra structures on \(A, B.\) 
\end{remark}
\begin{proof}
		Lurie \cite{SAG} proves the spectral algebraic geometry version of this in his Remark 17.3.9.2, but the proof goes through for animated rings; for the reader's benefit, we reproduce it. 

		By definition of square-zero extension of animated rings, there is a pullback square of animated rings
		\begin{center}
				\begin{tikzcd}
						B \ar[r, "f"] \ar[d] & A \ar[d, "i"] \\
						A \ar[r, "s"] & A \oplus M[1],
				\end{tikzcd}
		\end{center}
		where \(i\) is the inclusion and \(s\) is a certain derivation. 

		By Lurie's \cite{DAG} Proposition 5.3.7, the functor of points of a scheme is \emph{infinitesimally cohesive}, which implies that the induced diagram
		\begin{center}
				\begin{tikzcd}
						X(B) \ar[r, "X(f)"] \ar[d] & X(A) \ar[d, "X(i)"] \\
						X(A) \ar[r, "X(s)"] & X(A \oplus M[1])
				\end{tikzcd}
		\end{center}
		is a Cartesian diagram in \(\Ani.\) 

		The projection \(A \oplus M[1] \to A\) induces a map
		\[X(A \oplus M[1]) \to X(A).\]
		By definition of the cotangent complex, we have 
		\[\Map_A(\eta^*\mathbb{L}_{X/k}, M[1]) \heq X(A \oplus M[1]) \times_{X(A)} \{\eta\}.\]
		
		Taking the fiber above \(\eta\) in our \(X(B)\) Cartesian diagram, we then get a Cartesian diagram
		\begin{center}
				\begin{tikzcd}
						\{\eta\} \times_{X(A)} X(B) \ar[r] \ar[d] & \pt \ar[d] \\
						\pt \ar[r] & \Map_A(\eta^*\mathbb{L}_{X/k}, M[1]),
				\end{tikzcd}
		\end{center}
		as desired.
\end{proof}
\begin{cor} \label{cor617}
		Suppose \(X \to \Spec(k)\) is a \emph{smooth} \(k\)-scheme, and \(B \to A\) is a morphism of animated \(k\)-algebras such that \(\pi_0(A) \to \pi_0(B)\) is surjective with nilpotent kernel. Then \(X(B) \to X(A)\) is surjective on \(\pi_0.\) If \(X\) is moreover etale, then \(X(B) \to X(A)\) is an equivalence.
\end{cor}
\begin{proof}
		As in Remark 3.4.2 of Lurie's \cite{DAG}, it suffices to treat the case of square-zero extensions, so assume \(B\) is a square-zero extension of \(A\) by some \(M.\) 

		For any \(\eta \in X(A),\) we have by \autoref{lem589} a fiber sequence
		\[\{\eta\} \times_{X(A)} X(B) \to \pt \to \Map_A(\eta^*\mathbb{L}_{X/k}, M[1]).\]
		When \(X\) is smooth, \(\mathbb{L}_{X/k}\) lives in degree 0, so that
		\[\pi_0(\Map_A(\eta^*\mathbb{L}_{X/k}, M[1])) = \pt,\]
		and hence \(\pi_0(\{\eta\}\times_{X(A)} X(B)) = \pt,\) implying the claimed surjectivity.

		When \(X\) is moreover etale, we have \(\mathbb{L}_{X/k} = 0,\) so that in fact every fiber of \(X(A) \to X(B)\) has the homotopy type of a point. It then follows from Whitehead's theorem that \(X(A) \to X(B)\) is an equivalence in this situation, as desired.
\end{proof}

\begin{prop} \label{prop424} 
		Assume that \(f : X \to S\) is a \emph{smooth} morphism of \(\F_p\)-stacks, and that \(S\) admits a flat cover by a scheme. Then the morphisms
		\[\pi_{X/S} : X \to (X/S)^{\dR}, \pi_{X/S, \wedge} : X \to (X/S)^{\hat{\dR}}\]
		are surjective in the fpqc topology. Moreover, if \(f : X \to S\) is \emph{etale}, then \(\pi_{X/S}\) and \(\pi_{X/S, \wedge}\) are both equivalences.
\end{prop}
\begin{proof}
		By \autoref{prop385} and the assumption that \(S\) admits a flat cover by a scheme, we can assume without loss of generality that actually \(S\) is an affine \emph{scheme}, \(S = \Spec(k)\) (for \(k\) some \(\F_p\)-algebra).  

		In this case, it is the case that \((X/k)^{\dR}\) and \((X/k)^{\hat{\dR}}\) are defined by transmutation of \(\G_{a,k}^{\dR}\) and \(\G_{a,k}^{\hat{\dR}},\) the \(k\)-algebra stacks defined as the restrictions of \(\G_a^{\dR}, \G_a^{\hat{\dR}}\) to the category of \(k\)-algebras.

		Now, let \(A\) be a semiperfect \(k\)-algebra. By \autoref{prop316}, we have
		\[\G_a^{\hat{\dR}}(A) = \Cone(\G_a^{\#,\wedge}(A) \to A).\]
		It follows that the canonical map \(A \to \G_a^{\hat{\dR}}(A)\) is surjective, and has nilpotent kernel (since any element of \(A\) admitting divided powers is automatically nilpotent of order at most \(p\)). 

		Deformation theory (see \autoref{cor617}) tells us that, since \(X \to \Spec(k)\) is \emph{smooth}, and \(A \to \G_a^{\hat{\dR}}(A)\) is surjective on \(\pi_0\) with nilpotent kernel, the morphism \(X(A) \to X(\G_a^{\hat{\dR}}(A))\) of anima is surjective on \(\pi_0.\) As the semiperfect \(k\)-algebras form a basis for the flat topology, this immediately implies \(\pi_{X/S, \wedge}\) is surjective in the fpqc topology. When \(f\) is furthermore \emph{etale}, deformation theory (again see \autoref{cor617}) furthermore gives us that \(X(A) \to X(\G_a^{\hat{\dR}}(A))\) is an equivalence of anima; thus \(X(A) \to X^{\dR}(A)\) is an equivalence for all semiperfect \(A\), and hence \(\pi_{X/S}\) is an equivalence of stacks.

		The exact same argument works for \((X/S)^{\dR},\) so we conclude. 
\end{proof}	

\begin{prop} \label{prop443}
		Fix an \(\F_p\)-stack \(S.\) If \(j : U \to X\) is an etale morphism of \(S\)-stacks, and \(X\) admits a flat cover by a scheme, then 
		\begin{center}
				\begin{tikzcd}
						U \ar[r] \ar[d, "j"] & (U/S)^{\hat{\dR}} \ar[d, "j^{\hat{\dR}}"] \\
						X \ar[r] & (X/S)^{\hat{\dR}}
				\end{tikzcd}
		\end{center}
		is Cartesian. In particular, by \autoref{prop424}, if \(X\to S\) is smooth, then the map \(j^{\hat{\dR}}\) is etale as well. 

		The same statements holds for the non-sheared de Rham stacks. 
\end{prop}
\begin{proof}
		The same argument will hold for sheared and non-sheared de Rham stacks, so we just prove it for sheared de Rham stacks. By \autoref{prop385}, we have
		\[(U/S)^{\hat{\dR}} \times_{(X/S)^{\hat{\dR}}} X = (U/X)^{\hat{\dR}}.\]
		Since \(X\) admits a flat cover by a scheme, \autoref{prop424} applies and so we deduce that \(\pi_{U/X} : U \to (U/X)^{\hat{\dR}}\) is an equivalence, so we conclude. 
\end{proof}

\subsection{The gerbe property}

In characteristic \(p,\) despite the abstract seeming definitions, the de Rham stack and the sheared de Rham stack are actually quite concrete geometric objects; the purpose of this section is to describe their geometry. This can be viewed as a stacky form of Bezrukavnikov's observation that, in positive characteristic, the ring of differential operators is an Azumaya algebra (see \cite{bezazumaya}).

Fix \(f : X \to S\) a smooth, representable morphism of stacks. 
\begin{warn}
		We remind the reader that to us, representable always means representable by \emph{schemes}. 
\end{warn}

We write \(X^{(1)}\) for the relative Frobenius twist of \(X\) over \(S,\) defined as the fiber product
\begin{center}
		\begin{tikzcd}
				X^{(1)} \ar[r] \ar[d] & X \ar[d] \\
				S \ar[r, "F_S"] & S,
		\end{tikzcd}
\end{center}
where \(F_S : S \to S\) is the absolute Frobenius. 

We now give a relative version of proposition 2.7.1 in \cite{bhatt}. 

\begin{notn}
		For a stack \(X,\) and a vector bundle \(\mathcal{E}\) on \(X,\) we define
		\[\mathbb{V}(\mathcal{E}) := \SPEC_X \Sym_{\OO_X}^{\bullet}(\mathcal{E}^{\vee}).\] 
\end{notn}

\begin{construction} \label{nuconst}
		We construct morphisms of \(S\)-stacks
		\[\nu_{X/S} : (X/S)^{\dR} \to X^{(1)},\]
		\[\nu_{X/S, \wedge} : (X/S)^{\hat{\dR}} \to X^{(1)}.\]
		
		We start by defining 
		\[\nu_X : X^{\dR} \to X\]
		by transmutation of the natural map 
		\[\G_a^{\dR} \to \pi_0(\G_a^{\dR}) = F_*\G_a.\]
		In particular, 
		\[\nu_X \circ \pi_X = F_X\]
		since the composition is just the transmutation of 
		\[\G_a \to \G_a^{\dR} \to \pi_0(\G_a^{\dR}) = F_*\G_a.\]
		As \(\nu_X, \pi_X\) are both defined by transmutation, we automatically find that the diagram
		\begin{center}
				\begin{tikzcd}
						X \ar[r, "\pi_X"] \ar[d, "f"] & X^{\dR} \ar[d, "f^{\dR}"] \ar[r, "\nu_X"] & X \ar[d, "f"] \\
						S \ar[r, "\pi_S"] & S^{\dR} \ar[r, "\nu_S"] & S
				\end{tikzcd}
		\end{center}
		commutes. 

		Recall that \((X/S)^{\dR}\) is defined as the fiber product 
		\begin{center}
				\begin{tikzcd}
						(X/S)^{\dR} \ar[r, "\pi_S'"] \ar[d, "s"] & X^{\dR} \ar[d, "f^{\dR}"] \\
						S \ar[r, "\pi_S"] & S^{\dR},
				\end{tikzcd}
		\end{center}

		It follows that 
		\[F_S \circ s = \nu_S\pi_S \circ s = \nu_S \circ (f^{\dR} \circ \pi_{S'}) = f\nu_X\pi_{S'}.\]
		Thus \(s : (X/S)^{\dR} \to S\) and \(\nu_X\pi_{S'} : (X/S)^{\dR} \to X\) define a morphism 
		\[\nu_{X/S} : (X/S)^{\dR} \to X^{(1)}.\]
		Adding hats everywhere (or precomposing with the natural map \((X/S)^{\hat{\dR}} \to (X/S)^{\dR}\)) gives us a similar morphism
		\[\nu_{X/S, \wedge} : (X/S)^{\hat{\dR}} \to X^{(1)}.\]

		Moreover, the composition \(\nu_{X/S} \circ \pi_{X/S} : X \to X^{(1)}\) is the relative Frobenius \(F_{X/S},\) and ditto with the sheared version; this is immediate from the formula \(\nu_X\pi_X = F_X.\)
\end{construction}

In order to prove the gerbe property by descent, the following proposition will be useful. 
\begin{prop} \label{prop383}
		Let \(f : X \to S\) morphism of \(\F_p\)-stacks, and \(g : S' \to S\) the base change. Set \(X' = X\times_S S'.\) Then, the diagram
		\begin{center}
				\begin{tikzcd}
						(X'/S')^{\hat{\dR}} \ar[r] \ar[d, "\nu_{X'/S', \wedge}"] & (X/S)^{\hat{\dR}} \ar[d, "\nu_{X/S, \wedge}"] \\
						(X')^{(1)} \ar[r] & X^{(1)}
				\end{tikzcd}
		\end{center}
		commutes.
\end{prop}
\begin{proof}
		Immediate from combining \autoref{prop385} with the fact that formation of relative Frobenius twists is compatible with base change.
\end{proof}

\begin{remark}
		A forthcoming work of Gleb Terentiuk will give a more general, and much stronger, version of this gerbe property. Because of this, we only prove a somewhat weaker statement, assuming restrictive hypotheses which will suffice for our applications to logarithmic geometry. In showing the pseudotorsor claim, Terentiuk only needs to assume the morphism \(X\to S\) is lci, not smooth. And, while we only show non-emptiness flat locally, Terentiuk will show non-emptiness locally for the quasi-synotomic topology.
\end{remark}

\begin{theorem}[The de Rham stack is a gerbe] \label{gerbe} 
		Recall that we assumed \(X/S\) is smooth and representable (which, to us, means representable by \emph{schemes}). In particular, the cotangent complex \(\mathbb{L}_{X^{(1)}/S}\) is actually a vector bundle on \(X^{(1)}.\) 

		The morphism \(\nu_{X/S} : (X/S)^{\dR} \to X^{(1)}\) of \autoref{nuconst} is a \(B\mathbb{V}(\mathbb{L}_{X^{(1)}/S}^{\vee})^{\#}\)-torsor in the flat topology. Similarly, the morphism \(\nu_{X/S, \wedge} : (X/S)^{\hat{\dR}} \to X^{(1)}\) is a \(B\mathbb{V}(\mathbb{L}_{X^{(1)}/S}^{\vee})^{\#, \wedge}\)-torsor in the flat topolgoy. 

		In fact, both of these torsors are split after pullback along the relative Frobenius \(F_{X/S} : X \to X^{(1)}.\)
\end{theorem}
\begin{proof}
		Of course being a torsor can be checked flat locally on \(X^{(1)},\) and hence flat locally on \(S.\) Using \autoref{prop383}, we may therefore assume \(S = \Spec(k)\) is an \emph{affine} scheme. By representability of \(f\), we may now assume that \(X\) is a scheme as well. In this generality, with \(S\) an affine scheme, Bhatt's \cite{bhatt} proposition 2.7.1 literally applies. However, as this result is very important for us, and as \cite{bhatt} is a little sparse on details, we will give the full proof. 

		We can understand the fibers of \((X/S)^{\dR} \to X^{(1)}\) using deformation theory (more particularly, in the form of \autoref{lem589}). Fix a morphism of \(\F_p\)-schemes \(\Spec(A) \to S,\) and let \(x : \Spec(A) \to X^{(1)}\) be an \(A\)-point of \(X^{(1)}\), and let \(y : \Spec(F_*A) \to X\) be the corresponding point of \(X.\) An \(A\)-point of \((X/S)^{\dR}\) lifting \(x\) is then the same as a map 
		\[\bar{x} : \Spec \G_a^{\dR}(A) \to X\]
		lifting \(y.\) In other words, the \(A\)-points of the fiber of \((X/S)^{\dR} \to X^{(1)}\) is equivalent to the fiber of 
		\[X(\G_a^{\dR}(A)) \to X(F_*A)\]
		above \(y.\) 

		Deformation theory (see \autoref{lem589}) tells us that this fiber will either be empty, or a torsor under
		\begin{align*}
				\Hom_{F_*A}(y^*\mathbb{L}_{X/S}, F_*B\G_a^{\#}(A)) &= y^*\mathbb{L}_{X/S}^{\vee} \otimes_{F_*A} F_*B\G_a^{\#}(A) \\
																   &= x^*\mathbb{L}_{X^{(1)}/S}^{\vee} \otimes_A B\G_a^{\#}(A) \\ 
		\end{align*}
		We conclude that \((X/S)^{\dR} \to X^{(1)}\) is a pseudogerbe. The same argument works for the sheared de Rham stack \((X/S)^{\hat{\dR}}\) once we write \(\G_a^{\hat{\dR}}\) as a square-zero extension. As any 1-truncated animated ring is a square-zero extension of its \(\pi_0\) by its \(\pi_1[1],\) it suffices to compute the \(\pi_0, \pi_1\) of \(\G_a^{\hat{\dR}},\) but this was done in \autoref{cor270}. 

		Next, we must prove non-emptiness locally in the flat topology; the proof will be exactly the same for both \((X/S)^{\dR}\) and \((X/S)^{\hat{\dR}},\) so we just give the proof for \((X/S)^{\hat{\dR}}.\) As \(X \to X^{(1)}\) is a cover in the flat topology, it suffices to show that the morphism \((X/S)^{\hat{\dR}} \times_{X^{(1)}} X \to X\) admits a section. 

		Constructing such a section is equivalent to giving a map \(X \to (X/S)^{\hat{\dR}}\) together with a homotopy between the composition \(X \to (X/S)^{\hat{\dR}} \to X^{(1)}\) and the relative Frobenius \(F_{X/S} : X \to X^{(1)}.\) But we already have such a morphism: earlier, we constructed a map \(\pi_{X/S, \wedge} : X\to (X/S)^{\hat{\dR}}\) by transmutation of the natural map
		\[\G_{a, S} \to \G_{a, S}^{\hat{\dR}}.\]
		We constructed \(\nu_{X/S, \wedge} : (X/S)^{\hat{\dR}} \to X^{(1)}\) as the transmutation of \(\G_{a, S}^{\hat{\dR}} \to F_*\G_{a, S}.\) Thus \(\nu_{X/S, \wedge} \circ \pi_{X/S, \wedge} : X \to X^{(1)}\) is obtained by transmutation from the composition \(\G_{a, S} \to \G_{a, S}^{\hat{\dR}} \to F_*\G_{a, S}.\) But this composition is just the natural map \(\G_{a, S} \to F_*\G_{a, S},\) whose transmutation is nothing but the relative Frobenius. 
\end{proof}
\begin{notn} \label{cansplit}
		In the proof of \autoref{gerbe}, we constructed an explicit splitting of \((X/S)^{\hat{\dR}} \times_{X^{(1)}} X \to X,\) being induced from \(\pi_{X/S,\wedge} : X \to (X/S)^{\hat{\dR}}.\) We will now call this the \emph{canonical splitting} of \((X/S)^{\hat{\dR}} \times_{X^{(1)}} X \to X.\) 

		The canonical splitting induces a particular isomorphism
		\[\alpha : (X/S)^{\hat{\dR}} \times_{X^{(1)}} X \heq B\mathbb{V}(F_{X/S}^*\mathcal{T}_{X^{(1)}/S})^{\#, \wedge}.\] 
		In our discussion of \(p\)-curvature, we will use this \emph{particular} isomorphism \(\alpha.\) 
\end{notn}

\subsection{Relation to the classical notion of \(\mathcal{D}\)-modules} 

Fix \(f : X \to S\) a smooth, representable morphism of algebraic \(\F_p\)-stacks. We now show that quasicoherent sheaves on \((X/S)^{\hat{\dR}}\) recover the classical notion of \(\mathcal{D}_{X/S}\)-modules.

\subsubsection{The algebra \(\mathcal{D}_{X/S}\)} 

We start by defining a (non-commutative!) quasicoherent \(\OO_X\)-algebra \(\mathcal{D}_{X/S}.\) As \(f\) is smooth and representable, its cotangent complex is just the vector bundle \(\Omega^1_{X/S}.\) In particular, we can define the tangent bundle \(\mathcal{T}_{X/S} := \underline{\Hom}(\Omega_{X/S}^1, \OO_X).\) 

This \(\mathcal{T}_{X/S}\) has the structure of a Lie algebroid; that is, there is a natural action 
\[\mathcal{T}_{X/S} \oplus \OO_X \to \OO_X\]
of \(\mathcal{T}_{X/S}\) on functions (defined by sending \((\partial, f)\) to \(\partial(df),\) viewing \(\partial\) as a map \(\Omega_{X/S}^1 \to \OO_X\)), and a natural Lie bracket
\[\mathcal{T}_{X/S} \oplus \mathcal{T}_{X/S} \to \mathcal{T}_{X/S}\]
defined by \([\partial, \partial'](\omega) = \partial(\partial'(\omega)) - \partial'(\partial(\omega)).\) 

Note that the action \(\mathcal{T}_{X/S} \oplus \OO_X \to \OO_X\) is not \(\OO_X\)-bilinear, and likewise the Lie bracket is not \(\OO_X\)-bilinear; instead both obey suitable Leibniz rules.

We can then define \(\mathcal{D}_{X/S}\) to be the universal enveloping algebra of this Lie algebroid.
\begin{defn}
		We let \(\mathcal{D}_{X/S}\) be the quasicoherent \(\OO_X\)-algebra defined such that, for any morphism \(\phi : \Spec(A) \to X,\) the pullback \(\phi^*(\mathcal{D}_{X/S})\) is the associative \(A\)-algebra generated by the free \(A\)-module \(\phi^*(\mathcal{T}_{X/S}),\) subject to the relations that 
		\[\partial \cdot \partial' - \partial' \cdot \partial = [\partial, \partial'],\]
		\[\partial \cdot a - a \cdot \partial = \partial(a),\]
		where in the first relation we mean to take the bracket of \(\partial, \partial'\) (the Lie bracket on \(\mathcal{T}_{X/S}\) pulls back to a pairing on \(\phi^*(\mathcal{T}_{X/S})\)) and in the second relation by \(\partial(a)\) we mean the element of \(A\) obtained by acting \(\partial\) on \(a.\) 
\end{defn}
\begin{exmp}
		Consider the smooth morphism \(f : \pt \to B\G_m.\) By considering the Cartesian diagram
		\begin{center}
				\begin{tikzcd}
						\G_m \ar[r] \ar[d] & \pt \ar[d] \\
						\pt \ar[r] & B\G_m,
				\end{tikzcd}
		\end{center}
		one can compute that \(\mathcal{T}_{\pt/B\G_m} = k \cdot t\partial_t.\) That is, \(\mathcal{T}_{\pt/B\G_m}\) is free of dimension 1, with a generator we call \(t\partial_t.\) The action
		\[\mathcal{T}_{\pt/B\G_m} \oplus \OO_{\pt} \to \OO_{\pt}\]
		is the zero map, since \(d : \OO_{\pt} \to \Omega^1_{\pt/B\G_m}\) is the zero map. 

		Thus \(\mathcal{D}_{\pt/B\G_m} = k[t\partial_t].\) In particular, the pullback of \(\mathcal{D}_{\pt/B\G_m}\) to \(\G_m\) is just the \emph{commutative} algebra \(k[t^{\pm 1}, t\partial_t],\) which is \emph{not} the same as the non-commutative algebra \(\mathcal{D}_{\G_m/\pt}.\) This failure of base change is because the pullback of 
		\[\mathcal{T}_{\pt/B\G_m} \oplus \OO_{\pt} \to \OO_{\pt}\]
		to \(\G_m\) is the zero map
		\[\mathcal{T}_{\G_m/\pt} \oplus \OO_{\G_m} \to \OO_{\G_m},\]
		and not the natural action of \(\mathcal{T}_{\G_m/\pt}\) on \(\OO_{\G_m}.\) 
\end{exmp}

We can describe modules over \(\mathcal{D}_{X/S}\) concretely. 
\begin{lemma} \label{lem583} 
		Let \(\mathcal{E}\in\mathcal{D}_{\qc}(X).\) Then \(\mathcal{D}_{X/S}\)-module structures on \(\mathcal{E}\) are equivalent to giving maps
		\[\nabla : \mathcal{E} \to \mathcal{E} \otimes \Omega_{X/S}^1\]
		in the category of fppf sheaves on \(X,\) such that
		\[\nabla(f \cdot s) = f\nabla(s) + s \otimes df,\]
		and such that \(\nabla\) is \emph{flat}: the composition
		\[\mathcal{E} \xto{\nabla} \mathcal{E} \otimes \Omega_{X/S}^1 \xto{\nabla \otimes \id_{\Omega_{X/S}^1}} \mathcal{E} \otimes \Omega_{X/S}^1 \otimes \Omega_{X/S}^1 \xto{\id_{\mathcal{E}} \otimes \can} \mathcal{E} \otimes \Omega_{X/S}^2\]
		vanishes.
\end{lemma}
\begin{proof}
		To give a \(\mathcal{D}_{X/S}\)-action, one just needs to say how the generators act and then check the relations; the map \(\nabla\) encodes maps \(\nabla_{\partial} : \mathcal{E} \to \mathcal{E}\) for each \(\partial \in \mathcal{T}_{X/S},\) and the Leibniz rule and flatness condition are equivalent to the relations imposed on \(\mathcal{D}_{X/S}.\) 
\end{proof}
\begin{remark} \label{rem592}
		It follos that \(\mathcal{D}_{X/S}\)-modules admit a symmetric monoidal structure, defined by the usual formula
		\[\nabla \otimes \nabla' := \id \otimes \nabla' + \nabla \otimes \id.\]
\end{remark} 

\subsubsection{The comparison theorem}

\begin{theorem} \label{classicalcomparison} 
		Let \(f : X\to S\) be a smooth representable morphism of \(\F_p\)-stacks. The functor 
		\[\mathcal{D}_{\qc}((X/S)^{\hat{\dR}}) \to \mathcal{D}_{\qc}(X)\]
		obtained by \(\ast\)-pullback along \(\pi_{X/S, \wedge} : X \to (X/S)^{\hat{\dR}}\) factors through the forgetful functor \(\Mod_{\qc}(\mathcal{D}_{X/S}) \to \mathcal{D}_{\qc}(X)\) (where \(\Mod_{\qc}(\mathcal{D}_{X/S})\) denotes the \(\infty\)-category of quasicoherent \(\mathcal{D}_{X/S}\)-modules), and the induced functor
		\[\mathcal{D}_{\qc}((X/S)^{\hat{\dR}}) \to \Mod_{\qc}(\mathcal{D}_{X/S})\]
		is an equivalence. 

		Moreover, this equivalence is symmetric monoidal for the usual \(\otimes\)-structure on \(\mathcal{D}_{\qc}((X/S)^{\hat{\dR}}),\) and the \(\otimes\)-structure of \autoref{rem592} on \(\Mod_{\qc}(\mathcal{D}_{X/S}).\) 
\end{theorem}
\begin{proof}
		This is done in \cite{bmvz}. 
\end{proof}

For later use, we record the following observation. 
\begin{remark} \label{rem590} 
		When \(X'\to S'\) is a smooth morphism of \emph{schemes}, there is a natural inclusion
		\[0 \to \mathcal{T}_{X'/S'} \inclusion \mathcal{D}_{X'/S'},\]
		which takes the Lie bracket of vector fields to the commutator in the algebra \(\mathcal{D}_{X'/S'}.\) As these inclusions are compatible with base change, we can glue them to define a map
		\[0 \to \mathcal{T}_{X/S} \to \mathcal{D}_{X/S}\]
		for our smooth, representable morphism \(f : X \to S.\) 

		We warn that this morphism is \(\OO_X\)-linear only when we treat \(\mathcal{D}_{X/S}\) as a \emph{right} \(\OO_X\)-module. 
\end{remark}

\subsubsection{Describing \(\nu^*\) classically}

In \autoref{gerbe}, we produced a morphism
\[\nu_{X/S, \wedge} : (X/S)^{\hat{\dR}} \to X^{(1)}.\]
We now describe the resulting functor \[\nu_{X/S,\wedge}^* : \mathcal{D}_{\qc}(X^{(1)}) \to \mathcal{D}_{\qc}((X/S)^{\hat{\dR}}) \heq \Mod_{\qc}(\mathcal{D}_{X/S})\] concretely. 

\begin{prop} \label{prop721} 
		The functor
		\[\nu_{X/S,\wedge}^* : \mathcal{D}_{\qc}(X^{(1)}) \to \Mod_{\qc}(\mathcal{D}_{X/S})\]
		sends \(\mathcal{E}' \in \mathcal{D}_{\qc}(X^{(1)})\) to the pair \((\mathcal{E}, \nabla)\) where \(\mathcal{E} := F_{X/S}^*\mathcal{E}'\) and the connection \(\nabla\) obeys the following: for any scheme \(T\) and map \(T \to S,\) the pullback of \((\mathcal{E}, \nabla)\) to \((X_T/T)^{\hat{\dR}}\) is given by \((\mathcal{E}_T, \nabla_T)\) where \(\mathcal{E}_T = F_{X_T/T}^*(\mathcal{E}'_T)\) (for \(\mathcal{E}'_T\) the pullback of \(\mathcal{E}'\) to \((X_T/T)^{(1)}\)) and 
		\[\nabla_T(f \cdot F_{X_T/T}^{-1}(s)) = df \otimes F_{X_T/T}^{-1}(s).\] 
\end{prop}
\begin{proof}
		This claim can be checked etale locally, so we may assume \(S = \Spec(k)\) is an affine scheme and \(X = \A^d_k.\) In that case, the relative Frobenius twist \(X^{(1)}\) is an affine scheme, and in particular \(\mathcal{D}_{\qc}(X^{(1)})\) is generated under colimits by \(\OO_{X^{(1)}}.\) 

		In the statement of the proposition, we defined some functor \(F : \mathcal{D}_{\qc}(X^{(1)}) \to \Mod_{\qc}(\mathcal{D}_{X/S}),\) which sends \(\mathcal{E}'\) to the explicitly described connection \(F(\mathcal{E}') := (\mathcal{E}, \nabla).\) This functor \(F\) is symmetric monoidal (by explicit computation) and cocontinuous (since \(F_{X/S}^*\) is cocontinuous, and because the connection \(\nabla\) is defined via a tensor product and \(\otimes\) is compatible with colimits; alternatively, one can explicitly write down a right adjoint to \(F,\) which we do in \autoref{solfunctor}). 

		Thus \(\nu_{X/S,\wedge}^*\) and \(F\) are two cocontinuous \(\otimes\)-functors, implying they are equivalent since \(\mathcal{D}_{\qc}(X^{(1)})\) is generated under colimits by \(\OO_{X^{(1)}},\) the unit of \(\otimes.\) This generation is because \(X^{(1)}\) is an affine scheme in our situation, as we reduced at the start to considering \(X = \A^d_k.\) 
\end{proof}

\subsubsection{The Azumaya algebra property}

The fact that \((X/S)^{\dR}\) is a gerbe over \(X^{(1)}\) is closely related to an observation of Bezrukavnikov \cite{bmm}, that in the classical situation \(F_{X/S,*}\mathcal{D}_{X/S}\) is always an Azumaya algebra over its center, which can be identified with \(\OO_{\mathbb{V}(\Omega^1_{X^{(1)}/S}}.\) 

Unfortunately, Bezrukavnikov's amazing Azumaya algebra property is not always true for stacks. 
\begin{exmp} \label{noazumaya} 
		Consider the canonical morphism \(\pt \to B\G_m.\) Using the Cartesian diagram
		\begin{center}
				\begin{tikzcd}
						\G_m \ar[r] \ar[d] & \pt \ar[d] \\
						\pt \ar[r] & B\G_m,
				\end{tikzcd}
		\end{center} 
		and that \(\mathcal{D}_{\G_m/\pt} = k[t^{\pm 1}, \partial_t],\) we deduce via descent that \(\mathcal{D}_{\pt/B\G_m} = k[t\partial_t].\) 

		In particular, \(k[t\partial_t]\) is a \emph{commutative} algebra, and so \(F_{\pt/B\G_m, *}(k[t\partial_t])\) will be a certain commutative algebra on \(B\mu_p\), and in particular have center much too large for the claim about the center being the cotangent bundle to hold. 
\end{exmp} 
\begin{remark}
		This phenomenon is possible essentially because \(\mathcal{D}_{X/S}\) is defined via a certain Lie algebroid, and Lie algebroids do not pullback well (because of the anchor map). This prevents us from checking this Azumaya algebra property by a descent argument, and so the fact that the Azumaya algebra holds for \(\G_m \to \pt\) does not imply anything about the case of \(\pt \to B\G_m.\) 
\end{remark}
\begin{remark}
		Another way of understanding this failure is to note that the usual proof of the Azumaya property essentially goes by choosing etale coordinates, and then checking using those etale coordinates. But unfortunately, \(\pt\) is too small and so we can't reduce to an easier to check case. 
\end{remark}

\subsection{The solutions functor} \label{solfunctor}

As alluded to in the proof of \autoref{prop721}, we can use the explicit description of \autoref{prop721} to write down an explicit formula for the adjoint to \(\nu_{X/S,\wedge}^*,\) which we call the \emph{solutions functor}. This solutions functor will play a crucial role in our study of Cartier descent later.

\subsubsection{Technical point: representability of \(F_{X/S}\)}

We start with a useful lemma about the relative Frobenius.
\begin{lemma} \label{lem777} 
		Let \(f : X \to S\) be a smooth, representable morphism of \(\F_p\)-stacks, such that \(S\) admits a flat cover by a scheme.

		Then the relative Frobenius \(F_{X/S} : X \to X^{(1)}\) is representable and qcqs; in fact, it is even finite. Moreover, \(F_{X/S}\) is faithfully flat.
\end{lemma}
\begin{proof}
		First, we prove representability \emph{by algebraic spaces}. As \(S\) has an fppf cover by a scheme, we can base change and assume \(S\) is a scheme -- it suffices to check representability after such a base change because representability by algberaic spaces can be checked fppf locally on the base; see tag 04ZP of the Stacks Project \cite{stacks}. 

		And in the situation where \(S\) is a scheme, representability of \(X\) allows us to deduce that \(X\) is a scheme, and so \(F_{X/S} : X \to X^{(1)}\) is just a morphism of schemes, meaning of course it is representable by algebraic spaces. For smooth morphisms, the relative Frobenius is known to be finite and faithfully flat (see tag 0FW2 of the Stacks Project \cite{stacks}); as these properties can also be checked after fppf base change, we deduce that \(F_{X/S}\) is representable by algebraic spacse, finite, and faithfully flat.

		To conclude, we just need to prove that the relative Frobenius is representable by schemes. Take \(T\) a scheme, and \(T \to X^{(1)}\) a morphism. Then \(X\times_{X^{(1)}} T \to T\) is a finite morphism of algebraic spaces; by tag 03XX in the Stacks Project \cite{stacks}, it follows that \(X\times_{X^{(1)}} T\) is automatically a scheme; in particular, \(F_{X/S}\) is representable by schemes, as desired.
\end{proof}
\begin{remark} \label{rem791}
		The purpose of \autoref{lem777} is that, now that we know \(F_{X/S}\) is representable and qcqs, the (derived) pushforward \(F_{X/S,*}\) of fppf sheaves on \(X\) to fppf sheaves on \(S\) will send quasicoherent sheaves to quasicoherent sheaves. 

		This might seem pointless to check, since if \(F_{X/S,*}\) doesn't send quasicoherents to quasicoherents we could just take the right adjoint to \(F_{X/S}^*.\) However, because we are going to be pushforward things like connections \(\nabla : \mathcal{E} \to \mathcal{E} \otimes \Omega_{X/S}^1\) which are not \(\OO_X\)-linear, it is vital to know that this is just the usual pushforward of sheaves. 
\end{remark}

\subsubsection{Constructing solutions}

\begin{construction}[The solutions functor] 
		We construct a functor
		\[\Sol_{X/S} : \Mod_{\qc}(\mathcal{D}_{X/S}) \to \mathcal{D}_{\qc}(X^{(1)}).\] 
		If \(X/S\) is clear from context, we will sometimes just refer to \(\Sol_{X/S}\) as \(\Sol\) or \(\Sol_X.\) 

		Let \((\mathcal{E}, \nabla) \in \Mod_{\qc}(\mathcal{D}_{X/S}).\) We can pushforward the connection
		\[\nabla : \mathcal{E} \to \mathcal{E} \otimes \Omega_{X/S}^1\]
		along \(F_{X/S}\), to get a morphism 
		\[F_{X/S,*}\nabla : F_{X/S,*}\mathcal{E} \to F_{X/S,*}(\mathcal{E} \otimes \Omega_{X/S}^1).\]
		Here, we take the pushforward just at the level of fppf sheaves on \(X.\) However, despite the fact that \(\nabla\) was not \(\OO_X\)-linear, the Leibniz rule implies \(F_{X/S,*}\nabla\) \emph{is} \(\OO_{X^{(1)}}\)-linear; thus \(F_{X/S,*}(\nabla)\) is a morphism in \(\mathcal{D}_{\qc}(X^{(1)})\) (see \autoref{lem777}). 

		We then define 
		\[\Sol_{X/S}(\mathcal{E}, \nabla) := \fib(F_{X/S,*}(\nabla)).\] 
\end{construction}

\begin{prop} \label{prop813}
		The functor \(\Sol_{X/S}\) is right adjoint to \(\nu_{X/S,\wedge}^*.\) Moreover, the unit
		\[\id \to \Sol_{X/S} \circ \nu_{X/S, \wedge}^*\]
		of the adjunction is an equivalence, and thus \(\nu_{X/S,\wedge}^*\) is fully faithful. 
\end{prop}
\begin{proof}
		The adjunction can be understood more or less explicitly. Let \((\mathcal{E}, \nabla) \in \Mod(\mathcal{D}_{X/S}),\) and \(\FF \in \mathcal{D}_{\qc}(X^{(1)}).\) Then \(\Map_{(X/S)^{\hat{\dR}}}(\nu_{X/S,\wedge}^*\FF, (\mathcal{E}, \nabla))\) is equivalent to the subanima of \(\Map_X(F_{X/S}^*\FF, \mathcal{E})\) spanned by the maps \(\phi : F_{X/S}^*\FF \to \mathcal{E}\) for which the diagram
		\begin{center}
				\begin{tikzcd}
						F_{X/S}^*\FF \ar[r, "\phi"] \ar[d, "\nabla'"] & \mathcal{E} \ar[d, "\nabla"] \\
						F_{X/S}^*\FF \otimes \Omega_{X/S}^1\ar[r, "\phi \otimes \id_{\Omega_{X/S}^1}"] & \mathcal{E} \otimes \Omega_{X/S}^1 
				\end{tikzcd}
		\end{center}
		commutes. 

		Writing \(F_{X/S}^*\FF = F_{X/S}^{-1}\FF \otimes_{F_{X/S}^{-1}\OO_{X^{(1)}}} \OO_X,\) we have
		\[\nabla'(s \otimes f) = s' \otimes df.\]

		By adjunction, a morphism \(\phi : F_{X/S}^*\FF \to \mathcal{E}\) is equivalent to giving a morphism \(\psi : \FF \to F_{X/S,*}\mathcal{E},\) and \(\phi, \psi\) are related by the formula
		\[\phi(s \otimes f) = f\psi(s).\] 
		Thus asking that \(\phi\) is compatible with the connections is equivalent to demanding that
		\[f\nabla(\psi(s)) + \psi(s) \otimes df = \psi(s) \otimes df\]
		for all \(f \in \OO_X, s \in \FF.\) Of course, this is equivalent to demanding that \(\nabla(\psi(s)) = 0\) for all \(s,\) or equivalently that \(\psi\) factors through \(\Sol(\mathcal{E}, \nabla).\) 
\end{proof}	
\begin{remark} \label{rem808} 
		As \autoref{prop813} tells us
		\[\nu_{X/S,\wedge}^* : \mathcal{D}_{\qc}(X^{(1)}) \to \mathcal{D}_{\qc}((X/S)^{\hat{\dR}})\]
		is fully faithful, it is natural to ask for its essential image. This is answered by Cartier descent, the subject of our \autoref{cartiersec}.
\end{remark}

\section{On \(p\)-curvature of relative de Rham stack} \label{pcurvsec} 

In \autoref{rem808}, we noted that
\[\nu_{X/S,\wedge}^* : \mathcal{D}_{\qc}(X^{(1)}) \to \mathcal{D}_{\qc}((X/S)^{\hat{\dR}})\]
was fully faithful, and asked for its essential image. We now describe this essential image via \emph{\(p\)-curvature}. We start by defining \(p\)-curvature and introducing some basic properties, and then prove Cartier descent, which asserts that the essential image of \(\nu_{X/S,\wedge}^*\) is precisely those \(\mathcal{D}_{X/S}\)-modules with \(p\)-curvature zero. 

\subsection{Definition of \(p\)-curvature} 

Fix \(f : X\to S\) a smooth, representable morphism of \(\F_p\)-stacks. Using \autoref{gerbe}, we can define a certain \emph{linear} invariant of objects of \(\mathcal{D}_{\qc}((X/S)^{\hat{\dR}}).\) 

\subsubsection{The construction of \(p\)-curvature}

Pullback along the relative Frobenius \(F_{X/S}\) gives a \(\otimes\)-functor
\begin{equation} \label{eq164} \mathcal{D}_{\qc}((X/S)^{\hat{\dR}}) \to \mathcal{D}_{\qc}(F_{X/S}^*(X/S)^{\hat{\dR}}) \heq \mathcal{D}_{\qc}(B\mathbb{V}(F_X^*\mathcal{T}_{X/S})^{\#,\wedge}) \heq \mathcal{D}_{\qc}(\mathbb{V}(F_X^*\Omega^1_{X/S})),\end{equation}
where the first equivalence above is from \autoref{cansplit}, and the second equivalence is Cartier duality.
\begin{warn}
		For the resulting composition to be symmetric monoidal, we need to ensure \(\mathcal{D}_{\qc}(\mathbb{V}(F_X^*\Omega^1_{X/S}))\) is given its \emph{convolution} symmetric monoidal structure, \emph{not} its usual symmetric monoidal structure.
\end{warn}

\begin{notn} \label{pcurv}
		We let
		\[\psi_p : \mathcal{D}_{\qc}((X/S)^{\hat{\dR}}) \to \mathcal{D}_{\qc}(\mathbb{V}(F_X^*\Omega^1_{X/S}))\]
		denote the composition \autoref{eq164}.
\end{notn} 

It is often useful to think of a quasicoherent sheaf on \(\mathbb{V}(F_X^*\Omega^1_{X/S})\) as being the data of a quasicoherent sheaf \(\FF\) on \(X\) together with a linear map \(\theta : \FF \to \FF \otimes F_X^*\Omega^1_{X/S}\) obeying \(\theta \wedge \theta = 0.\) From this perspective, it is useful to ask: given a quasicoherent \(\mathcal{D}_{X/S}\)-module \(\mathcal{E},\) what is the quasicoherent sheaf on \(X\) associated to \(\psi_p(\mathcal{E})\)? The following diagram explains it is just \(\mathcal{E}\) itself. 
\begin{prop} \label{itcommutes}
		The triangle
						\begin{center}
								\begin{tikzcd}
										\mathcal{D}_{\qc}((X/S)^{\hat{\dR}}) \ar[rr,"\psi_p"] \ar[dr] && \mathcal{D}_{\qc}(\mathbb{V}(F_X^*\mathbb{L}_{X/S})) \ar[dl] \\
																									  & \mathcal{D}_{\qc}(X) 
								\end{tikzcd}
						\end{center}
				is commutative, where the left morphism is \(\ast\)-pullback along \(X \to (X/S)^{\hat{\dR}},\) and the right morphism is restriction of scalars (that is, pushforward along the affine morphism \(\mathbb{V}(F_X^*\mathbb{L}_{X/S}) \to X\)).
\end{prop}
\begin{proof}
		Recall that \(\nu_{X/S, \wedge} \circ \pi_{X/S, \wedge} = F_{X/S}.\) The pair of maps \(\pi_{X/S, \wedge} : X \to (X/S)^{\hat{\dR}}\) and \(\id : X \to X\) therefore induce a map
		\[\alpha : X \to F_{X/S}^*(X/S)^{\hat{\dR}},\]
		where by \(F_{X/S}^*\) we mean the base change 
		\begin{center}
				\begin{tikzcd}
						F_{X/S}^*(X/S)^{\hat{\dR}} \ar[r] \ar[d] & (X/S)^{\hat{\dR}} \ar[d, "\nu_{X/S, \wedge}"] \\
						X \ar[r, "F_{X/S}"] & X^{(1)}. 
				\end{tikzcd}
		\end{center}

		Thus we have a commutative triangle 
		\begin{center}
								\begin{tikzcd}
										\mathcal{D}_{\qc}((X/S)^{\hat{\dR}}) \ar[rr,"F_{X/S}^*"] \ar[dr, "\pi^*"] && \mathcal{D}_{\qc}(F_{X/S}^*(X/S)^{\hat{\dR}}) \ar[dl, "\alpha^*"] \\
																									  & \mathcal{D}_{\qc}(X) 
								\end{tikzcd}
				\end{center}
		The Cartier duality equivalence
		\[\mathcal{D}_{\qc}(F_{X/S}^*(X/S)^{\hat{\dR}}) \heq \mathcal{D}_{\qc}(\mathbb{V}(F_X^*\mathbb{L}_{X/S}))\]
		is well-known to be compatible with \(\ast\)-pullback to \(X\) on the left with restriction of scalars on the right; thus the above commutative triangle implies the triangle of the proposition statement commutes.
\end{proof} 

\subsubsection{Treating \(p\)-curvature as a morphism} \label{pcurvmorphism}

Let \((\mathcal{E}, \nabla) \in \mathcal{D}_{\qc}((X/S)^{\hat{\dR}}).\) Then the \(p\)-curvature of \((\mathcal{E}, \nabla)\) can be viewed as a morphism
\[\psi_p : \mathcal{E} \to \mathcal{E} \otimes F_{X/S}^*\Omega_{X^{(1)}/S}^1\]
of quasicoherenet sheaves on \(X.\) The functor \(\pi_{X/S,\wedge}^* : \mathcal{D}_{\qc}((X/S)^{\hat{\dR}}) \to \mathcal{D}_{\qc}(X)\) is faithful (though not full); so, we may ask: does \(\psi_p\) come from a morphism 
\[(\mathcal{E}, \nabla) \to (\mathcal{E}, \nabla) \otimes \nu_{X/S,\wedge}^*\Omega_{X^{(1)}/S}^1\]
on \((X/S)^{\hat{\dR}}\)? And the answer is fortunately yes.

\begin{remark}
		The author first learned this from Ogus' \cite{ogus}. We will give a slightly different proof of this fact, communicated to the author by Longke Tang and Qixiang Wang. Tang-Wang's argument is more in line with the techniques of this paper than Ogus', so we include this proof despite the result already being known. The author thanks Tang and Wang for finding this argument and for allowing us to include it in this paper.
\end{remark}

The relative Frobenius \(F_{X/S} : X \to X^{(1)}\) factors through \((X/S)^{\hat{\dR}}.\) The pullback of the gerbe \((X/S)^{\hat{\dR}} \to X^{(1)}\) to \((X/S)^{\hat{\dR}}\) is split (as a gerbe pulled back to its total space is already split); hence we get a commutative diagram 
\begin{center}
		\begin{tikzcd}
				B\mathbb{V}(F^*\mathcal{T}_{X^{(1)}/S})^{\#, \wedge} \ar[r] \ar[d] & B\mathbb{V}(\nu^*\mathcal{T}_{X^{(1)}/S})^{\#, \wedge}) \ar[r] \ar[d] & (X/S)^{\hat{\dR}} \ar[d, "\nu"] \\
				X \ar[r, "\pi"] & (X/S)^{\hat{\dR}} \ar[r, "\nu"] & X^{(1)}
		\end{tikzcd}
\end{center}
in which all squares are Cartesian.

Pulling back \((\mathcal{E}, \nabla)\) to \(B\mathbb{V}(\nu^*\mathcal{T}_{X^{(1)}/S})^{\#, \wedge})\) and applying Cartier duality, we get a map
\[(\mathcal{E}, \nabla) \to (\mathcal{E}, \nabla) \otimes \nu^*\Omega^1_{X^{(1)}/S}.\]
This morphism will pullback along \(\pi^*\) to the usual \(p\)-curvature \(\mathcal{E} \to \mathcal{E} \otimes F^*\Omega^1_{X^{(1)}/S}\) because of compatibility of Cartier duality with base change. 

\begin{notn}
		Given \(\mathcal{E}^{\hat{\dR}} \in \mathcal{D}_{\qc}((X/S)^{\hat{\dR}}),\) we will write
		\[\psi_p(\mathcal{E}^{\hat{\dR}}) : \mathcal{E}^{\hat{\dR}} \to \mathcal{E}^{\hat{\dR}} \otimes \nu_{X/S, \wedge}^*\Omega_{X^{(1)}/S}^1\]
		for the morphism constructed in this section, uniquely characterized by the property that \(\pi_{X/S,\wedge}^*(\psi_p(\mathcal{E}^{\hat{\dR}}))\) is the \(p\)-curvature \(\mathcal{E} \to \mathcal{E} \otimes F_{X/S}^*\Omega^1_{X^{(1)}/S}.\) 
\end{notn}

\subsection{Relative Cartier descent} \label{cartiersec} 

Fix \(f : X \to S\) a smooth, representable morphism of algebraic \(\F_p\)-stacks. Assume \(S\) admits an fppf cover \(\mathcal{U} \to S,\) where \(\mathcal{U}\) is a scheme. 

In this section, we will prove a relative version of Cartier descent, finally allowing us to characterize the essential image of \(\nu_{X/S,\wedge}^*\) in terms of \(p\)-curvature. 
\begin{construction} \label{const845}
		Let \((\mathcal{E}, \nabla) \in \mathcal{D}_{\qc}((X/S)^{\hat{\dR}}).\) Then the \((\nu^*, \Sol)\) adjunction gives us a counit map
		\[\epsilon : \nu^*\Sol(\mathcal{E}, \nabla) \to (\mathcal{E}, \nabla).\]
		Applying \(\pi^*,\) we get a map
		\begin{equation} \label{eq849} F_{X/S}^*\Sol(\mathcal{E}, \nabla) \to \mathcal{E}.\end{equation} 
\end{construction}	

\begin{theorem}[Relative Cartier descent] \label{cartierdescent}
		The functor
		\[\nu_{X/S, \wedge}^* : \mathcal{D}_{\qc}(X^{(1)}) \to \mathcal{D}_{\qc}((X/S)^{\hat{\dR}})\]
		is fully faithful. Moreover, for an object \((\mathcal{E}, \nabla) \in \mathcal{D}_{\qc}((X/S)^{\hat{\dR}}),\) the following are equivalent:
		\begin{enumerate}
				\item \((\mathcal{E}, \nabla)\) lies in the essential image of \(\nu_{X/S, \wedge}^*,\)
				\item the map \(F_{X/S}^*\Sol(\mathcal{E}, \nabla) \to \mathcal{E}\) constructed in \autoref{const845} is an equivalence,
				\item the \(p\)-curvature morphism \(\psi_p(\mathcal{E}, \nabla) : \mathcal{E} \to \mathcal{E} \otimes \nu_{X/S}^*\Omega_{X^{(1)}/S}^1\) is the zero morphism.
		\end{enumerate} 
\end{theorem}
\begin{proof}
		\autoref{prop813} already tells us about the fully faithfulness, so we just need to identify its essential image. 

		For an object \((\mathcal{E}, \nabla) \in \mathcal{D}_{\qc}((X/S)^{\hat{\dR}}),\) we have that \((\mathcal{E}, \nabla)\) belongs to the essential image of \(\nu_{X/S,\wedge}^*\) if and only if the co-unit 
		\[\epsilon : \nu_{X/S,\wedge}^*(\Sol_{X/S}(\mathcal{E}, \nabla)) \to (\mathcal{E}, \nabla)\]
		is an equivalence. 

		As \(\pi : X \to (X/S)^{\hat{\dR}}\) is a flat cover, by flat descent we have that \(\epsilon\) is an equivalence if and only if \(\pi^*(\epsilon)\) is an equivalence. But \(\pi^*(\epsilon)\) is exactly the morphism \autoref{eq849}, and so conditions 1 and 2 are equivalent. 

		We now show condition 1 implies condition 3; that is, we check every \(\nu^*\FF\) has \(p\)-curvature zero. Recall that \(p\)-curvature is defined using the Cartesian diagram
		\begin{equation} \label{diagram861}
				\begin{tikzcd}
						B\mathbb{V}(F_X^*\mathcal{T}_{X/S}) \ar[r, "F'"] \ar[d, "\nu'"] & (X/S)^{\hat{\dR}} \ar[d, "\nu"] \\
						X \ar[r, "F"] & X^{(1)}.
				\end{tikzcd}
		\end{equation}
		It follows from the commutativity of \autoref{diagram861} that
		\begin{equation} \label{eq880} F^{',*}(\nu^*\FF) = \nu^{',*}(F^*\FF).\end{equation}
		Using Cartier duality, we can describe objects of \(\mathcal{D}_{\qc}(B\mathbb{V}(F_X^*\mathcal{T}_{X/S})) \heq \mathcal{D}_{\qc}(\mathbb{V}(F_X^*\Omega_{X/S}^1))\) as pairs \((\mathcal{E}', \theta : \mathcal{E}' \to \mathcal{E}' \otimes F_X^*\Omega_{X/S}^1)\) where \(\theta \wedge \theta = 0.\) 

		In this description, we have
		\[F^{',*}(\nu^*\FF) = (F^*\FF, \pi^*(\psi_p(\nu^*\FF)))\]
		(by definition of the \(p\)-curvature morphism) and
		\[\nu^{',*}(F^*\FF) = (F^*\FF, 0).\]
		Thus \autoref{eq880} tells us \(\pi^*(\psi_p(\nu^*\FF)) = 0,\) which since \(\pi\) is a flat cover implies \(\psi_p(\nu^*\FF) = 0.\) In particular, condition 1 implies condition 3. 

		Finally, we prove condition 3 implies condition 2; this is the hardest part of the argument. Assume condition 3 holds; we will show \(\epsilon\) is an equivalence. As being an equivalence can be checked on cohomology sheaves, it suffices to prove this assuming that \((\mathcal{E}, \nabla) \in \mathcal{D}_{\qc}((X/S)^{\hat{\dR}})^{\heart}.\) 

		Moreover, as being an equivalence can be checked after flat descent, we can assume \(S = \Spec k\) is an affine scheme and \(X = \A^d_k.\) In this case, classical Cartier descent (see Theorem 5.1 of \cite{katz}) shows that our \(\psi_p = 0\) assumption implies condition 2 holds, as desired. 
\end{proof}

\section{Hitchin-type morphisms} \label{hitchinsec}
We now turn our attention towards non-abelian Hodge theory. To state our result, we will need some Hitchin morphisms; we now give some general setup in which one can define Hitchin morphisms, since we will use it twice below. 

\subsection{Generalities} \label{hitchingeneralities}

Fix\(k\) be a field of characteristic \(p\) and \(X\) an algebraic \(k\)-stack with \(X \to \Spec(k)\) proper, with \(\omega\) a line bundle on \(X.\) In this situation, we can define a \emph{Hitchin morphism} as follows. 

Let \(L = \SPEC_X \Sym^{\bullet}(\omega^{\vee})\) denote the total space of \(\omega.\) Write \(\pi : L \to X\) for the canonical map. 

\begin{notn}
		We say a twisted linear endomorphism \(\phi : \mathcal{E} \to \mathcal{E} \otimes \omega\) is \emph{integrable} if the map
		\[\phi \wedge \phi : \mathcal{E} \xto{\phi} \mathcal{E} \otimes \omega \xto{\phi \otimes \id} \mathcal{E} \otimes \omega^2 \to \mathcal{E} \otimes \wedge^2(\omega)\]
		vanishes. 
\end{notn}	

\subsubsection{The tautological section}

There is a natural inclusion
\[\omega^{\vee} \to \Sym^{\bullet}(\omega^{\vee}) = \pi_*\OO_L,\]
which by adjunction gives us a map
\[\pi^*(\omega^{\vee}) \to \OO_L\]
of quasicoherent sheaves on \(L.\) The dual of this map is a map \(\OO_L \to \pi^*(\omega),\) or in other words a global section \(\lambda\) of \(\pi^*(\omega).\) 

\begin{defn}
		The \emph{tautological section} of \(\pi^*(\omega)\) is the section \(\lambda\) just described.
\end{defn}

\subsubsection{The characteristic polynomial of an \(\omega\)-twisted endomorphism}

\begin{defn}
		We let \(\mathcal{M}(X, n, \omega)\) denote the \(k\)-stack whose groupoid of \(T\)-points is the groupoid of pairs \((E, \theta),\) where \(E\) is a rank \(n\) vector bundle on \(X\times_k T\) and \(\theta : E \to E \otimes p_X^*(\omega)\) is a linear map, for \(p_X : X\times_k T \to X\) the projection.
\end{defn}
\begin{notn}
		For \(T \to \Spec(k)\) a \(k\)-scheme, we will write \(\omega_T := p_X^*(\omega).\)
\end{notn}

\begin{remark} \label{rem87} 
		In a moment, we will use that there is a natural inclusion
		\begin{equation} \label{eq89} H^0(X, \omega) \inclusion H^0(L, \pi^*(\omega)).\end{equation} 

		Indeed, adjunction gives us a map \(\omega \to \pi_*\pi^*\omega,\) and taking global sections of this morphism gives us exactly \autoref{eq89}. Since \(\pi\) is faithfully flat, the morphism \(\omega \to \pi_*\pi^*\omega\) is always injective, and hence \autoref{eq89} is injective. 
\end{remark}

\begin{construction}[Characteristic polynomials] \label{charpolys}
		Consider a \(k\)-scheme \(T,\) and a \(T\)-point \((E, \theta)\) of \(\mathcal{M}(X, n,\omega).\) We will now define the characteristic polynomial of \(\theta.\) 

		Let \(\pi_T : L_T \to X_T\) be the canonical map. Then \(\pi_T^*(E)\) is a rank \(n\) vector bundle on \(L_T,\) carrying a twisted-linear endomorphism
		\[\pi_T^*(\theta) : \pi_T^*(E) \to \pi_T^*(E) \otimes \pi_T^*(\omega_T).\]
		Set
		\[p(\lambda) := \det(\lambda \cdot \id_{\pi_T^*(E)} - \pi_T^*(\theta)) \in \pi_T^*(\omega_T^{\otimes n}).\] 

		This will be a monic polynomial in \(\lambda,\) of the form
		\[p(\lambda) = \lambda^n + a_1\lambda^{n-1} + \cdots + a_n,\]
		where apriori we have \(a_i \in H^0(L_T, \pi_T^*(\omega_T^{\otimes i})).\) However (see \autoref{rem87}) there is a natural injection
		\[H^0(X_T, \omega_T^{\otimes i}) \inclusion H^0(L_T, \pi_T^*(\omega_T^{\otimes i})),\] 
		and we have a miracle: the \(a_i\) actually belong to this subspace \(H^0(X_T, \omega_T^{\otimes i})\) subspace. This claim can be checked using local coordinates, where it follows immediately after choosing a chart in which \(\omega\) is the trivial line bundle.
\end{construction}

\begin{prop} \label{cayleyhamilton}
		Let \((E, \theta)\) be a \(T\)-point of \(\mathcal{M}(X, n, \omega),\) with characteristic polynomial
		\[p(\lambda) = \lambda^n + a_1\lambda^{n-1} + \cdots + a_n.\]
		Set \(Y_a \subseteq L_T\) the closed subscheme cut out by \(p(\lambda) = 0.\) The integrability condition on \(\theta\) allows us to descend \((E, \theta)\) to a coherent sheaf \(E'\) on \(L_T\); then \(E'\) is the pushforward of some coherent sheaf on \(Y_a.\) 
\end{prop}	
\begin{proof}
		Scheme-theoretic support can be checked locally, where it follows form the Cayley-Hamilton theorem.	
\end{proof}

The closed subschemes \(Y_a\) arising in \autoref{cayleyhamilton} are called \emph{spectral varieties}. 
\begin{defn} \label{artdefn684}
		A \emph{\(T\)-family of spectral spectral varieties} for \((X, \omega)\) is a closed subscheme of \(\mathbb{V}(\omega_T)\) defined by
		\[\lambda^n + a_{n-1}\lambda^{n-1} + \cdots + a_0 = 0,\]
		where \(\lambda\) is the tautological 1-form on \(\mathbb{V}(\omega_T)\), \(a_i \in H^0(X_T, \omega_T^{\otimes i}),\) and we form the above sum in \(\pi_T^*(\omega_T^{\otimes n}),\) thought of as a line bundle on \(\mathbb{V}(\omega_T).\) 
\end{defn}

\subsubsection{Finally, the Hitchin morphism} 

\begin{defn}
		The \emph{Hitchin base} \(\mathcal{H}(X, n, \omega)\) is the \(k\)-space whose set of \(T\)-points is the \(k\)-vector space
		\[\bigoplus_{i=1}^n H^0(X_T, \omega_T^{\otimes i}),\]
		where \(\omega_T = p_X^*(\omega)\) for \(p_X : X_T\to T\) the projection.

		In particular, \(\mathcal{H}(X, n, \omega)\) is represented by the \(k\)-scheme \(\bigoplus_{i=1}^n \Spec \Sym^{\bullet}_k(H^0(X, \omega^{\otimes i})^{\vee}).\) 
\end{defn}

We then define the \emph{Hitchin morphism} to be the map of \(k\)-stacks
\[h : \mathcal{M}(X, n, \omega) \to \mathcal{H}(X, n, \omega)\]
which sends a \(T\)-family of Higgs fields to the coefficients of the corresponding characteristic polynomial (see \autoref{charpolys}). 

\subsubsection{The BNR correspondence}

We now prove a variant of the Beauville--Narasimhan--Ramanan (BNR) correspondence for stacky curves, which we will use in our study of logarithmic non-abelian Hodge theory. Beauville--Narasimhan--Ramanan \cite{bnr} proved the first version of the BNR correspondence for \emph{smooth} spectral curves, and Schaub \cite{schaub} later generalized their work to arbitrary special curves. Our treatment follows Schaub \cite{schaub}. 

So, we now will now specialize to the case where \(X\) is a multi-root stack over a smooth connected curve \(C.\) We let \(\rho : X \to C\) be the map to the course moduli space. 

We now make a definition of torsion-free rank 1 sheaf which applies to spectral curves of such an \(X.\) 
\begin{defn}
		Let \(\tilde{X}\) be a multi-root stack over a curve \(\tilde{C},\) with \(\tilde{\rho} : \tilde{X} \to \tilde{C}\) the corresponding coarse moduli. We say that \(\FF\in\Qcoh(\tilde{X})\) is \emph{torsion-free of rank 1} if \(\GG := \tilde{\rho}_*\FF\) is a coherent sheaf on \(\tilde{C}\) which is torsion-free of rank 1 in the sense of Schaub \cite{schaub}; we recall that to Schaub, \(\GG\) is 
		\begin{enumerate}
				\item \emph{torsion-free} if the map \(\GG \to \GG \otimes Q(\OO_{\tilde{C}})\) (for \(Q(\OO_{\tilde{C}})\) the total ring of fractions of \(\OO_{\tilde{C}}\)) is injective, and
				\item \emph{rank 1} if for every generic point \(\eta_i\) of \(\tilde{C},\) we have \(\ell(\GG_{\eta_i}) = \ell(\OO_{\tilde{C}, \eta_i}),\) where \(\ell(-)\) is is used to indicate the length of an \(\OO_{\tilde{C}, \eta_i}\)-module. 
		\end{enumerate}
\end{defn}
\begin{exmp} \label{exmp1305} 
		If \(\tilde{C} = \Spec k[x, t^2],\) then there are two torsion-free coherent sheaves of rank 1: 
		\begin{enumerate}
				\item the sheaf \(\OO_{\tilde{C}}\) itself; 
				\item the coherent sheaf associated to the \(k[x, t^2]\)-module \(M := k[x] \oplus k[x],\) where \(t\) acts as \(t \cdot (f(x), g(x)) = (0, f(x)).\) 
		\end{enumerate}

		It's clear that \(\OO_{\tilde{C}}\) is torsion-free of rank 1. For the second example, the torsion-freeness is clear, and the rank 1 condition is because \(\tilde{C}\) has a unique generic point \(\eta\) corresponding to the minimal prime \((t),\) and 
		\[\ell(\OO_{\tilde{C}, \eta}) = \ell((k[x, t]/(t^2))_{(t)}) = 2.\]

		As \(M_{(t)}\) is a 2-dimensional \(k(x)\)-vector space, we deduce it has length 2 as well, and hence we get our rank 1 condition. 
\end{exmp}

\begin{prop}[The BNR correspondence] \label{bnr}
		Assume \(X\) is a multi-root stack over a smooth connected curve \(C\). 

		Let \(T\) be a \(k\)-scheme, and fix a \(T\)-point point \(a \in \mathcal{H}(X, n, \omega)(T)\). Let \(Y_a \subseteq \mathbb{V}(\omega_T)\) denote the spectral curve corresponding to \(a.\) 

		Set \(\GG(T)\) the groupoid of quasicoherent sheaves on \(Y_a.\) Define a map
		\begin{equation} \label{eq1033} \mathcal{M}(X, n, \omega)(T) \to \GG(T)\end{equation}
		of groupoids by sending \((\mathcal{E}, \theta)\) to the corresponding coherent sheaf on \(Y_a.\)

		Then the map \autoref{eq1033} is a fully faithful. Moreover, for \(\pi : Y_a \inclusion \mathbb{V}(\omega_T) \to X_T\) the canonical map, the essential image of \autoref{eq1033} consists exactly of those \(\FF \in \GG(T)\) which are torsion free of rank 1. 
\end{prop}
\begin{proof}
		The full faithfulness is immediate from the fact that the category of effective descent data for \(\mathbb{V}(\omega_T) \to X_T\) emebds fully faithfully into the category of quasicoherent sheaves on \(\mathbb{V}(\omega_T),\) and so it suffices to check the claim about the essential image. 

		Now, we characterize the essential image. Suppose \(\FF \in \GG(T)\) is torsion free of rank 1. Then, as \(\pi\) is finite flat of rank \(n,\) we deduce that \(\mathcal{E} := \pi_*\FF\) is a rank \(n\) vector bundle \(\mathcal{E}\) on \(X_T,\) using that \(X_T\) is just a multi-root stack over the corresponding spectral curve for \(C\) and applying Schaub \cite{schaub} Lemma 2.2 (in the same way Schaub \cite{schaub} applies it to deduce his proposition 2.1). 

		Moreover, \(\mathcal{E}\) has a natural Higgs field structure: just let \(\theta : \mathcal{E} \to \mathcal{E} \otimes \omega_T\) be the morphism obtained by remembering that \(\FF\) has a natural action of \(\Sym^{\bullet}_{\OO_{X_T}}(\omega_T^{\vee}),\) which will automatically be integrable. Moreover, because \(\FF\) is a \emph{torsion free} (as it is a line bundle!) on \(Y_a,\) we deduce that \(\theta\) has characteristic polynomial \(a.\) We remark that without this torsion free condition, we could only deduce that \(\theta\) obeyed \(a,\) not necessarily that the characteristic polynomial of \(\theta\) was \(a.\) In particular, \(\FF\) is the image of \((\mathcal{E}, \theta)\) under \autoref{eq1033}. The converse, that every object of the essential image if torsion free of rank 1, is similar, so we conclude.
\end{proof}
\begin{exmp}
		Under the BNR correspondence, the two sheaves of \autoref{exmp1305} correspond to the two possible isomorphism classes of a nilpotent endomorphism of the rank 2 vector bundle on \(\A^1\): the zero map, or the map \(\begin{pmatrix} 0 & 1 \\ 0 & 0 \end{pmatrix}.\)
\end{exmp}

\subsubsection{Functoriality of Hitchin morphisms}

When comparing our logarithmic non-abelian Hodge theorem to de Catalado--Zhang's \cite{dCZ}, it will be useful to have the following result. 

\begin{construction}[Functoriality of Hitchin morphisms] \label{hitchinfunctoriality}
		Let \(f : Y \to X\) be a morphism of \(\F_p\)-schemes, and \(\omega_X\) a line bundle on \(X.\) Then \(f\) induces a natural commutative diagram
		\begin{center}
				\begin{tikzcd}
						\mathcal{M}(X, n, \omega_X) \ar[r, "\mathcal{M}(f)"] \ar[d] & \mathcal{M}(Y, n, f^*\omega_X) \ar[d, "h_Y"] \\
						\mathcal{H}(X, n, \omega_X) \ar[r, "\mathcal{H}(f)"] & \mathcal{H}(Y, n, f^*\omega_X),
				\end{tikzcd}
		\end{center}
		where \(\mathcal{M}(f)\) is the map sending a \(T\)-point \((\mathcal{E}_T, \theta_T) \in \mathcal{M}(X, n, \omega_X)(T)\) to \((f^*\mathcal{E}_T, f^*\theta_T) \in \mathcal{M}(Y, n, f^*\omega_X)(T),\) and \(\mathcal{H}(f)\) is the morphism of Hitchin bases induced by the natural maps 
		\[H^0(X, \omega_X^{\otimes i}) \to H^0(Y, f^*\omega_X^{\otimes i})\]
		coming from taking global sections of the adjunction map
		\[\omega_X^{\otimes i} \to f_*f^*\omega_X^{\otimes i}.\] 
\end{construction}

\subsection{On the characteristic polynomial of \(p\)-curvature} \label{charpolypsip}

We now explain how to treat the characteristic polynomial of \(p\)-curvature. 
\begin{remark}
		Using de Rham stacks give us some advantage here; indeed, in the usual story of characteristic polynomials of \(p\)-curvature (see Laszlo-Pauly \cite{laszlopauly}, and also Zhang \cite{siqingfix}), one views the \(p\)-curvature as a \(F^*\Omega^1_{X/S}\)-twisted linear endomorphism on \(X\); this makes it non-obvious that the Hitchin base for \(p\)-curvature is correct, and indeed one has to prove a theorem constraining the possible characteristic polynomials of \(p\)-curvature to get the Hitchin bases to align. In our de Rham stack approach, this constraint just comes from the fact that the \(p\)-curvature actually lives on \(X^{\dR}.\) 
\end{remark}

Fix \(f : X \to S\) smooth, representable of relative dimension 1. Set \(\omega := \nu_{X/S,\wedge}^*\Omega^1_{X^{(1)}/S}.\) By \autoref{prop813}, the unit 
\[\left(\Omega^1_{X^{(1)}/S}\right)^{\otimes i} \to \Sol(\omega^{\otimes i})\]
of the \((\nu^*, \Sol)\) adjunction is an equivalence for every \(i\geq 0\). Taking global sections, we deduce that
\[H^0(X^{(1)}, \omega_{X^{(1)}/S}^{\otimes i}) \heq H^0(X^{(1)}, \Sol(\omega^{\otimes i})).\]

But by adjunction,
\begin{align*}
		H^0(X^{(1)}, \Sol(\omega^{\otimes i})) &= \Hom(\OO_{X^{(1)}}, \Sol(\omega^{\otimes i})) \\
											   &= \Hom((\OO_X, d), \omega^{\otimes i}) \\
											   &= H^0((X/S)^{\hat{\dR}}, \omega^{\otimes i}).
\end{align*}

Therefore 
\begin{equation} \label{eq557} H^0((X/S)^{\hat{\dR}}, \omega^{\otimes i}) = H^0(X^{(1)}, (\Omega^1_{X^{(1)}/S})^{\otimes i}).\end{equation} 

Thus the Hitchin base for \(\omega\) on \((X/S)^{\hat{\dR}}\) is exactly the same as the Hitchin base for \(\Omega^1_{X^{(1)}/S}\) on \(X^{(1)}.\)

\section{Logarithmic de Rham stacks} \label{logdrsec} 

We now introduce \emph{logarithmic} de Rham stacks; the main reason we took care to develop the theory of de Rham stacks over a general base was so that we could define logarithmic de Rham stacks by working relatively to \((\A^1/\G_m)^n.\) We start by explaining Olsson's \cite{olsson} (very nice!) observation that one can do logarithmic geometry by working relatively to \((\A^1/\G_m)^n,\) and then define logarithmic de Rham stacks.

Throughout this section, we fix a base field \(k\) of characteristic \(p.\) 

\begin{notn}
		Throughout this section, put \(\Theta := \A^1/\G_{m},\) thought of as a \(k\)-stack (so here \(\A^1\) means \(\Spec k[t],\) and \(\G_m\) means \(\Spec k[x, x^{-1}]\)). 
\end{notn}

\subsection{Recollections on Olsson's approach to logarithmic geometry} 

\begin{warn}
		None of the results in this section are original, and all are essentially due to Olsson \cite{olsson}; we recall them here merely for the reader's convenience. In preparing this section, the author also found the works of Borne \cite{borne2006}, \cite{borne2007}, Esnault--Groechenig \cite{eg} and Laaroussi \cite{laaroussi} very helpful. 
\end{warn}

Let \(X\) be a smooth \(k\)-scheme, and \(D\) a strict normal crossings divisor on \(X.\) Olsson \cite{olsson} observed that the logarithmic Kahler differentials \(\Omega^1_{X/k}(\log D)\) could be written as the Kahler differentials of \(X\) relative to an appropriate power of \(\Theta.\) We now explain this perspective; first, we recall how a divisor can be encoded via a map to \(\Theta.\) 
\begin{prop} \label{prop855} 
		Let \(X\) be a \(k\)-scheme. Then morphisms \(X \to \A^1/\G_m\) are in bijection with pairs \((\LL, i : \LL \to \OO_X)\) where \(\LL\) is a line bundle on \(X\) and \(i\) is a morphism of line bundles.
\end{prop}
\begin{proof}
		See \cite{olsson3}, example 5.13.
\end{proof}

We can now make precise our earlier statement about logarithmic Kahler differentials. 
\begin{theorem}[Olsson] \label{olsson875} 
		Let \(X\) be a smooth \(k\)-scheme, and \(D\) a simple normal crossings divisor on \(X,\) with \(n\) components \(D_1, ..., D_n.\) Let \(c_{D_i} : X \to \Theta\) denote the morphism associated to \(\OO_X(-D_i) \inclusion \OO_X.\) Put \(c_D : X \to \Theta^n\) the product of the \(c_{D_i}.\) 

		Then \(c_D\) is a smooth, representable morphism, and the relative Kahler differentials \(\Omega^1_{X/\Theta^n}\) are canonically isomorphic to the logarithmic Kahler differentials \(\Omega^1_{X/k}(\log D).\) Moreover, \(\Omega^1_{D_i/B\G_m} = \phi_i^*\Omega^1_X(\log D),\) for \(\phi_i : D_i \inclusion X\) the inclusion. 
\end{theorem}
\begin{proof}
		We include this statement and proof for the reader's convenience, but the ideas are all due to Olsson \cite{olsson3}, \cite{olsson}, and we claim no originality here. 

		To start we explain why \(c_D\) is representable; for this it suffices to prove each \(c_{D_i}\) is representable. 

		The fiber product \(\A^1 \times_{\A^1/\G_m} X\) is just the frame bundle of the line bundle \(\OO_X(-D_i),\) and in particular is a scheme. As representability by \emph{algebraic spaces} can be checked etale locally (see \cite{stacks}, tag 04ZP), we deduce that \(c_{D_i} : X \to \Theta\) is representable by algebraic spaces. Checking an algebraic space is a scheme can be done Zariski locally, so we can now assume \(X\) is affine. Then \(c_D : X \to \A^1/\G_m\) is an affine morphism of stacks (because affineness can be checked after a flat cover, and the pullback \(X \times_{\A^1/\G_m} \A^1\) is just the frame bundle of the line bundle \(\OO_X(-D),\) which will be affine whenever \(X\) is affine). It follows that, when \(X\) is affine, \(c_D\) is affine; hence \(c_D\) is representable by algebraic spaces and affine, implying it is representable by schemes. 

		Thus \(c_D\) is a smooth, representable morphism. Now, we compute its relative Kahler differentials. To start, we recall that there is a natural map
		\[d : \OO_X \to \Omega^1_{X/k} \inclusion \Omega^1_{X/k}(\log D)\]
		of sheaves on \(X.\) 

		We claim that this \(d\) is actually a derivation relative to \(\Theta^n,\) and that the induced map
		\[\Omega_{X/\Theta^n} \to \Omega^1_{X/k}(\log D)\]
		is an isomorphism of quasicoherent sheaves on \(X.\) 

		Both the property of being a derivation, and the property of a morphism being an equivalence, can be checked locally. Thus we can assume \(X = \A^{n+m}\) and \(c_D : X \to \Theta\) is the map classifying the divisor \(x_1\cdots x_n = 0\) in \(\Spec k[x_1, ..., x_n, y_1, ..., y_m].\) In this case, we have a Cartesian diagram
		\begin{center}
				\begin{tikzcd}
						\A^{n+m} \times \Spec k[t_1^{\pm 1}, ..., t_n^{\pm 1}] \ar[r] \ar[d, "c_D'"] & \A^{n+m} \ar[d, "c_D"] \\
						\A^n \ar[r] & \Theta^n,
				\end{tikzcd}
		\end{center}
		and \(c_D'\) has relative differentials \(\Omega^1_{c_{D'}}\) generated by the symbols \(dx_i, dy_i, dt_i^{\pm 1},\) together with the relations
		\[t_i dx_i + x_i dt_i = 0.\]

		In particular, \(t_i^{-1} dt_i\) behaves as \(dx_i/x_i\); hence \(\Omega^1_{X/k}(\log D)\) pulls back to \(\Omega^1_{c_{D'}}.\) As \(\Omega^1_{X/\Theta^n}\) also pulls back to \(\Omega^1_{c_{D'}}\) by compatibility of Kahler differentials with base change, we conclude. The claims about \(\Omega^1_{D/B\G_m}\) now follow easily from compatibility of Kahler differentials with base change along the \(n\) natural maps \(B\G_m \to \Theta^n\) (on factor \(i,\) embed \(B\G_m \inclusion \Theta\) as the closed point; on the other factors, view \(B\G_m = \pt \times \cdots \times \pt \times B\G_m \times \pt \times \cdots \times \pt,\) and map \(\pt \inclusion \Theta\) as the open point). 
\end{proof}
\begin{notn} \label{bgmfactor}
		The morphism \(f_i : B\G_m \to \Theta^n\) constructed in the proof of \autoref{olsson875} as the map
		\[(\Spec k)^{i-1} \times B\G_m \times (\Spec k)^{n-i} \to \Theta^n\]
		will arise often in the sequel.
\end{notn}

\subsubsection{Root stacks and Frobenius twists}

We will need to use Frobenius twists relative to \((\A^1/\G_m)^n\) in our work below; already the relative Frobenius twist appears in our Cartier descent theorem \autoref{cartierdescent}, and it will also appear in our statement of the logarithmic non-abelian Hodge theorem. Thus in this section, we describe such relative Frobenius twists concretely. They will end up being a \emph{multiroot stack} over the usual Frobenius twist.

\begin{defn}[Multiroot stacks]
		Let 
		\[f : X \to (\A^1/\G_m)^n\]
		be a morphism of \(k\)-stacks. The \emph{\((r_1, ..., r_n)\)-multiroot stack of \(f\)} is defined as the fiber product
		\begin{center}
				\begin{tikzcd}
						\mathcal{X}_{r_1, ..., r_n} \ar[r] \ar[d] & X \ar[d, "f"] \\
						(\A^1/\G_m)^n \ar[r, "{([r_1], [r_2], ..., [r_n])}"] & (\A^1/\G_m)^n,
				\end{tikzcd}
		\end{center}
		where \([r] : \A^1/\G_m \to \A^1/\G_m\) is the morphism induced by the map \(t\mapsto t^r.\) 
\end{defn}
\begin{remark}
		Quasicoherent sheaves on multiroot stacks can be described concretely using the language of \emph{parabolic} bundles, due to Borne \cite{borne2006}, \cite{borne2007}. We describe this perspective in more detail in (appendix omitted in this draft), as this language can be very useful for giving concrete examples. Though, while parabolic bundles are often useful for explicit computations, the author finds in practice it is convenient to try working abstractly with the category of quasicoherent sheaves on a root stack, and delay the translation to parabolic bundles as much as possible.
\end{remark}

Frobenius twists relative to \(\Theta^n\) are just multiroot stacks over the usual Frobenius twists. 
\begin{remark}
		The same lemma also appears as Lemma 3.16 of Esnault--Groechenig's \cite{eg}; though we reproduce the short proof for the reader's benefit. 
\end{remark}
\begin{prop} \label{prop1433} 
		Let \(X\) be a smooth \(k\)-scheme and \(D\) a simple normal crossings divisor on \(X\) with \(n\) components; let \(c_D : X \to \Theta^n\) be the associated morphism (see \autoref{olsson875}). Set \(X' := (X/k)^{(1)}, D' = (D/k)^{(1)}\) the Frobenius twists of \(X, D\) relative to \(k.\) 

		Then the relative Frobenius twist \((X/\Theta^n)^{(1)}\) is isomorphic to the \((p, p, ..., p)\)-multiroot stack of \(X'\) along the divisors \(D_1', D_2', ..., D_n'.\)
\end{prop}
\begin{proof}
		We will now be a bit more specific about our base fields; recalling that \(\Theta\) was defined to live over \(k,\) we will now write 
		\[\Theta_k = \A^1_k/\G_{m, k}.\]
		Then \(\Theta_k = \A^1_{\F_p}/\G_{m, \F_p} \times_{\Spec(\F_p)} \Spec(k) = \Theta_{\F_p} \times \Spec(k).\) 

		The absolute Frobenius of \(\Theta_{\F_p}\) is just the \(p^{\text{th}}\) power map \(\Theta_{\F_p} \xto{[p]} \Theta_{\F_p}.\) Hence we have a Cartesian diagram
		\begin{center}
				\begin{tikzcd}
						(X/\Theta_k^n)^{(1)} \ar[r] \ar[d] & X' \ar[r] \ar[d, "c_{D'}"] & X \ar[d, "c_D"] \\
						(\Theta_{\F_p} \times \Spec(k))^n \ar[r, "{([p] \times \id_k)^n}"] & (\Theta_{\F_p} \times \Spec(k))^n \ar[r, "(\id_{\Theta_{\F_p}} \times \Frob_k)"] & (\Theta_{\F_p} \times \Spec(k))^n.
				\end{tikzcd}
		\end{center}
		Thus \((X/\Theta_k^n)^{(1)}\) can also be computed as the fiber product of \(X'\) over \(([p], ..., [p]) : \Theta_k^n \to \Theta^k_n,\) which is by definition the \((p, p, ..., p)\)-mutliroot stack of \(X'\) along the \(D_1', ..., D_n'.\)
\end{proof}

\subsection{Logarithmic de Rham stacks} \label{logdrsubsec}

Let \(k\) be a field of characteristic \(p,\) \(X\) a smooth \(k\)-scheme, and \(D\) a strict normal crossings divisor on \(X\) with \(n\) components \(D_1, ..., D_n.\) Set \(X'\) the Frobenius twist of \(X\) relative to \(k,\) and \(X^{(1)}\) the Frobenius twist of \(c_D : X \to \Theta^n.\) We will keep this setup throughout the rest of this section. 

\begin{defn}
		In this situation, we define the \emph{logarithmic de Rham stack} of \((X, D)\) to be the relative de Rham stack
		\[X^{\dR}(\log D) := (X/\Theta^n)^{\dR}.\]

		Similarly, we define the \emph{sheared logarithmic de Rham stack} of \((X, D)\) to be the sheared relative de Rham stack
		\[X^{\hat{\dR}}(\log D) := (X/\Theta^n)^{\hat{\dR}}.\] 
\end{defn}

In the rest of this section, we will make explicit the general results of \autoref{derhamsection} in the special case of logarithmic de Rham stacks, for the benefit of the reader.

We start by relating quasicoherent sheaves on \(X^{\dR}(\log D)\) with connections having logarithmic poles.
\begin{theorem}[Quasicoherent sheaves on logarithmic de Rham stacks] 
		The category \(\mathcal{D}_{\qc}(X^{\hat{\dR}}(\log D))\) is canonically equivalent to the category of pairs \((\mathcal{E}, \nabla)\) with \(\mathcal{E} \in \mathcal{D}_{\qc}(X)\) and
		\[\nabla : \mathcal{E} \to \mathcal{E} \otimes \Omega_X^1(\log D)\]
		a logarithmic connection on \(\mathcal{E}.\) 
\end{theorem}
\begin{proof}
		\autoref{classicalcomparison} describes \(\mathcal{D}_{\qc}(X^{\hat{\dR}}(\log D))\) in terms of \(\mathcal{D}_{X/\Theta^n}\)-modules; combining the concrete description of \(\mathcal{D}\)-modules given in \autoref{lem583} with Olsson's description of \(\Omega^1_{X/\Theta^n}\) from \autoref{olsson875}, we conclude.
\end{proof}

\begin{defn}
		We define \(\mathcal{D}_X(\log D) := \mathcal{D}_{X/\Theta^n}.\) 
\end{defn} 

\subsection{The log Azumaya property} \label{logazumaya}

We now explain an Azumaya algebra property of logarithmic differential operators. This was first observed by a wonderful paper of Hablicsek \cite{azumayaparabolic}, in the language of parabolic sheaves (which by Borne \cite{borne2007} are equivalent to quasicoherent sheaves on root stacks). 

We take \((X, D)\) as above. 

\begin{theorem}[Azumaya property of logarithmic differential operators] 
		The pushforward \(F_{X/S,*}\mathcal{D}_X(\log D)\) has center \(\OO_{\mathbb{V}(\omega_{X^{(1)}/S})},\) and is in fact Azumaya over its center.
\end{theorem}
\begin{proof}
		This is the main theorem of Hablicsek's \cite{azumayaparabolic}, though in the language of parabolic bundles; by Borne \cite{borne2007}, though, our statement is equivalent. 
\end{proof}

For the reader's benefit, we include an illustration of the main computation used in Hablicsek's proof.
\begin{exmp}[The Azumaya property for \(\A^1\)] \label{exmp1485} 
		Let \(S = \A^1_k/\G_{m,k}\) for \(k\) an \(\F_p\)-algebra, and let \(X = \Spec k[x].\) Take \(D\) the divisor cut out by \(x = 0.\) The logarithmic Frobenius twist \(X^{(1)}\) is the \(p^{\text{th}}\) root stack of \(X' = \Spec k[x^p]\) along the divisor \(x^p = 0.\) 

		Then \(\mathcal{D}_X(\log D) = k[x, x\partial],\) with the two generators having commutator 
		\[x\partial \cdot x - x \cdot x\partial = x.\]
		Note that \((x\partial)^p - x\partial\) and \(x^p\) are both central elements of \(\mathcal{D}_X(\log D).\) 

		The pushforward \(F_{X/S,*}\mathcal{D}_X(\log D),\) thought of as a quasicoherent sheaf on \(X'\) with parabolic structure (in the sense of Borne \cite{borne2007}), is just the \(k[x^p]\)-algebra \(k[x, x\partial]\) with parabolic structure
		\[\cdots \subseteq F^1 = x \cdot k[x, x\partial] \subseteq F^0 = k[x, x\partial] \subseteq F^{-1} = x^{-1} \cdot k[x, x\partial] \subseteq F^{-2} = x^{-2} \cdot k[x, x\partial] \subseteq \cdots.\] 

		The logarithmic cotangent bundle of \(X^{(1)}\) is just the trivial bundle. Writing \(\mathbb{V}(\omega_{X^{(1)}/S}) = X^{(1)} \times_k \Spec k[\zeta],\) we can then upgrade \(F_{X/S,*}\mathcal{D}_X(\log D)\) to an algebra on \(\mathbb{V}(\omega_{X^{(1)}/S})\) by having \(\zeta\) act on \(F_{X/S,*}\mathcal{D}_X(\log D)\) as multiplication by the central element \((x\partial)^p - x\partial.\) In this way, we view \(F_{X/S,*}\mathcal{D}_X(\log D)\) as a quasicoherent algebra on \(\mathbb{V}(\omega_{X^{(1)}/S}).\)

		As the map \(X \to X'\) is faithfully flat, the induced map \(X[\sqrt[p]{pD}] \to X^{(1)}\) of \(p^{\text{th}}\) root stacks is faithfully flat as well. As being Azumaya can be checked on a flat cover, it suffices to prove that the pullback of \(F_{X/S,*}\mathcal{D}_X(\log D)\) to \(X[\sqrt[p]{pD}] \times \A^1_{\zeta}\) is a matrix algebra. As a sheaf on \(X \times \A^1_{\zeta}\) with parabolic structure, this pullback \(\mathcal{D}'\) can be modelled as the parabolic sheaf
		\[F^i\mathcal{D}' = \sum_j (k[x, x\partial] \cdot x^j) \otimes_{k[x^p, (x\partial)^p - x\partial]} (k[x, (x\partial)^p - x\partial] \cdot x^{i-j}),\]
		where \(\zeta\) acts as \((x\partial)^p - x\partial.\) 

		View \(\mathcal{D}_X(\log D)\) as a sheaf on \(X[\sqrt[p]{pD}] \times \A^1_{\zeta}\) via the parabolic bundle
		\[\cdots \subseteq F^1 = k[x, x\partial] \cdot x \subseteq F^0 = k[x, x\partial] \subseteq F^{-1} = k[x, x\partial] \cdot x^{-1} \subseteq \cdots,\]
		where again we have \(\zeta\) acting as \((x\partial)^p - x\partial.\) The structure sheaf of \(X[\sqrt[p]{pD}] \times \A^1_{\zeta}\) can be identified as the subbundle 
		\[\cdots \subseteq F^1 = k[x, (x\partial)^p - x\partial] \cdot x \subseteq F^0 = k[x, (x\partial)^p - x\partial] \subseteq F^{-1} = k[x, (x\partial)^p - x\partial] \cdot x^{-1} \subseteq \cdots\]
		of \(\mathcal{D}_X(\log D).\) We denote this subbundle by \(\mathcal{A}_X.\) 

		Consider now the sheaf of endomorphisms \(\underline{\End}_{\mathcal{A}_X}(\mathcal{D}_X(\log D)),\) viewing \(\mathcal{D}_X(\log D)\) as a \emph{right} \(\mathcal{A}_X\)-module. There is a natural map of parabolic bundles
		\[\mathcal{D}' \to \underline{\End}_{\mathcal{A}_X}(\mathcal{D}_X(\log D)),\]
		which sends \((f \cdot x^j) \otimes (g \otimes x^{i-j}) \in F^i\mathcal{D}'\) to the endomorphism
		\[D \mapsto (fx^j)D(gx^{i-j}).\] 
		This endomorphism is \(\mathcal{A}_X\)-linear; indeed, it commutes with \(D \mapsto Dx\) since \(g \in k[x, (x\partial)^p - x\partial]\) commutes with \(x,\) and it commutes with right multiplication by \((x\partial)^p - x\partial\) for the same reason. 

		This endomorphism lies in \(F^i\underline{\End}_{\mathcal{A}_X}(\mathcal{D}_X(\log D))\) because if \(D\in F^n\mathcal{D}_X(\log D),\) then \(fx^jDgx^{i-j}\) lies in \(F^{n+i}\mathcal{D}_X(\log D)\) (as \(x^j, x^{i-j}\) change degree in a predictable way, and \(f, g\) both act only by increasing degree). This map is an isomorphism by the same argument used by Bezuknivakov \cite{bmm}.
\end{exmp}

\subsection{Residues}

We now explain how residues of logarithmic connections can be geometrized using \((\pt/B\G_m)^{\hat{\dR}}.\) 

\subsubsection{On \((\pt/B\G_m)^{\hat{\dR}}\)} 

\begin{remark}
		The author is grateful to Sanath Devalapurkar for helpful correspondence on the material of this section. In particular, Devalapurkar explained to the author the relationship between Artin-Schreier and \(p\)-curvature (see \autoref{lem1662}).
\end{remark}

We start by computing the sheared relative de Rham stack \((\pt/B\G_m)^{\hat{\dR}}\) explicitly. Some general lemmata will aid in this task.
\begin{lemma} \label{lem1538} 
		Let \(G\) be a finite type, affine commutative group scheme over \(\Spec k.\) Then 
		\[(BG)^{\dR} \heq B(G^{\dR})\]
		and
		\[(BG)^{\hat{\dR}} \heq B(G^{\hat{\dR}}).\]
\end{lemma}
\begin{remark}
		Here, by \(BG,\) we mean the classifying stack of \emph{flat} \(G\)-gerbes. 
\end{remark}
\begin{remark}
		Thanks to this lemma, we will often write \(BG^{\dR}\) instead of \(B(G^{\dR})\) or \((BG)^{\dR},\) since now no confusion can arise.
\end{remark}
\begin{proof}
		The key point is that formation of both de Rham and sheared de Rham stacks preserve quasisyntomic covers; see Lemma 6.3 of Bhatt--Lurie \cite{bhattlurie} for a slightly stronger claim in the non-sheared case. In the sheared case, a similar argument is used to show the analogous result in the upcoming work on sheared prismatization \cite{bkmvz}.

		The diagram
		\begin{center}
				\begin{tikzcd}
						G \ar[r] \ar[d] & \Spec(k) \ar[d] \\
						\Spec(k) \ar[r] & BG
				\end{tikzcd}
		\end{center}
		is Cartesian even in the world of derived stacks; thus applying \(\dR,\) we find that
		\begin{center}
				\begin{tikzcd}
						G^{\dR} \ar[r] \ar[d] & \Spec(k) \ar[d] \\
						\Spec(k) \ar[r] & (BG)^{\dR}
				\end{tikzcd}
		\end{center}
		is Cartesian (since formation of de Rham stacks commutes with derived limits). 

		In characteristic 0, finite type group schemes are smooth; in positive characteristic, finite type group schemes are quasisyntomic. Thus \(G \to \Spec(k)\) is a quasisynotomic surjection; hence \(\Spec(k) \to BG\) is a quasisyntomic surjection. As formation of \(\dR\) preserves quasisynotomic covers, we deduce that \(\Spec(k) \to (BG)^{\dR}\) is still a quasisyntomic cover. Hence we have a surjection \(\pt \to (BG)^{\dR}\) whose fiber is \(G^{\dR}\); this is exactly the data needed to identify \(B(G^{\dR})\heq (BG)^{\dR},\) so we conclude. The same argument works for \(\hat{\dR}.\) 
\end{proof}	
\begin{cor} \label{cor674} 
		We have
		\[\G_m^{\dR} = \cone(\G_m^{\#} \to \G_m),\]
		\[\G_m^{\hat{\dR}} = \cone(\G_m^{\#, \wedge} \to \G_m),\]
		\[B\G_m^{\dR} = \cone(B\G_m^{\#} \to B\G_m),\] 
		\[B\G_m^{\hat{\dR}} = \cone(B\G_m^{\#, \wedge} \to B\G_m).\] 

		Moreover, under these canonical identifications, the morphisms 
		\[B\G_m \to B\G_m^{\dR} \to B\G_m^{\hat{\dR}}\]
		become the canonical maps
		\[B\G_m \to \cone(B\G_m^{\#} \to B\G_m) \to \cone(B\G_m^{\#, \wedge} \to B\G_m).\]
\end{cor}
\begin{proof}
		The first two claims follow immediately from the definition of (sheared) de Rham stacks; the other claims then follow from \autoref{lem1538}. 
\end{proof}
\begin{cor} \label{cor1587} 
		We have
		\[(\pt/B\G_m)^{\dR} \heq B\G_m^{\#},\]
		and
		\[(\pt/B\G_m)^{\hat{\dR}} \heq B\G_m^{\#, \wedge}.\]
\end{cor}
\begin{proof}
		By definition, \((\pt/B\G_m)^{\dR}\) is defined as the fiber product
		\begin{center}
				\begin{tikzcd}
						(\pt/B\G_m)^{\dR} \ar[r] \ar[d] & \pt^{\dR} \ar[d] \\
						B\G_m \ar[r] & B\G_m^{\dR}. 
				\end{tikzcd}
		\end{center}
		As \(\pt^{\dR} = \pt\) and \(B\G_m \to B\G_m^{\dR}\) is just the canonical map \(B\G_m \to \cone(B\G_m^{\#} \to B\G_m),\) we conclude. The sheared version is proven identically.
\end{proof}
\begin{remark} 
		In particular, by Cartier duality, quasicoherent sheaves on \((\pt/B\G_m)^{\hat{\dR}}\) are just the same as \(k\)-vector spaces endowed with an endomorphism; this is exactly in line with the computation \(\mathcal{D}_{\pt/B\G_m} \heq k[x\partial_x]\) of \autoref{noazumaya}. 
\end{remark}

\begin{defn}
		Let \(V^{\nabla} \in \mathcal{D}_{\qc}((\pt/B\G_m)^{\hat{\dR}}).\) The \emph{residue} of \(V^{\nabla}\) is endomorphism \(\res : V \to V\) of the \(k\)-vector space \(V\) underlying \(V^{\nabla}\) obtained by looking at the action of \(x\partial_x \in \mathcal{D}_{\pt/B\G_m}\) on \(V.\) 
\end{defn}
\begin{remark}
		It follows from \autoref{cor1587} that a quasicoherent sheaf on \((\pt/B\G_m)^{\hat{\dR}}\) is equivalent to a pair \((V, \res)\) of a vector space and a residue.
\end{remark} 

We also spell out the \(p\)-curvature for \(\pt \to B\G_m.\) The relative Frobenius twist of \(\pt\) over \(B\G_m\) is \(B\mu_p.\) As \(\mathbb{L}_{B\mu_p/B\G_m} \heq \OO_{B\mu_p}\) (computed for instance by pulling back along \(\pt \to B\G_m\) to get \(\G_m/\mu_p \heq \G_m\)), we find that \((\pt/B\G_m)^{\hat{\dR}} \to B\mu_p\) is a \(\G_a^{\#, \wedge}\)-gerbe. This is exactly what we'd expect: \((\pt/B\G_m)^{\hat{\dR}} = B\G_m^{\#, \wedge},\) and the map
\[B\G_m^{\#, \wedge} \to B\mu_p\]
is induced by the group homomorphism \(\G_m^{\#, \wedge} \to \mu_p,\) which is surjective with kernel \(\G_a^{\#, \wedge}\) according to the following lemma.

\begin{lemma} \label{prop593} 
		There is a short exact sequence of commutative affine group schemes
		\begin{equation} \label{eq595} 0 \to \G_a^{\#,\wedge} \to \G_m^{\#,\wedge} \to \mu_p \to 0.\end{equation}

		Moreover, \(\G_m^{\#, \wedge} \to \mu_p\) admits a section \emph{at the level of schemes}. 
\end{lemma}
\begin{warn} \label{warn1617} 
		The sequence \autoref{eq595} is not split as a short exact sequence of group schemes! We only construct a scheme-theoretic section; it will not respect the group laws. In fact, no section can be a group homomorphism, because the Cartier dual to \autoref{eq595} is the Artin-Schreier sequence
		\[0 \to \Z/p \to \G_a \to \G_a \to 0,\]
		which famously is not split. 
\end{warn}
\begin{remark}
		Sanath Devalapurkar, in a private communication to the author, explained that the non-sheared version of \autoref{eq595} splits group theoretically, and hence \((\pt/B\G_m)^{\dR}\to B\mu_p\) is the trivial gerbe. We will not need this in this paper, so we do not include it.
\end{remark}
\begin{proof}
		For an \(\F_p\)-algebra \(A,\) we have
		\[\mu_p(A) = \{a \in A \mid a^p = 1\},\]
		\[\G_m^{\#,\wedge}(A) = \{(a, \gamma_{\bullet}) \mid a \in A^{\times}, \gamma_{\bullet}\text{ is a system of nilpotent divided powers for \(1-a\)}\},\]
		\[\G_a^{\#,\wedge}(A) = \{(a, \gamma_{\bullet}) \mid a \in A, \gamma_{\bullet}\text{ is a system of nilpotent divided powers for \(a\)}\}.\]

		For surjectivity, just note that if \(a^p = 1,\) then \(1-a\) admits the system of nilpotent divided powers
		\[\gamma_n(1-a) := \begin{cases} \frac{(1-a)^n}{n!} & n < p, \\ 0 & n \geq p. \end{cases}\]
		In fact, this defines a function
		\[s(A) : \mu_p(A) \to \G_m^{\#,\wedge}(A)\]
		which is clearly functorial in \(A,\) and hence gives us a scheme-theoretic section \(s : \mu_p\to\G_m^{\#,\wedge}\) of \(\G_m^{\#,\wedge}\to\mu_p.\)

		The kernel of \(\G_m^{\#,\wedge}\to\mu_p\) is then given by all the systems of divided powers for \(0 \in A.\) 

		These are equivalent to giving elements \(\gamma_n\in A\) so that
		\[\gamma_0 = 1, \gamma_1 = 0,\] 
		\[\gamma_n \cdot \gamma_m = \frac{(n+m)!}{n!m!}\gamma_{n+m}.\]
		
		It's easy to see these conditions imply \(\gamma_n = 0\) for all \(n < p.\) Hence if \(p \nmid n,\) the second condition let's us write
		\[\frac{n!}{i!(n-i)!}\gamma_n = \gamma_{n-i} \cdot \gamma_i = 0\]
		for \(i\) the remainder of \(n\) when divided by \(p.\) Note that \(n!/(i!(n-i)!)\) is an integer not divisible by \(p\), and hence \(\gamma_n = 0\) as well. 

		However, the \(\gamma_{pn}\) have no reason to be zero. They are just elements of \(A\) obeying
		\[\gamma_{p \cdot 0} = 1,\]
		\[\gamma_{pn} \cdot \gamma_{pm} = \frac{(pn+pm)!}{(pn)!(pm)!}\gamma_{p(n+m)} = \frac{(n+m)!}{n!m!}\gamma_{p(n+m)}.\] 
		The last equality is from Lucas' theorem. 

		Hence giving the \(\gamma_n\) are equivalent to giving some \(\gamma_p \in A\) together with any PD structure on \(\gamma_p.\) This is exactly the data of an element of \(\G_a^{\#}(A),\) and so we conclude. 
\end{proof}

The Artin-Schreier sequence appearing in \autoref{warn1617} allows us to compute \(p\)-curvature in terms of residue.
\begin{lemma} \label{lem1662}
		Let \((V, \res)\) be a quasicoherent sheaf on \((\pt/B\G_m)^{\hat{\dR}}.\) Then the \(p\)-curvature morphism
		\[\psi_p : V \to V \otimes F^*\Omega_{B\mu_p/B\G_m} = V\]
		is just the endomorphism
		\[\psi_p = \res^p - \res.\] 
\end{lemma}
\begin{proof}
		Recall that \(p\)-curvature is defined in terms of the Cartesian diagram
		\begin{center}
				\begin{tikzcd}
						(\pt/B\G_m)^{\hat{\dR}} \times_{B\mu_p} \pt \ar[r] \ar[d] & (\pt/B\G_m)^{\hat{\dR}} \ar[d] \\
						\pt \ar[r] & B\mu_p.
				\end{tikzcd}
		\end{center}
		As \((\pt/B\G_m)^{\hat{\dR}} = B\G_m^{\#, \wedge},\) we find from \autoref{prop593} that the fiber product is just \(B\G_a^{\#, \wedge},\) with the map \(B\G_a^{\#, \wedge} \to B\G_m^{\#, \wedge}\) in the above Cartesian diagram being the one induced by the inclusion \(\G_a^{\#, \wedge} \to \G_m^{\#, \wedge}.\) 

		As in \autoref{warn1617}, after applying Cartier duality this inclusion becomes the Artin-Schreier cover; as \(\psi_p\) and residue are both defined in terms of Cartier duality, we deduce that \(\psi_p\) is the image of \(\res\) under Artin-Schreier, and hence \(\psi_p = \res^p - \res,\) as desired.
\end{proof}

\subsubsection{Relation to residue of logarithmic connections}

\begin{theorem} \label{residuerelation}
		Let \(X\) be a smooth curve over an algebraically closed field \(k\), and \(D\) a reduced effective divisor in \(X\) with \(n\) components \(D_1, ..., D_n \in X(k).\) Let \(c_D : X \to (\A^1/\G_m)^n\) be the resulting morphism. 

		For each \(i,\) let \(\phi_i : B\G_m \to (\A^1/\G_m)^n\) be the morphism induced as the product of \(n-1\) copies of the open immersion \(\Spec(k) \inclusion \A^1/\G_m\) with one copy of the closed immersion \(B\G_m \inclusion \A^1/\G_m,\) with this one closed immersion being on the \(i^{\text{th}}\) factor. Then the diagram
		\begin{center}
				\begin{tikzcd}
						\pt \ar[r, "D_i"] \ar[d] & X \ar[d, "c_D"] \\
						B\G_m \ar[r, "\phi_i"] & (\A^1/\G_m)^n
				\end{tikzcd}
		\end{center}
		is Cartesian, and the resulting morphism
		\[\phi_i^{\hat{\dR}} : (\pt/B\G_m)^{\hat{\dR}} \to (X/\Theta^n)^{\hat{\dR}}\]
		is such that \(\phi_i^{\hat{\dR}, *} : \mathcal{D}_{\qc}((X/\Theta^n)^{\hat{\dR}}) \to \mathcal{D}_{\qc}((\pt/B\G_m)^{\hat{\dR}})\) sends the logarithmic connection
		\[\nabla : \mathcal{E} \to \mathcal{E} \otimes \Omega_X^1(\log D)\]
		to \((\mathcal{E}|_{D_i}, \res_i),\) where
		\[\res_i : \mathcal{E}|_{D_i} \xto{\nabla} \mathcal{E}|_{D_i} \otimes \Omega_X^1(\log D)|_{D_i} \xto{\id_{\mathcal{E}} \otimes \res} \mathcal{E}|_{D_i},\]
		for \(\res : \Omega_X^1(\log D)|_{D_i} \to \OO_{D_i}\) Poincare's residue map (sending \(dx/x\) to 1, for \(x\) a local coordinate at \(D_i\)). 
\end{theorem}
\begin{remark}
		The map \(\res_i : \mathcal{E}|_{D_i} \to \mathcal{E}|_{D_i}\) is automatically \(\OO_{D_i}\)-linear since \(\nabla\) is \(k\)-linear and \(\OO_{D_i} = k\); this automatic linearity of the restriction of \(\nabla\) is unique to dimension 1, and makes residues especially easy to geometrize for curves.
\end{remark}
\begin{proof}
		The Cartesian diagram is immediate; the claim about restriction is because in general, \(\ast\)-pullbacks of morphisms between relative de Rham stacks coming from Cartesian diagrams are commuted via restrictions. 
\end{proof}

\subsubsection{Residue and Cartier descent}

Take some \(\mathcal{V} \in \Qcoh(X^{(1)}).\) It's pullback \(\nu^*\mathcal{V} \in \Qcoh(X^{\dR}(\log D))\) is some vector bundle with logarithmic connection; and in particular, along each component \(D_i\) of \(D,\) we can form the residue of \(\nu^*\mathcal{V}\) along \(D_i.\) 

Is there an easy way to read off this residue directly from \(\mathcal{V}\)? The answer is, fortunately, yes; it will turn out, as we shall see, that the parabolic structure of \(\mathcal{V}\) encodes the residue of \(\nu^*\mathcal{V}.\) 

Each inclusion \(D_i \inclusion X\) gives rise to a commutative diagram 
\begin{center}
		\begin{tikzcd}
				(\pt/B\G_m)^{\dR} \ar[r] \ar[d] & X^{\dR}(\log D) \ar[d] \\
				B\mu_p \ar[r] & X^{(1)}
		\end{tikzcd}
\end{center}
of relative de Rham stacks. As the residue is obtained by pulling back along \((\pt/B\G_m)^{\dR} \to X^{\dR}(\log D),\) using this commutative diagram we see that the residue of \(\nu^*\mathcal{V}\) can be computed by first pulling back \(\mathcal{V}\) to \(B\mu_p.\)

\begin{prop} \label{thm1748}
		Let \(\mathcal{V} \in \Qcoh(X^{(1)}).\) Then \(\nu^*\mathcal{V}\) has residue zero along each \(D_i\) if and only if \(\mathcal{V}\) has trivial parabolic structure (that is, \(\mathcal{V}\) lies in the essential image of \(\pi^* : \Qcoh(X') \to \Qcoh(X^{(1)})\)).
\end{prop}
\begin{proof}
		This is proven essentially by staring at the commutative diagram
		\begin{center}
				\begin{tikzcd}
						X^{\dR}(\log D) \ar[r, "\pi'"] \ar[d, "\nu"] & X^{\dR} \ar[d, "\nu'"] \\
						X^{(1)} \ar[r, "\pi"] & X'.
				\end{tikzcd}
		\end{center}

		Indeed, \(\nu^*\mathcal{V}\) has zero residue along each \(D_i\) if and only if it is a usual connection (that is, doesn't have any logarithmic poles) if and only if it lies in the essential image of \((\pi')^* : \Qcoh(X^{\dR}) \to \Qcoh(X^{\dR}(\log D)).\) 

		Thus, the commuativity of our diagram immediately implies that if \(\mathcal{V}\) has trivial parabolic structure, then \(\nu^*\mathcal{V}\) has residue zero along each \(D_i.\) And conversely, if \(\nu^*\mathcal{V}\) has residue zero, then we can write \(\nu^*\mathcal{V} = (\pi')^*(\mathcal{E}, \nabla)\); but \(\nabla\) will have \(p\)-curvature zero since \(\nu^*\mathcal{V}\) does, and so by Cartier descent we deduce \((\mathcal{E}, \nabla) = (\nu')^*(\mathcal{V}')\) for some \(\mathcal{V}'.\) Thus
		\[\nu^*\mathcal{V} \heq \nu^*(\pi^*\mathcal{V}'),\]
		which by fully faithfulness of \(\nu^*\) implies \(\mathcal{V} \heq \pi^*\mathcal{V}',\) as desired.
\end{proof}

\section{Logarithmic non-abelian Hodge theory} \label{lognonsec}

We now prove our logarithmic non-abelian Hodge theorem for curves in positive characteristic. We begin by fixing our setup. 

Fix \(k\) an \emph{algebraically closed} field of positive characteristic \(p,\) and \(X\) a smooth proper curve over \(k.\) Take \(D\) a reduced effective divisor, with components \(D_1, ..., D_n \in X(k)\) (so these \(D_i\) are just distinct points). 

We set \(S := (\A^1_k/\G_{m,k})^n,\) and let \(f : X \to S\) be the morphism classifying \(D\) (see \autoref{olsson875}).

We write \(X^{(1)}\) for the Frobenius twist of \(X\) \emph{relative to \(S\)}, and \(X'\) for the Frobenius twist of \(X\) relative to \(k.\) Recall from \autoref{prop1433} that \(X^{(1)}\) is the \((p, p, ..., p)\)-multiroot stack of \(X',\) along the points \(D_1', ..., D_n'.\) 

We also fix a positive integer \(r\); this will be the rank of the vector bundles we consider.
\subsection{Stating the etale local equivalence}

The non-abelian Hodge theorem concerns two moduli stacks: a moduli stack of vector bundles with connection, and a moduli stack of Higgs fields. We now introduce both. 

\subsubsection{The de Rham side} In our situation, \(\nu_{X/S,\wedge}^*\Omega_{X^{(1)}/S}^1\) is a line bundle on \((X/S)^{\hat{\dR}}.\) Applying \autoref{hitchingeneralities} to the \(k\)-stack \((X/S)^{\hat{\dR}}\) and the line bundle \(\nu_{X/S,\wedge}^*\Omega_{X^{(1)}/S}^1\), we get a Hitchin morphism
\begin{equation} \label{eq609} h : \mathcal{M}((X/S)^{\hat{\dR}}, r, \nu_{X/S,\wedge}^*\Omega_{X^{(1)}/S}^1) \to \mathcal{H}((X/S)^{\hat{\dR}}, r, \nu_{X/S,\wedge}^*\Omega_{X^{(1)}/S}^1).\end{equation}

Thanks to \autoref{eq557}, the Hitchin base of \autoref{eq609} is actually the same as the Hitchin base \(\mathcal{H}(X^{(1)}, r, \Omega_{X^{(1)}/S}^1).\) We will now, following Groechenig, label this Hitchin base as
\[\mathcal{A}^{(1)} := \mathcal{H}(X^{(1)}, r, \Omega_{X^{(1)}/S}^1).\] 

The moduli stack \(\mathcal{M}((X/S)^{\hat{\dR}}, r, \nu_{X/S,\wedge}^*\Omega_{X^{(1)}/S}^1)\) is still a little big, though. We define \(\mathcal{M}_{\dR}(X/S)\) the moduli stack of rank \(r\) vector bundles on \((X/S)^{\hat{\dR}}.\) By \autoref{pcurvmorphism}, taking \(p\)-curvature of a vector bundle with flat connection gives us a natural morphism of stacks
\[\psi_p : \mathcal{M}_{\dR}(X/S) \to \mathcal{M}((X/S)^{\hat{\dR}}, r, \nu_{X/S,\wedge}^*\Omega_{X^{(1)}/S}^1)\]
We define
\[h_{\dR} : \mathcal{M}_{\dR}(X/S) \to \mathcal{A}^{(1)}\]
to be the composition \(h \circ \psi_p.\) 

\begin{remark}
		A \(T\)-point of \(\mathcal{M}_{\dR}(X/S)\) is, by definition, a rank \(r\) vector bundle on 
		\[(X/S)^{\hat{\dR}} \times_k T.\]
		We can rewrite this in a convenient way. As \(k\) is a field, \(T \to \Spec(k)\) is automatically flat; hence \(S_T := S \times_k T\) is a flat \(S\)-scheme, and so the Cartesian diagram
		\begin{center}
				\begin{tikzcd}
				X_T \ar[r] \ar[d] & X \ar[d] \\
				S_T \ar[r] & S
		\end{tikzcd}
		\end{center}
		is Cartesian even in the derived sense; as formation of de Rham stacks commutes with derived limits, we deduce that \((X_T)^{\hat{\dR}} \heq X^{\hat{\dR}} \times_{S^{\hat{\dR}}} (S_T)^{\hat{\dR}}.\) Thus
		\begin{align*}
				(X_T/S_T)^{\hat{\dR}} &:= (X_T)^{\hat{\dR}} \times_{(S_T)^{\hat{\dR}}} S_T \\
									  &\heq (X^{\hat{\dR}} \times_{S^{\hat{\dR}}} (S_T)^{\hat{\dR}}) \times_{(S_T)^{\hat{\dR}}} S_T \\
									  &\heq X^{\hat{\dR}} \times_{S^{\hat{\dR}}} (S \times_k T) \\
									  &\heq (X^{\hat{\dR}} \times_{S^{\hat{\dR}}} S) \times_k T \\
									  &\heq (X/S)^{\hat{\dR}} \times_k T.
		\end{align*}

		Thus \(T\)-points of \(\mathcal{M}_{\dR}(X/S)\) are just the rank \(n\) vector bundles on \((X_T/S_T)^{\hat{\dR}}.\) 
\end{remark}

\subsubsection{The Dolbeault side}

The Dolbeault side is, fortunately, much quicker to define: we put 
\[\mathcal{M}_{\Dol}(X^{(1)}) := \mathcal{M}(X^{(1)}, r, \Omega_{X^{(1)}/S}^1).\]
Applying \autoref{hitchingeneralities} in this situation gives us a Hitchin morphism
\[h_{\Dol} : \mathcal{M}_{\Dol}(X^{(1)}) \to \mathcal{A}^{(1)}.\] 

\subsubsection{The equivalence statement}

We can now state the first form of our non-abelian Hodge theorem.
\begin{theorem}[Logarithmic non-abelian Hodge theory] \label{thm619} 
		There is an etale cover \(\mathcal{U} \to \mathcal{A}^{(1)}\) such that, after base changing to \(\mathcal{U},\) there is an isomorphism of \(\mathcal{U}\)-stacks
		\[\mathcal{M}_{\dR}(X/S) \times_{\mathcal{A}^{(1)}} \mathcal{U} \heq \mathcal{M}_{\Dol}(X^{(1)}) \times_{\mathcal{A}^{(1)}} \mathcal{U}.\] 
\end{theorem}

Actually, \autoref{thm619} will follow from a slightly more precise statement; but to make this more precise comparison, we first need to introduce the \emph{stack of splittings} of the Azumaya algebra \(F_{X/S,*}\mathcal{D}_{X/S}.\) In fact, most of the proof of \autoref{thm619} is really just getting in to a position where the objects of the precise statement are defined; once they are, proving the more precise statement is simple. 

\begin{remark}
		Our proof of \autoref{thm619} is basically identical to Groechenig's \cite{groechenig} proof in the non-logarithmic setting. We view this as an advantage to our approach, as it reveals that, when phrased appropriately, Groechenig's proof also covers logarithmic non-abelian Hodge theory; more generally, it seems that most arguments involving de Rham stacks should work in the relative case without much additional input, allowing one to use Olsson's ideas to get logarithmic versions of results.
\end{remark}

\subsection{The stack of splittings} \label{sec1205} 

We will now study the stack of \emph{splittings} of the Azumaya algebra \(F_{X/S,*}\mathcal{D}_{X/S}.\) 

\subsubsection{Splittings on spectral curves}

\begin{notn} \label{notn1209} 
		Let \(T\) be a \(k\)-scheme, and \(a : T \to \mathcal{A}^{(1)}\) a \(T\)-point of the Hitchin base. We will denote by \(C^{(1)}_T\) the spectral curve associated to \(a\) (see \autoref{cayleyhamilton}). 

		Let \(\mathcal{D} := F_{X/S,*}\mathcal{D}_{X/S}\) be the Azumaya algebra of differential operators (see \autoref{logazumaya}), thought of as an algebra on \(\mathbb{V}(\omega_{X^{(1)}/S}).\) We write \(\mathcal{D}_a := (p^*\mathcal{D})|_{C^{(1)}_T}\) where \(p : \mathbb{V}(\omega_{X^{(1)}/S}) \times_k T \to \mathbb{V}(\omega_{X^{(1)}/S})\) is the canonical projection. Note that \(\mathcal{D}_T\) is then an Azumaya algebra on \(C^{(1)}_T.\)
\end{notn}

\begin{defn}
		We define a \(\mathcal{A}^{(1)}\)-stack \(\mathcal{S},\) called the \emph{stack of splittings}, as the functor sending \(T \to \mathcal{A}^{(1)}\) to \emph{relative splittings} of the Azumaya algebra \(F_{X/S,*}\mathcal{D}_{X/S}\) over the \(T\)-family of spectral curves \(C^{(1)}_T\); that is, the groupoid of \(T\)-points of \(\mathcal{S}\) is the groupoid of pairs \((\mathcal{E}, \phi)\) where \(\mathcal{E}\) is a vector bundle on \(C^{(1)}_T\) and \(\phi : \underline{\End}(\mathcal{E}) \to \mathcal{D}_a\) is an isomorphism (see \autoref{notn1209} for the meaning of \(\mathcal{D}_a\)). 
\end{defn}
\begin{remark}
		We will show that \(\mathcal{S}\) is a stack for the fppf topology, and in fact an algebraic stack, in \autoref{algcor1911}.
\end{remark}

\begin{defn}
		We let \(\mathcal{P}\) denote the \(\mathcal{A}^{(1)}\)-stack whose functor of points sends \(T \to \mathcal{A}^{(1)}\) to the groupoid of line bundles on \(C^{(1)}_T.\) 
\end{defn} 
\begin{remark} \label{rem1159} 
		This \(\mathcal{P}\) is of course the relative Picard stack of the universal spectral curve over \(\mathcal{A}^{(1)}.\) 
\end{remark}

\begin{lemma}
		The functor \(\mathcal{P}\) is an algebraic stack (and in particular is a stack; that is, its functor of points obeys fppf descent). 
\end{lemma}
\begin{proof}
		This \(\mathcal{P}\) is the relative Picard stack of the universal spectral curve on the scheme \(\mathcal{A}^{(1)}\) (see \autoref{rem1159}); this universal spectral curve is an algebraic stack (as it is a root stack over a scheme), and so Lieblich's \cite{lieblich} Lemma 2.3.2 implies \(\mathcal{P}\) is an algebraic stack. 
\end{proof}

There is a natural action of \(\mathcal{P}\) on \(\mathcal{S}\): if \(\LL\) is a line bundle on \(C^{(1)}_T\) and \((\mathcal{E}, \phi) \in \mathcal{S}(T),\) then there is a canonical equivalence \(c : \underline{\End}(\mathcal{E} \otimes \LL) \heq \underline{\End}(\mathcal{E}),\) so we can define \(\LL \cdot (\mathcal{E}, \phi) := (\mathcal{E} \otimes \LL, \phi \circ c) \in \mathcal{S}(T).\) We now prove that this action makes \(\mathcal{S}\) a \(\mathcal{P}\)-torsor. 

First we check simple transitivity. 
\begin{lemma} \label{quasitorsor} 
		The action of \(\mathcal{P}(T)\) on \(\mathcal{S}(T)\) is simply transitive for any \(T \to \mathcal{A}^{(1)}.\) 
\end{lemma}
\begin{proof}
		Take \((\mathcal{E}, \phi)\) and \((\mathcal{E}', \phi')\) two \(T\)-points of \(\mathcal{S}.\) Put \(\LL := \underline{\Hom}_{\mathcal{D}_a}(\mathcal{E}, \mathcal{E}')\), using \(\phi, \phi'\) to turn \(\mathcal{E}, \mathcal{E}'\) into \(\mathcal{D}_a\)-modules. As \(\mathcal{E}, \mathcal{E}'\) are two simple algebras over an Azumaya algebra, by the general theory of Azumaya algebras we deduce that \(\LL\) is a line bundle and that
		\[\mathcal{E} \otimes \LL \heq \mathcal{E}'\]
		as \(\mathcal{D}_a\)-modules. This proves the transitivity. 

		If \(\LL \in \Pic(C^{(1)}_T)\) has \((\mathcal{E}, \phi) \cdot \LL \heq (\mathcal{E}, \phi),\) then in particular \(\mathcal{E} \otimes \LL \heq \mathcal{E},\) so taking determinants we deduce that \(\det(\mathcal{E}) \otimes \LL \heq \det(\mathcal{E}),\) and thus \(\LL \heq \OO.\) This gives us that the action has trivial stabilizers.
\end{proof}

The non-emptiness of \(\mathcal{S}\) is more subtle. We start by proving a non-emptiness claim when testing on \(k\)-points of \(\mathcal{A}^{(1)},\) and then extend.

\subsubsection{Tsen's theorem for root stacks}

\begin{lemma} 
		The morphism \(\mathbb{V}(\Omega^1_{X^{(1)}/S}) \to \mathbb{V}(\Omega^1_{X'/S})\) induced by the map \(X^{(1)} \to X'\) witnesses \(\mathbb{V}(\mathcal{T}_{X^{(1)}/S})\) as the \((p, p, ..., p)\)-multiroot stack of \(\mathbb{V}(\Omega^1_{X'/S})\) along the divisors \((\pi^{-1}(D_1'), ..., \pi^{-1}(D_n'))\), for \(\pi : \mathbb{V}(\Omega^1_{X'/S}) \to X'\) the canonical map and \(D_i'\) the Frobenius twists of the \(D_i\) (that is, \(D_i'\) is the image of \(D_i\) under the relative Frobenius). 
\end{lemma}
\begin{proof}
		Let \(c_{D'} : X' \to S = (\A^1/\G_m)^n\) be the morphism classifying the divisor \(D' = D_1' + \cdots + D_n'.\) 

		Then \(\mathbb{V}(\Omega^1_{X^{(1)}/S}) = \mathbb{V}(\Omega^1_{X'/S}) \times_{X'} X^{(1)},\) by compatibility of relative Kahler differentials with pullback. As \(X^{(1)} = (X', c_{D'}) \times_S (S, ([p], ..., [p])),\) we deduce that 
		\[\mathbb{V}(\Omega^1_{X^{(1)}/S}) = \mathbb{V}(\Omega^1_{X'/S}) \times_S (S, ([p], ..., [p]),\]
		and so our claim follows immediately from the definition of root stacks.
\end{proof}
\begin{cor} \label{lem808}
		Every spectral curve (see \autoref{artdefn684}) for \((X^{(1)}, \Omega^1_{X^{(1)}/S})\) is a \(([p], ..., [p])-\)multiroot stack of a spectral curve for \((X', \Omega^1_{X'/k}(\log D')).\) 
\end{cor}
\begin{proof}
		Set \(\omega' = \Omega^1_{X'/k}(\log D'), \omega^{(1)} = \Omega^1_{X^{(1)}/S}.\) Let our spectral curve \(C^{(1)}\) for \((X^{(1)}, \Omega^1_{X^{(1)}/S})\) be cut out by
		\[(\lambda^{(1)})^n + a_{n-1}^{(1)}(\lambda^{(1)})^{n-1} + \cdots + a_0^{(1)} = 0,\]
		where \(a_i^{(1)} \in H^0(X^{(1)}, (\omega^{(1)})^{\otimes i}).\) Recall the morphism
		\[\alpha_i : H^0(X^{(1)}, (\omega^{(1)})^{\otimes i}) \to H^0(X', (\omega')^{\otimes i})\]
		of \autoref{eq89}. Set \(a_i' := \alpha_i(a_i^{(1)}).\) 

		As the diagram
		\begin{center}
				\begin{tikzcd}
						\mathbb{V}(\omega^{(1)}) \ar[r] \ar[d] & \mathbb{V}(\omega') \ar[d] \\
						X^{(1)} \ar[r] \ar[d] & X' \ar[d] \\
						S \ar[r, "{[p]}"] & S
				\end{tikzcd}
		\end{center}
		has all squares Cartesian, our spectral curve \(C^{(1)}\) is the base change along \(\mathbb{V}(\omega^{(1)}) \to \mathbb{V}(\omega')\) of the spectral curve \(C'\) cut out by
		\[(\lambda')^n + a_{n-1}'(\lambda')^{n-1} + \cdots + a_0' = 0.\]
		In particular, the diagram
		\begin{center}
				\begin{tikzcd}
						C^{(1)} \ar[r] \ar[d] & C' \ar[d] \\
						S \ar[r, "{([p], \cdots, [p])}"] & S
				\end{tikzcd}
		\end{center}
		is Cartesian. Our claim now follows from the definition of multiroot stacks.
\end{proof}

We now recall a result of Achenjang \cite{achenjang}.
\begin{lemma} \label{lem809}
		Let \(\mathcal{Y}\) be a tame algebraic stack over \(k\), with coarse moduli space \(c : \mathcal{Y} \to Y.\) Assume all stabilizers of \(k\)-points of \(\mathcal{Y}\) are groups of the form \(\mu_n\) (with \(n\) allowed to depend on the geometric point) and that \(c\) is an isomorphism over some nonempty open subscheme \(U \subseteq Y.\) 

		Let \(D'\) be an effective Cartier divisor in \(Y.\) Let \(\mathcal{Y}_m\) be the \(m^{\text{th}}\) root stack of \(\mathcal{Y}\) along \(c^*(D').\) 

		Assume also that \(\dim Y = 1.\) If \(H^2(\mathcal{Y}, \G_m) = 0,\) then \(H^2(\mathcal{Y}_m, \G_m) = 0.\) 
\end{lemma}
\begin{proof}
		By Achenjang's \cite{achenjang} Lemma 5.3, we have an exact sequence \begin{equation} \label{eq848} H^2(\mathcal{Y}, \G_m) \to H^2(\mathcal{Y}_m, \G_m) \to H^1(D', \underline{\Z/m\Z}).\end{equation}
		As \(\dim Y = 1,\) the divisor \(D'\) is a disjoint union of geometric points; hence \(H^1(D', \underline{\Z/m\Z}) = 0.\) By assumption, \(H^2(\mathcal{Y}, \G_m) = 0.\) Thus \autoref{eq848} becomes an exact sequence
		\[0 \to H^2(\mathcal{Y}_m, \G_m) \to 0,\]
		and we conclude our desired vanishing. 
\end{proof}
\begin{theorem}[Tsen's theorem for stacky spectral curves, apres Achenjang] \label{tsen}
		Every spectral curve \(C^{(1)}\) of \((X^{(1)}, \Omega^1_{X^{(1)}/S})\) has \(H^2(C^{(1)}, \G_m) = 0.\) In particular, \(\Br(C^{(1)}) = 0,\) so every Azumaya algebra on \(C^{(1)}\) is split. 
\end{theorem}
\begin{warn} \label{warning}
		In our proof of this theorem, we cite Theorem 3.21 of Groechenig \cite{groechenig}. Achenjang pointed out to the author that Theorem 3.21 of Groechenig \cite{groechenig} has a minor error, and Groechenig must include a \emph{generic schemeyness} assumption for that theorem to be true. However, we will only apply that Theorem 3.21 to situations where we have this assumption, so this error will not impact our application. 
\end{warn}
\begin{proof}
		By Theorem 3.21 of Groechenig \cite{groechenig} (though see \autoref{warning}), this theorem is true for the spectral curves \(C'\) of \((X', \Omega^1_{X'/k}(\log D')),\) because those spectral curves are (possibly singular!) schemes.

		By \autoref{lem808}, our spectral curve \(C^{(1)}\) is a multiroot stack over some spectral curve \(C'.\) By repeated application of \autoref{lem809}, we deduce that the vanishing of \(H^2(-, \G_m)\) for \(C'\) implies the same for \(C^{(1)},\) and hence conclude. 
\end{proof}

\autoref{tsen} allows us to split the Azumaya algebra \(\mathcal{D}_a\) for \(a : \Spec(k) \to \mathcal{A}^{(1)}\) any \(k\)-point of our Hitchin base. To upgrade this non-emptiness of \(\mathcal{S}\) on \(k\)-points to a proof that \(\mathcal{S}\) is a \(\mathcal{P}\)-torsor, we will show that \(\mathcal{S}\to\mathcal{A}^{(1)}\) is smooth. 

he proof of this smoothness will occupy the rest of this section. To begin, we check \(\mathcal{S}\) is an algebraic stack; this is actually easiest to do by giving an alternative description of \(\mathcal{S}.\) This alternative description will also be useful in our proof of the non-abelian Hodge theorem.

\subsubsection{Intermezzo: a de Rham BNR correspondence} 

We now give an analgoue of the BNR correspondence (which allows us to describe \(\mathcal{M}_{\Dol}\) using spectral curves) on the de Rham side. Our result is a logarithmic version of Groechenig's \cite{groechenig} Proposition 3.15.

We start with a linear algebra lemma.
\begin{lemma} \label{lem1507} 
		Let \(Y\) be a scheme over a base \(T,\) smooth of relative dimension \(d,\) and let \(F_{Y/T} : Y \to Y'\) be the relative Frobenius. Fix a line bundle \(\LL\) on \(Y',\) and let
		\[\theta : \mathcal{E} \to \mathcal{E} \otimes F^*\LL\]
		be a linear morphism on \(Y,\) for \(\mathcal{E}\) a vector bundle on \(Y.\) If \(\theta\) has characteristic polynomial \(a,\) then \(F_*(\theta) : F_*\mathcal{E}\to F_*\mathcal{E} \otimes \LL\) has characteristic polynomial \(a^{p^d}.\) 
\end{lemma}
\begin{proof}
		The characteristic polynomial is local on \(Y,\) so we may assume \(T = \Spec(A), Y = \A^d_T, \LL = \OO_{Y'}, \mathcal{E} = \OO_Y^r.\) Write \(Y = \Spec A[y_1, ..., y_d].\) Let \(\theta_{ij}\) be the entries of the \(r\times r\) matrix representing \(\theta\) in the \(e_1, ..., e_r\) basis of \(\OO_Y^r.\) 

		Write
		\[\theta_{ij} = \sum_{(f_1, ..., f_d), 0 \leq f_i < p} y_1^{f_1}\cdots y_d^{f_d}\theta_{ij}^{f_1, ..., f_d},\]
		where \(\theta_{ij}^{(f_1, ..., f_d)} \in A[y_1^p, ..., y_d^p].\) Write \(\theta^{f_1, ..., f_d}\) the \(r \times r\) matrix with coefficients \(\theta_{ij}^{f_1, ..., f_d}.\) 

		The pushforward \(F_*\OO_Y^r\) is the trivial bundle of rank \(p^dr\) on \(Y'.\) We can represent \(F_*(\theta)\) as a \(p^d \times p^d\) block matrix, where each block is size \(r \times r,\) and the blocks are indexed by pairs \(((f_1, ..., f_d), (f_1', ..., f_d'))\) of multi-indices \(0 \leq f_i \leq p-1.\) This matrix is written in the basis \(y_1^{f_1}\cdots y_d^{f_d}e_i\) of \(F_*\OO_Y^r.\) Thus the \(r\times r\) block with index \(((f_1, ..., f_d), (f_1', ..., f_d'))\) is just \(\theta^{f_1, ..., f_d}.\) 

		Hence \(F_*(\theta)\) is just a \(p^d \times p^d\) block matrix, each block is \(\theta\); the claim now follows immediately from the definition of characteristic polynomials. 
\end{proof}

\begin{lemma} \label{lem1726} 
		Let \(\mathcal{E}\) be a torsionfree coherent sheaf on \(X^{(1)}\) together with a \(F_{X/S,*}\mathcal{D}_{X/S}\)-module structure. Then there is a unique quasicoherent \(\mathcal{D}_{X/S}\)-module \(\FF\) on \(X\) such that \(F_{X/S,*}\FF \heq \mathcal{E}.\) 
\end{lemma}
\begin{proof}
		Let \(F_{X/k} : X \to X'\) denote the usual relative-to-\(k\) Frobenius on \(X.\) Forgetting the parabolic structures on \(\mathcal{E}\) and \(F_{X/S,*}\mathcal{D}_{X/S},\) we find that \(\mathcal{E}\) can be thought of as a quasicoherent sheaf on \(X'\) with a \(F_{X/k,*}\mathcal{D}_X(\log D)\)-module structure. As \(F_{X/k}\) is a universal homeomorphism, the functors \(F_{X/k}^{-1}\) and \(F_{X/k,*}\) furnish inverse equivalences between \(\Sh(X)\) and \(\Sh(X').\) Setting \(\FF = F_{X/k}^{-1}\mathcal{E},\) we find that \(\FF\) has a \(F_{X/k}^{-1}F_{X/k,*}\mathcal{D}_X(\log D) = \mathcal{D}_X(\log D)\)-module structure. In particular, \(\FF\) is a sheaf of \(\OO_X\)-modules; moreover, by compatibility of relative Frobenius with open immersions, we deduce that \(\FF\) is quasicoherent on \(X.\) 
	
		Thus we have \(F_{X/k,*}\FF \heq \mathcal{E}\) as \(F_{X/k,*}\mathcal{D}_X(\log D)\)-modules on \(X'.\) To upgrade this to a statement on \(X^{(1)},\) we will argue that for any torsionfree coherent \(F_{X/k,*}\mathcal{D}_X(\log D)\)-module \(\mathcal{E}\) on \(X',\) there is a unique way of giving \(\mathcal{E}\) parabolic structures compatible with the \(\mathcal{D}\)-algebra action (in the sense that there is a unique \(F_{X/S,*}\mathcal{D}_{X/S}\)-module on \(X^{(1)}\) which gives \(\mathcal{E}\) when you forget parabolic structures). 

		And indeed, consider a point \(P \in D.\) Choose a uniformizer \(x\) of \(P\) in \(X.\) A parabolic structure for \(\mathcal{E}\) at \(P\) is a \(p\)-step filtration
		\[x^p\mathcal{E} = F^p \subseteq F^{p-1} \subseteq \cdots \subseteq F^0 = \mathcal{E}.\]
		The filtration is trivial away from \(P,\) so it suffices to define it only in a small enough neighborhood of \(P\) where the local coordinate \(x\) makes sense. 

		If this filtration is compatible with the \(\mathcal{D}_{X/S}\)-module structure, then \(x\) must act as degree 1. Thus
		\[F^i \subseteq x^iF^0 = x^i\mathcal{E}.\]
		
		Moreover, take some \(s \in F^i.\) Then \(x^{p-i}s \in F^p = x^p\mathcal{E},\) so we can write \(x^{p-i}s = x^pt\) for some \(t\in\mathcal{E}.\) As \(\mathcal{E}\) is torsion free, this implies \(s = x^it\), so \(x^i\mathcal{E} \subseteq F^i.\) We deduce that \(x^i\mathcal{E} = F^i,\) so our filtration is forced to be the filtration by order of divisibility by \(x.\) In particular, there is a unique way to put a parabolic structure on \(\mathcal{E}\) compatible with the \(\mathcal{D}\)-module structure, as desired.
\end{proof}

\begin{prop}[The de Rham BNR correspondence, apres Groechenig] \label{drbnr}
		Fix a \(k\)-scheme \(T,\) and \(a : T \to \mathcal{A}^{(1)}\) a \(T\)-point of the Hitchin base. Thinking of \(a\) as a degree \(r\) characteristic polynomial, we write \(a^p\) for the corresponding polynomial of degree \(pr,\) and we write \(C^{(1)}_{a^p}\) for the degree \(pr\) spectral curve associated to \(a^p.\) 

		The following three groupoids are canonically equivalent:
		\begin{enumerate}
				\item the groupoid \(h_{\dR}^{-1}(a),\)
				\item the groupoid \(\GG_1\) of torsion free, rank 1 coherent sheaves on \(C^{(1)}_{a^p}\) with a \(\mathcal{D}_{a^p}\)-module structure,
				\item the groupoid \(\GG_2\) of torsion free, rank \(p\) coherent sheaves on \(C^{(1)}_a\) with a \(\mathcal{D}_a\)-module structure.
		\end{enumerate}
\end{prop}
\begin{remark}[Comparison to de Cataldo-Zhang \cite{dCZ}] 
		de Cataldo--Zhang (\cite{dCZ} Lemma 4.1) prove a very similar theorem, but without using these stacky spectral curves. The relationship between our two theorems is explained by the following remark: if \(\FF\) is a torsion free coherent sheaf on \(C'_a\) (the non-stacky spectral curve) with a \(\mathcal{D}_a\)-module structure, then there is a unique way to endow \(\FF\) with a parabolic structure compatible with the \(\mathcal{D}_a\)-module structure. Indeed, this is more or less the argument given to prove \autoref{lem1726}, which will be a crucial ingredient in our proof. 
\end{remark}
\begin{proof}
		Put \(\pi : C^{(1)}_a \to X^{(1)}_T\) the projection map, which is finite flat of degree \(r.\) 

		Given \(\FF \in \GG_2,\) we can form \(\pi_*\FF,\) which will be a vector bundle on \(X^{(1)}_T\) of rank \(pr\) carrying a Higgs field of characteristic polynomial \(a^p\) and an action of \(F_*\mathcal{D}_{X_T/S_T}.\) Descending to the cotangent bundle, analogously to \autoref{bnr} we can produce a torsion free, rank 1 coherent sheaf on \(C^{(1)}_{a^p}\) with a \(\mathcal{D}_{a^p}\)-module structure. In this way, essentially the same argument as \autoref{bnr} shows \(\GG_1 \heq \GG_2.\) 

		To conclude, we give an equivalence \(h_{\dR}^{-1}(a) \heq \GG_2.\) Given \((\mathcal{E}, \nabla) \in h_{\dR}^{-1}(a),\) the Frobenius pushforward \(F_*\mathcal{E}\) will be a vector bundle on \(X^{(1)}_T\) of rank \(pr\) carrying an action of \(F_*\mathcal{D}_{X_T/S_T}.\) 

		Moreover, the \(p\)-curvature \(\psi_p : \mathcal{E} \to \mathcal{E} \otimes F^*\omega_{X^{(1)}_T/S_T}\) gives us a morphism
		\[F_*(\psi_p) : F_*\mathcal{E} \to F_*\mathcal{E} \otimes \omega_{X^{(1)}_T/S_T}.\]
		Equivalently, this \(F_*\psi_p\) encodes the action of \(Z(F_*\mathcal{D}_{X_T/S_T}\) on \(F_*\mathcal{E}.\) By \autoref{lem1507}, the characteristic polynomial of \(F_*(\psi_p),\) so by \autoref{bnr} we can descend \(\mathcal{E}\) to a torsion free rank 1 coherent sheaf on \(C^{(1)}_{a^p}\) with a \(\mathcal{D}_{a^p}\)-module structure. Thus we have produced a fully faithful functor
		\[h_{\dR}^{-1}(a) \to \GG_1.\]

		For essential surjectivity, take \(\FF \in \GG_1.\) Let \(\pi' : C^{(1)}_{a^p} \to X^{(1)}_T\) be the projection map, which is finite flat of degree \(pr.\) Thus \(\pi'_*\FF\) is a rank \(pr\) vector bundle on \(X^{(1)}_T,\) carrying the structure of a \(F_*\mathcal{D}_{X_T/S_T}\)-module with \(p\)-curvature having characteristic polynomial \(a^p.\) By \autoref{lem1726}, we can choose \(\mathcal{E}\) a rank \(r\) vector bundle on \(X_T\) such that \(F_*\mathcal{E} \heq \pi'_*\FF,\) and moreover such that \(\mathcal{E}\) has a \(\mathcal{D}_{X/S}\)-module structure making \(F_*\mathcal{E} \heq \pi'_*\FF\) an equivalence of \(F_{X_T/S_T,*}\mathcal{D}_{X_T/S_T}\)-modules. This \(\mathcal{E}\) will then have \(p\)-curvature with characteristic polynomial \(a\) (as whatever its characteristic polynomial \(b\) is, we will have \(b^p = a^p,\) which implies \(b = a\) as \(b, a\) are polynomials with coefficients in an \(\F_p\)-vector space). This shows the essential surjectivity, so we conclude.
\end{proof}

We can now redescribe \(\mathcal{S}.\)
\begin{defn}
		We write \(\mathcal{M}_{\dR}^{\reg}\) for the stack of connections having \emph{regular \(p\)-curvature}: that is, those connections whose corresponding torsion free rank \(p\) coherent sheaf under \autoref{drbnr} is a rank \(p\) vector bundle. 
\end{defn}
\begin{lemma}[Splittings and regular connections] \label{regconnections}
		There is an equivalence \(\mathcal{S} \heq \mathcal{M}_{\dR}^{\reg}\) over \(\mathcal{A}^{(1)}.\)
\end{lemma}
\begin{proof}
		By definition, a splitting of \(\mathcal{D}_a\) is a pair of a rank \(p\) vector bundle \(\mathcal{V}\) on \(C^{(1)}_{a^p}\) together with an isomorphism \(\underline{\End}(\mathcal{V}) \heq \mathcal{D}_a.\) But such an identification is identical to a \(\mathcal{D}_a\)-module structure on \(\mathcal{V}\) and so we conclude. 
\end{proof}
\begin{cor} \label{algcor1911}
		\(\mathcal{S}\) is an algebraic stack over \(\mathcal{A}^{(1)}.\)
\end{cor}
\begin{proof}
		As \(\mathcal{M}_{\dR}, \mathcal{M}_{\dR}^{\reg}\) are both known to be algebraic stacks, this follows from \autoref{regconnections}. 
\end{proof}

\subsubsection{Smoothness of \(\mathcal{S}\to\mathcal{A}^{(1)}\)}

Now that we know \(\mathcal{S}\) is an algebraic stack, we can prove smoothness; this will make it much easier to show \(\mathcal{S}\) is a \(\mathcal{P}\)-torsor.

We need one preliminary lemma. 
\begin{lemma} \label{lem2056} 
		Let \(i : \mathcal{Y} \inclusion \mathcal{Y}'\) be a square zero extension of algebraic stacks of \emph{dimension 1}, and \(\mathcal{A}\) an Azumaya algebra on \(\mathcal{Y}'.\) Any splitting of \(i^*\mathcal{A}\) can be extended to a splitting of \(\mathcal{A}\) on \(\mathcal{Y}'.\) 
\end{lemma}
\begin{proof}
		There is a natural injection \(\Br(\mathcal{Y}') \inclusion H^2(\mathcal{Y}', \G_m)\) for any algebraic stack (see Remark 2.19 of \cite{achenjang}). 

		Thus it suffices to prove that the map \(i^* : H^2(\mathcal{Y}', \G_m) \to H^2(\mathcal{Y}, \G_m)\) is injective. Fortunately for us, this basically follows from an argument of Groechenig (specifically, the bottom of page 17 of \cite{groechenig}). Let \(\mathcal{Y}\) be cut out by the quasicoherent sheaf of ideals \(\II\) with \(\II^2 = 0\). Then there is a short exact sequence
		\[0 \to \II \xto{\exp} \OO_{\mathcal{Y}'}^{\times} \to i_*\OO_{\mathcal{Y}}^{\times} \to 0,\]
		where \(\exp : \II \to \OO_{\mathcal{Y}'}^{\times}\) is defined by \(\exp(f) = 1+f\); this is a homomorphism since \(\II^2 = 0.\) 

		Thus the injectivity of \(H^2(\mathcal{Y}', \G_m) \to H^2(\mathcal{Y}, \G_m)\) is immediate from the fact that \(H^2(\mathcal{Y}', \II) = 0,\) which is true by our dimension assumption as \(\II\) is quasicoherent.
\end{proof}

\begin{prop} \label{smooth1243} 
		The morphisms \(\mathcal{P} \to \mathcal{A}^{(1)}\) and \(\mathcal{S} \to \mathcal{A}^{(1)}\) are both smooth (but not representable!).
\end{prop}
\begin{proof}
		This is proven for \(\mathcal{P}\) in tag 0DPK of the Stacks project \cite{stacks} under a few more assumptions; however, the argument given in the Stacks project still works for us in light of our \autoref{lem2056} and the fact that \(\mathcal{S}, \mathcal{P}\) are still algebraic stacks in our situation. 
\end{proof}

Finally, we may conclude.

\begin{theorem} \label{thm1436} 
		The action of \(\mathcal{P}\) on \(\mathcal{S}\) makes \(\mathcal{S}\) into an etale \(\mathcal{P}\)-torsor.
\end{theorem}
\begin{proof}
		By \autoref{quasitorsor}, we only need show non-emptiness etale locally on \(\mathcal{A}^{(1)}.\) But \autoref{smooth1243} shows \(\mathcal{S} \to \mathcal{A}^{(1)}\) is smooth, and \autoref{tsen} implies it is surjective; so etale locally on \(\mathcal{A}^{(1)},\) the map \(\mathcal{S} \to \mathcal{A}^{(1)}\) admits a section. 
\end{proof}
\begin{remark}
		In de Cataldo-Zhang \cite{dCZ}, the exact same stack \(\mathcal{S}\) appears (though they call it \(\mathcal{H}\)). However, they are not able to prove a torsor property for \(\mathcal{S},\) essentially because they don't use stacky spectral curves. As explained in Lopez \cite{lopez}, a stacky curve in general has a larger Picard group than its coarse moduli space. de Cataldo-Zhang \cite{dCZ} only work with the coarse moduli spaces of our spectral curves, and hence their Picard groups are smaller, and so the action of their Picard groups on \(\mathcal{S}\) is not transitive. We will revisit this point in more detail in \autoref{sec2085} and \autoref{sec2086}.
\end{remark}

\subsection{The more precise comparison}

Now that we understand \(\mathcal{S},\) we can state a more precise version of \autoref{thm619}. 

\begin{construction}[Action of \(\mathcal{P}\) on \(\mathcal{M}_{\Dol}(X^{(1)})\)] 
		Fix \(T\) a \(k\)-scheme and \(a : T \to \mathcal{A}^{(1)}\) a morphism. A \(T\)-point of \(\mathcal{P}\) is, by definition, a line bundle \(\LL\) on \(C^{(1)}_a\); and a \(T\)-point of \(\mathcal{M}_{\Dol}(X^{(1)})\) is, by the BNR correspondence \autoref{bnr}, equivalent to a torsion free rank 1 coherent sheaf \(\FF\) on \(C^{(1)}_a.\) The tensor product \(\LL \otimes \FF\) is still a torsion free rank 1 coherent sheaf on \(C^{(1)}_a,\) and hence by \autoref{bnr} is a \(T\)-point of \(\mathcal{M}_{\Dol}(X^{(1)}).\) In this way, we get an action of \(\mathcal{P}\) on \(\mathcal{M}_{\Dol}(X^{(1)}).\) 
\end{construction}

We will now seek to prove the following more precise variant of \autoref{thm619}. 
\begin{theorem} \label{thm1447} 
		There is an equivalence of \(\mathcal{A}^{(1)}\)-stacks
		\[\mathcal{S} \times^{\mathcal{P}} \mathcal{M}_{\Dol}(X^{(1)}) \heq \mathcal{M}_{\dR}(X).\] 
\end{theorem}

Before proving \autoref{thm1447}, we explain why it implies \autoref{thm619}. 
\begin{proof}[Proof of \autoref{thm619}] 
		By \autoref{thm1436}, \(\mathcal{S}\) is an etale \(\mathcal{P}\)-torsor. Hence after some etale cover \(U\to \mathcal{A}^{(1)}\) we have \(\mathcal{S}_U \heq \mathcal{P}_U,\) and so after pulling back to \(U,\) \autoref{thm1447} becomes
		\[\mathcal{M}_{\Dol}(X^{(1)}) \times_{\mathcal{A}^{(1)}} U \heq \mathcal{S}_U \times^{\mathcal{P}_U} (\mathcal{M}_{\Dol}(X^{(1)}) \times_{\mathcal{A}^{(1)}} U) \heq \mathcal{M}_{\dR}(X) \times_{\mathcal{A}^{(1)}} U,\]
		as desired. 
\end{proof}

To prove \autoref{thm619}, we will follow the strategy of Groechenig \cite{groechenig} in the non-logarithmic case. Thus we first prove a \(\mathcal{P}\)-equivariant analogue of \autoref{thm619}, and then take a quotient.
\begin{lemma}[Logarithmic version of Groechenig \cite{groechenig} Theorem 3.29] \label{lem1683}
		There is an isomorphism 
		\begin{equation} \label{eq1685} \phi : \mathcal{S} \times_{\mathcal{A}^{(1)}} \mathcal{M}_{\Dol} \heq \mathcal{S} \times_{\mathcal{A}^{(1)}} \mathcal{M}_{\dR}\end{equation}
		of \(\mathcal{A}^{(1)}\)-stacks. 
\end{lemma}
\begin{proof}
		Let \(T\) be a test scheme, equipped with a map \(a : T \to \mathcal{A}^{(1)}\). The \(T\)-points of \(\mathcal{S} \times_{\mathcal{A}^{(1)}} \mathcal{M}_{\Dol}\) form the groupoid of triples \((\mathcal{V}, \phi : \End(\mathcal{V}) \heq \mathcal{D}_a, \LL),\) where \((\mathcal{V}, \phi)\) is a splitting of the Azumaya algebra \(\mathcal{D}_a\) on the spectral curve \(C^{(1)}_a,\) and \(\LL\) is a rank 1 torsion free coherent sheaf on \(C^{(1)}_a.\) 

		By \autoref{drbnr}, the groupoid of \(T\)-points of \(\mathcal{S} \times_{\mathcal{A}^{(1)}} \mathcal{M}_{\dR}\) is the groupoid of triples \((\mathcal{V}, \phi : \End(\mathcal{V}) \heq \mathcal{D}_a, \mathcal{E}),\) where \((\mathcal{V}, \phi)\) is a splitting of \(\mathcal{D}_a\) and \(\mathcal{E}\) is a torsion free, rank \(p\) coherent sheaf on \(C^{(1)}_a\) together with a \(\mathcal{D}_a\)-module structure. 

		Define 
		\[\phi(T)(\mathcal{V}, \phi, \LL) = (\mathcal{V}, \phi, \LL \otimes \mathcal{V}).\]
		Note that \(\LL \otimes \mathcal{V}\) is torsionfree of rank \(p,\) and has a natural action of \(\End(\mathcal{V}),\) which thanks to \(\phi\) furnishes a \(\mathcal{D}_a\)-module structure on \(\mathcal{E} \otimes \mathcal{V}.\) This is how we view \((\mathcal{V}, \phi, \LL\otimes\mathcal{V})\) as a \(T\)-point of \(\mathcal{S}\times_{\mathcal{A}^{(1)}} \mathcal{M}_{\dR}.\) 

		To see essential surjectivity of \(\phi(T),\) take some a \(T\)-point \((\mathcal{V}, \phi, \mathcal{E})\) on the de Rham side. Then \(\LL := \underline{\Hom}_{\mathcal{D}_a}(\mathcal{V}, \mathcal{E})\) will be torsion free of rank 1, and the evaluation map \(\LL \otimes \mathcal{V} \to \mathcal{E}\) is an euqivalence. The full faithfulness of \(\phi(T)\) is similar, and so we conclude. 
\end{proof}

\subsubsection{The proof of \autoref{thm1447}}

To conclude our proof of \autoref{thm1447}, we will essentially quotient the equivalence \autoref{eq1685} by \(\mathcal{P}.\) This strategy is slightly different than Groechenig's \cite{groechenig} original approach, and is essentially identical to the non-logarithmic argument made by de Cataldo--Groechenig--Zhang \cite{dCGZ} in their Section 4.2. 

First, observe that \(\mathcal{P}\) has natural actions on both sides of \autoref{eq1685}: multiplication by \(\mathcal{T} \in \mathcal{P}(T)\) can be defined to 
\begin{enumerate}
		\item send \((\mathcal{V}, \phi, \LL)\) to \((\mathcal{V} \otimes \mathcal{T}, \phi, \LL \otimes \mathcal{T}^{-1}),\)
		\item send \((\mathcal{V}, \phi, \mathcal{E})\) to \((\mathcal{E} \otimes \mathcal{T}, \phi, \mathcal{E}).\) 
\end{enumerate}
With these actions, the isomorphism \(\phi\) of \autoref{lem1683} is \(\mathcal{P}\)-equivariant. Hence, quotienting by \(\mathcal{P},\) it induces an isomorphism of stacks
\[\mathcal{S}\times^{\mathcal{P}} \mathcal{M}_{\Dol} \heq (\mathcal{S}/\mathcal{P}) \times_{\mathcal{A}^{(1)}} \mathcal{M}_{\dR}.\]

As \(\mathcal{S}\) is a \(\mathcal{P}\)-torsor, the quotient \(\mathcal{S}/\mathcal{P} \heq \mathcal{A}^{(1)},\) and hence we deduce \autoref{thm1447}. 

\subsection{de Cataldo--Zhang's Hitchin base} \label{sec2085}

In the rest of this section, we will explain how to use our logarithmic non-abelian Hodge theorem to prove de Cataldo--Zhang's \cite{dCZ} logarithmic non-abelian Hodge theorem. 

To start, we of course must \emph{state} their theorem. This requires us to introduce a \emph{new} Hitchin base. In the logarithmic setting, besides \(p\)-curvature, the de Rham side has another linear invariant: the \emph{residue} of a connection. 

By \autoref{residuerelation}, each \(P_i \in D\) gives us a map
\[(\pt/B\G_m)^{\hat{\dR}} \xto{P_i^{\hat{\dR}}} X^{\hat{\dR}}(\log D),\]
pullback along which realizes residue. Applying \autoref{hitchinfunctoriality} to these morphisms, we produce a map
\[\mathcal{A}^{(1)} \to \prod_{P_i \in D} \A^r,\]
where we view \(\A^r\) as the Hitchin base on the point. There are similarly natural maps
\[h_i : \mathcal{M}_{\dR}(X) \to \A^r\]
for each \(P_i\in D,\) sending a connection to the characteristic polynomial of its residue at \(P_i.\) By \autoref{lem1662}, the characteristic polynomial of the residue at \(P_i,\) and the characteristic polynomial of \(\psi_p\) at \(P_i,\) are tightly related. This inspired de Cataldo-Zhang \cite{dCZ} to introduce the following modification of the Hitchin base. 

\begin{defn}[de Cataldo--Zhang's \cite{dCZ} Hitchin base]
		Following de Cataldo--Zhang \cite{dCZ}, we let \(\mathcal{A}^{\DCZ}\) denote the fiber product
		\begin{center}
				\begin{tikzcd}
						\mathcal{A}^{\DCZ} \ar[r] \ar[d] & \mathcal{A}^{(1)} \ar[d] \\
						\prod_{P_i \in D} \A^r \ar[r, "\AS"] & \prod_{P_i \in D} \A^r,
				\end{tikzcd}
		\end{center}
		where the map \(\AS\) is obtained as follows: viewing an element of \(\A^r\) as the coefficients of a degree \(r\) monic polynomial \(p\), let \(\AS(p)\) be the coefficients of the polynomial whose roots are these Artin-Scrheier images of the roots of \(p.\)
\end{defn}
\begin{remark}
		As mentioned above, \autoref{lem1662} implies that the maps \(h_{\dR}\) and \(h_i\) give rise to a map
		\[h_{\DCZ} : \mathcal{M}_{\dR}(X, r) \to \mathcal{A}^{\DCZ},\]
		which we call \emph{de Cataldo--Zhang's Hitchin morphism}.
\end{remark}

The stack of splittings \(\mathcal{S}\) still appears in de Cataldo--Zhang's logarithmic non-abelian Hodge theorem. However, they use a smaller Picard group.
\begin{defn}
		We let \(\mathcal{P}^{\DCZ}\) denote the \(\mathcal{A}^{(1)}\)-stack whose functor of points sends \(a : T \to \mathcal{A}^{(1)}\) to the groupoid of line bundles on \(C'_a,\) which recall is the \emph{non-stacky} spectral curve obtained as the coarse moduli space of our stacky spectral curve \(C^{(1)}_a.\) 
\end{defn}
\begin{remark}
		There is a natural inclusion \(\mathcal{P}^{\DCZ}\inclusion\mathcal{P},\) but in general the stacky spectral curve \(C^{(1)}_a\) can have many more line bundles than \(C'_a\); see for instance Theorem 1.1 of Lopez's \cite{lopez}. 
\end{remark}

de Cataldo--Zhang also use a smaller Dolbeault side than we do.
\begin{defn}
		We let \(\mathcal{M}_{\Dol}^{\DCZ} \inclusion \mathcal{M}_{\Dol}\) denote the \(\mathcal{A}^{(1)}\)-stack whose groupoid of \(T\)-points is the full subgroupoid of \(\mathcal{M}_{\Dol}(T)\) spanned by those torsion-free rank 1 coherent sheaves on \(C^{(1)}_T\) which have trivial parabolic structure (that is, are pulled back from \(C'_T\)). 
\end{defn}

Finally, we can state de Cataldo--Zhang's theorem.
\begin{remark}
		Identifying \(\mathcal{S} \heq \mathcal{M}_{\dR}^{\reg}\) as in \autoref{regconnections} and using de Cataldo--Zhang's Hitchin morphism \(\mathcal{M}_{\dR} \xto{h_{\DCZ}} \mathcal{A}^{\DCZ},\) we can view \(\mathcal{S}\) as a \(\mathcal{A}^{\DCZ}\)-stack.
\end{remark}
\begin{theorem}[de Cataldo--Zhang \cite{dCZ}] \label{dczmainthm}
		Let \(\widetilde{\mathcal{M}_{\Dol}^{\DCZ}}\) denote the base change of \(\mathcal{M}_{\Dol}^{\DCZ}\) from an \(\mathcal{A}^{(1)}\)-stack to a \(\mathcal{A}^{\DCZ}\)-stack. Then there is a natural morphism of \(\mathcal{A}^{\DCZ}\)-stacks
		\[\mathcal{S}\times^{\widetilde{\mathcal{P}^{\DCZ}}} \widetilde{\mathcal{M}_{\Dol}^{\DCZ}} \to \mathcal{M}_{\dR},\]
		which is an isomorphism over the open subscheme of \(\mathcal{A}^{\DCZ}\) where \(\mathcal{S} \to \mathcal{A}^{\DCZ}\) is nonempty.
\end{theorem}

To deduce \autoref{dczmainthm} from our results boils down to understanding how parabolic structures and residues interrelate. We now explore this.

\subsection{Residues and parabolic structures} \label{sec2086}

Our strategy for showing \autoref{dczmainthm} will be to prove that \(\mathcal{P}^{\DCZ} \inclusion \mathcal{P}\) can alternatively be characterized as the subgroupoid of line bundles which fix the residue of a connection. Once we have this characterization, deducing \autoref{dczmainthm} from our \autoref{thm1447} will be formal.

Fix a \(T\)-point \(a : T \to \mathcal{A}^{(1)}\) of the Hitchin base, for \(T\) a \(k\)-scheme. Let \(\nabla : \mathcal{E} \to \mathcal{E} \otimes \Omega^1_{X_T/S_T}\) be a logarithmic connection on a rank \(r\) vector bundle \(\mathcal{E}\) on \(X_T,\) whose \(p\)-curvature \(\psi_p\) has characteristic polynomial \(a.\) The de Rham BNR correspondence associates to this \(\nabla\) a torsion free rank \(p\) coherent sheaf \(\FF\) on \(C^{(1)}_a,\) together with a \(\mathcal{D}_a\)-module structure.

\begin{prop} \label{trivialpar}
		Let \(\LL \in \mathcal{P}(a)\) be a line bundle on \(C^{(1)}_a.\) The following are equivalent: 
		\begin{enumerate}
				\item \(\LL\) has trivial parabolic structure;
				\item for \emph{any} such connection \((\mathcal{E}, \nabla),\) \(\FF\) and \(\LL\otimes\FF\) correspond to \(\mathcal{D}_{X_T/S_T}\)-modules with the same residue at every point; 
				\item for \emph{one} connection \((\mathcal{E}, \nabla)\) as above, \(\FF\) and \(\LL \otimes \FF\) correspond to \(\mathcal{D}_{X_T/S_T}\)-modules with the same residue at every point.
		\end{enumerate}
\end{prop}
\begin{proof}
		It's easier to prove this if we think about how \(\LL\) acts on \((\mathcal{E}, \nabla)\) directly, instead of juggling these intermediaries \(\FF.\) For this, just note that a line bundle \(\LL\) on \(C^{(1)}_a\) will pushforward to a rank \(r\) vector bundle \(\mathcal{V}\) on \(X^{(1)}\) under the finite flat degree \(r\) map \(C^{(1)}_a \to X^{(1)}.\) Then \(\LL \otimes \FF\) is the de Rham BNR partner of \(\mathcal{E} \otimes F_{X/S}^*\mathcal{V},\) with its natural tensor product connection. 

		By \autoref{thm1748}, the residue of the natural connection on \(F_{X/S}^*\mathcal{V}\) is zero if and only if \(\mathcal{V}\) has trivial parabolic structure. In a tensor product of connections, the residue is \(\res \otimes \id + \id \otimes \res',\) and thus we deduce immediately that our three conditions are equivalent. 
\end{proof}

\subsubsection{Deducing de Cataldo--Zhang's theorem}

\begin{proof}[Proof of \autoref{dczmainthm}]
		Fix some \(\mathcal{E} \in \mathcal{S}_a.\) 

		For any \((\mathcal{E}', \LL')\in (\mathcal{S}\times^{\mathcal{P}} \mathcal{M}_{\Dol})_a,\) by transitivity of \(\mathcal{P}\) on \(\mathcal{S},\) we can find some \(\LL\) so that \((\mathcal{E}', \LL') \heq (\mathcal{E}, \LL)\) in the groupoid \((\mathcal{S}^{\times{\mathcal{P}}} \mathcal{M}_{\Dol})(a).\) By \autoref{trivialpar}, because \(\mathcal{E}\) and \(\mathcal{E} \otimes \LL\) have the same spectral data, we deduce \(\LL\) has trivial parabolic structure. 

		Thus the map
		\begin{equation} \label{eq2017} (\mathcal{S} \times^{\tilde{\mathcal{P}^{\DCZ}}} \tilde{\mathcal{M}_{\Dol}^{\DCZ}})(a) \to (\mathcal{S} \times^{\mathcal{P}} \mathcal{M}_{\Dol})(a)\end{equation}
		is essentially surjective, since we can represent every object of the target groupoid via a pair \((\mathcal{E}, \LL)\) where \(\LL\) has trivial parabolic structure.

		It is also fully faithful, because if \((\mathcal{E}, \LL), (\mathcal{E}', \LL')\) are two objects of the source category (so, \(\LL, \LL'\) both have trivial parabolic structure), then morphisms 
		\[(\mathcal{E}, \LL) \to (\mathcal{E}', \LL')\]
		are equivalent to giving a triple consisitng of  
		\begin{enumerate}
				\item a line bundle \(\mathcal{V}\) on \(C^{(1)}_a,\)
				\item an isomorphism \(\mathcal{E} \otimes \mathcal{V} \heq \mathcal{E}'\) of \(\mathcal{D}_a\)-modules, and
				\item an isomorphism \(\mathcal{V}^{\vee} \otimes \LL\heq \LL'.\) 
		\end{enumerate}

		In the source category, morphisms are the same but the line bundle \(\mathcal{V}\) is required to have trivial parabolic structure. But again by \autoref{trivialpar}, an ismomorphism \(\mathcal{E} \otimes \mathcal{V} \heq \mathcal{E}'\) existing implies \(\mathcal{V}\) has trivial parabolic structure. 
\end{proof}

\bibliographystyle{plain}
\bibliography{refs}

\begin{thebibliography}{10}

\bibitem{achenjang}
Niven Achenjang.
\newblock On brauer groups of tame stacks.
\newblock 2025.

\bibitem{weightfiltrations}
Toni Annala and Piotr Pstrągowski.
\newblock A note on weight filtrations at the characteristic.
\newblock \url{https://arxiv.org/abs/2502.19626}, 2025.
\newblock arXiv:2502.19626 [math.KT].

\bibitem{stacks}
The Stacks~Project Authors.
\newblock {\em The Stacks Project}.
\newblock 2025.
\newblock Version as of 2025-10-21.

\bibitem{bnr}
A.~Beauville, M.~S. Narasimhan, and S.~Ramanan.
\newblock Spectral curves and the generalised theta divisor.
\newblock {\em Journal f{\"u}r die reine und angewandte Mathematik}, 398:169--179, 1989.

\bibitem{bmm}
Roman Bezrukavnikov, Ivan Mirković, and Dmitriy Rumynin.
\newblock Localization of modules for a semisimple lie algebra in prime characteristic.
\newblock {\em arXiv preprint arXiv:math/0205144}, 2002.
\newblock Last revised: Oct.\,13\,2006.

\bibitem{bezazumaya}
Roman Bezrukavnikov, Ivan Mirković, and Dmitriy Rumynin.
\newblock Localization of modules for a semisimple lie algebra in prime characteristic.
\newblock {\em Annals of Mathematics}, 167(3):945--991, 2008.
\newblock available at \url{https://annals.math.princeton.edu/wp-content/uploads/annals-v167-n3-p05.pdf}.

\bibitem{bhatt}
Bhargav Bhatt.
\newblock Prismatic $f$-gauges.
\newblock \url{https://www.math.ias.edu/~bhatt/teaching/mat549f22/lectures.pdf}, 2022.
\newblock Accessed: October 15th, 2025.

\bibitem{bkmvz}
Bhargav Bhatt, Artem Kanaev, Akhil Mathew, Vadim Vologodsky, and Mingjia Zhang.
\newblock Sheared prismatization, 2025.
\newblock Work in progress.

\bibitem{bhattlurie}
Bhargav Bhatt and Jacob Lurie.
\newblock The prismatization of $p$-adic formal schemes, 2022.
\newblock Version v1.

\bibitem{bmvz}
Bhargav Bhatt, Akhil Mathew, Vadim Vologodsky, and Mingjia Zhang.
\newblock Sheared witt vectors, 2025.
\newblock Work in progress.

\bibitem{BMS}
Bhargav Bhatt, Matthew Morrow, and Peter Scholze.
\newblock Topological hochschild homology and integral $p$-adic hodge theory.
\newblock {\em Publications Mathématiques de l’IHÉS}, 129:199--310, 2019.

\bibitem{borne2006}
Niels Borne.
\newblock Fibrés paraboliques et champ des racines.
\newblock {\em arXiv:math/0604458}, 2006.

\bibitem{borne2007}
Niels Borne.
\newblock Sur les représentations du groupe fondamental d'une variété privée d'un diviseur à croisements normaux simples.
\newblock arXiv:0704.1236v2 [math], 2008.
\newblock Submitted 10 Apr 2007; revised 15 Feb 2008 (v2).

\bibitem{siqingfix}
M.~A. de~Cataldo and A.~Fernandez Herrero.
\newblock Factorization of the \emph{p}-curvature morphism (appendix to ``geometry of the logarithmic hodge moduli space'').
\newblock {\em Journal of the London Mathematical Society}, 2024.
\newblock Appendix to ``Geometry of the logarithmic Hodge moduli space''.

\bibitem{dCZ}
Mark~Andrea de~Cataldo and Siqing Zhang.
\newblock Logarithmic non-abelian hodge theory for curves in prime characteristic, 2025.
\newblock Version v1.

\bibitem{dCGZ}
Mark Andrea~A. de~Cataldo, Michael Groechenig, and Siqing Zhang.
\newblock The de rham stack and the variety of very good splittings of a curve.
\newblock {\em arXiv preprint}, 2022.
\newblock To appear in *Mathematical Research Letters* (MRL), Vol.\,31, No.\,5 (2024), pp. 1353–1389.

\bibitem{prismatization}
Vladimir Drinfeld.
\newblock Prismatization, 2020.
\newblock Version v7, revised 30 April 2024.

\bibitem{drinfeld}
Vladimir Drinfeld.
\newblock On a notion of ring groupoid, 2021.
\newblock Version v1.

\bibitem{shearedwitt}
Vladimir Drinfeld.
\newblock Ring stacks conjecturally related to the stacks $bt_n^{G,\mu}$, 2025.

\bibitem{eg}
H{\'e}l{\`e}ne Esnault and Michael Groechenig.
\newblock Cristallinity of rigid flat connections revisited.
\newblock {\em arXiv preprint}, arXiv:2309.15949, 2023.

\bibitem{groechenig}
Michael Groechenig.
\newblock Moduli of flat connections in positive characteristic.
\newblock {\em Math. Res. Lett.}, 23(4):989--1047, 2016.

\bibitem{azumayaparabolic}
Márton Hablicsek.
\newblock Hodge theorem for the logarithmic de rham complex via derived intersections.
\newblock {\em Research in the Mathematical Sciences}, 7(1):24, 2020.
\newblock Published 03 August 2020, article 24.

\bibitem{katz}
Nicholas~M. Katz.
\newblock Nilpotent connections and the monodromy theorem: applications of a result of turrittin.
\newblock {\em Publications Mathématiques de l’Institut des Hautes Études Scientifiques}, 39:175--232, 1970.

\bibitem{laaroussi}
Abdessamad Laaroussi.
\newblock {\em Connexions paraboliques et champ des racines}.
\newblock PhD thesis, Universit{\'e} de Lille, 2022.
\newblock Thesis available online.

\bibitem{laszlopauly}
Yves Laszlo and Christian Pauly.
\newblock On the hitchin morphism in positive characteristic.
\newblock {\em International Mathematics Research Notices}, 2001(3):129--143, 2001.

\bibitem{lisun}
Mao Li and Hao Sun.
\newblock Tame parahoric nonabelian hodge correspondence in positive characteristic over algebraic curves, 2024.

\bibitem{lieblich}
Max Lieblich.
\newblock Remarks on the stack of coherent algebras.
\newblock {\em International Mathematics Research Notices}, 2006.

\bibitem{lopez}
Rose Lopez.
\newblock Picard groups of stacky curves.
\newblock {\em arXiv preprint}, arXiv:2306.08227, 2023.
\newblock Submitted 14 June 2023.

\bibitem{lorenzon}
Pierre Lorenzon.
\newblock Indexed algebras associated to a log structure and a theorem of p-descent on log schemes.
\newblock {\em manuscripta mathematica}, 101(3):271--299, 2000.

\bibitem{DAG}
Jacob Lurie.
\newblock Derived algebraic geometry.
\newblock \url{https://people.math.harvard.edu/~lurie/papers/DAG.pdf}, 2004--2011.

\bibitem{HTT}
Jacob Lurie.
\newblock {\em Higher Topos Theory}, volume 170 of {\em Annals of Mathematics Studies}.
\newblock Princeton University Press, Princeton, NJ, 2009.
\newblock Available at \url{https://www.math.ias.edu/~lurie/papers/HTT.pdf}.

\bibitem{SAG}
Jacob Lurie.
\newblock Spectral algebraic geometry.
\newblock \url{https://www.math.ias.edu/~lurie/papers/SAG-rootfile.pdf}, 2018.

\bibitem{ogus}
Arthur Ogus.
\newblock Higgs cohomology, $p$-curvature, and the cartier isomorphism.
\newblock {\em Compositio Mathematica}, 121(3):265--294, 2000.

\bibitem{ov}
Arthur Ogus and Vadim Vologodsky.
\newblock Nonabelian hodge theory in characteristic p.
\newblock {\em Publications Mathématiques de l'IHÉS}, 106:1--138, 2007.

\bibitem{olsson3}
Martin Olsson.
\newblock Logarithmic geometry and algebraic stacks.
\newblock {\em Annales Scientifiques de l'{\'E}cole Normale Sup{\'e}rieure}, 36(5):747--791, 2003.

\bibitem{olsson}
Martin Olsson.
\newblock Crystalline cohomology of algebraic stacks and hyodo--kato cohomology.
\newblock {\em Astérisque}, 316, 2007.

\bibitem{schaub}
Daniel Schaub.
\newblock Courbes spectrales et compactifications de jacobiennes.
\newblock {\em Mathematische Zeitschrift}, 227(2):295--312, 1998.

\bibitem{schepler}
Daniel~Kenneth Schepler.
\newblock {\em Logarithmic nonabelian Hodge theory in characteristic p}.
\newblock PhD thesis, University of Califronia, Berekely, 2005.

\bibitem{simpson}
Carlos Simpson.
\newblock Homotopy over the complex numbers and generalized de rham cohomology.
\newblock In M.~Maruyama, editor, {\em Moduli of Vector Bundles (Taniguchi Symposium December 1994)}, volume 189 of {\em Lecture Notes in Pure and Applied Mathematics}, pages 229--263. Dekker, New York, 1996.
\newblock Toulouse preprint no. 50, April 1995.

\end{thebibliography}

\end{document}